\documentclass[11pt]{article} 

\usepackage[utf8]{inputenc} 
\usepackage{latexsym,amsmath,amssymb,textcomp,amsthm}

\usepackage{geometry} 
\usepackage{tikz}
\usepackage{graphicx} 

\usepackage{calrsfs}
\usepackage[english]{babel}
\usepackage{booktabs} 
\usepackage{array} 
\usepackage{paralist} 
\usepackage{verbatim} 
\usepackage{subfig} 
\usepackage{pgf}
\usepackage{fancyhdr} 
\usepackage{hyperref}
\usepackage{sectsty}
\allsectionsfont{\sffamily\mdseries\upshape} 
\usepackage{amsmath,amsfonts,amssymb}

\usepackage[nottoc,notlof,notlot]{tocbibind} 
\usepackage[titles,subfigure]{tocloft} 

\usetikzlibrary{arrows,automata}

\newtheorem{prop}{Proposition}[subsection]
\newtheorem{lem}[prop]{Lemma}
\newtheorem{theo}{Theorem}
\newtheorem{theo2}{Theorem}
\newtheorem{defin}{Definition}
\newtheorem{rmq}{Remark}
\newtheorem{theoA}{Theorem}

\newcommand{\cg}{[\kern-0.15em [}
\newcommand{\cd}{]\kern-0.15em]}

\newcommand{\E}{\mathbb{E}}
\newcommand{\Prob}{\mathbb{P}}
\newcommand{\prob}{\text{P}}

\newcommand{\Qprob}{\mathbb{Q}}
\newcommand{\R}{\mathbf{R}}

\newcommand{\N}{\mathbf{N}}
\newcommand{\Z}{\mathbf{Z}}

\newcommand{\dd}{\mathrm{d}}

\title{Limit theorem for sub-ballistic Random Walks in Dirichlet Environment in dimension $d\geq 3$ }
\author{R\'emy Poudevigne--Auboiron}
\date{} 

\begin{document}
\maketitle
\section{Introduction and results}
\subsection{Introduction}

Random walks in random environments (RWRE) have been studied for several decades and are now rather well understood in the one dimensional case (see Solomon $\cite{Solomon}$,Kesten, Kozlov, Spitzer $\cite{Kesten}$ and Sina\"{\i} $\cite{Sinai}$). Important progress has been made in higher dimension, mainly in 3 directions: under a ballisticity condition, for small perturbation of the simple random walk ($\cite{BricKupia},\cite{SznitZeitCLTBrownian},\cite{ExitBall0},\cite{LowDis1},\cite{LowDis2}$) and in Dirichlet environment. \\
The most studied ballisticity conditions come from the conditions $(T)$ and $(T^{\prime})$ introduced by Sznitman in $\cite{SznitmanT},\cite{SznitmanEffective}$. They have been shown to be equivalent in $\cite{RamiEquiv}$ and also to be equivalent to an effective polynomial condition $\cite{BergEffectPol},\cite{ElliptCriteria}$. By assuming any of these, in the ballistic regime, directional transience, ballisticity, and a CLT have been proved. Quenched CLTs have also been proved in various cases, either by assuming an annealed CLT, uniform ellipticity and a condition introduced by Kalikow $\cite{SznitmanCLT}$, or by assuming the existence of high enough moments for the renewal times (see $\cite{SznitmanZerner}$ for a definition of the renewal times) and uniform ellipticity of the environment $\cite{SeppaTCL}$ and $\cite{BergerZeitouniTCL}$ in dimension $d\geq 4$. \\
All these results show limit theorems in the ballistic case, that is to say that the walk has a positive speed. In dimension 2 and higher no complete limit theorems are known for the RWRE in the sub-ballistic case. However in dimension 1 we know that a sub-ballistic regime exists, where the walk can behave like the inverse of a stable subordinator $\cite{Kesten}$ $\cite{LimLawsZ}$. This sub-ballistic regime is caused by the existence of traps where the walk spends most of its time. This trapping phenomenon appears in other models closely related to the RWRE for instance the Bouchaud Trap Model (see $\cite{BouchaudBenArous}$ for a precise definition and an overview of the results). The model of random walks in random conductances also exhibits a similar trapping phenomenon. Indeed an annealed limit theorem (the limit is the inverse of a stable subordinator) and an equivalent to the CLT $\cite{FriberghSubbal}$ have been proved for the biased random walk in random conductances. Similar results have been obtained for the biased walk in the percolation cluster and in Galton-Watson trees, but in both cases there is no convergence to a limit law $\cite{FriberghSubbalTree}, \cite{FriberghSubbalPerc}$. In the special case of iid RWRE a trapping phenomenon that leads to sub-ballistic behaviour has been identified in $\cite{BouchetSubbal},\cite{ballistic1}\text{ and }\cite{friberghtrap}$ but no limit theorem has been proved.\\
The random walk in Dirichlet environment (RWDE) is a model where the transition probabilities are iid Dirichlet random variables (see $\cite{DirOverview}$ for an overview). It was first introduced because of its link to the linearly directed-edge reinforced random walk ($\cite{Pemantle}$,$\cite{RWDE}$). It also has a property of invariance by time reversal that allows explicit calculations (see $\cite{transdim}$). In particular, it gives a simple criterion for existence of absolutely continuous invariant distribution from the point of view of the particle, directional transience and ballisticity in dimension $d\geq 3$ $(\cite{IntegExitTime},\cite{BouchetSubbal},\cite{transdir2},\cite{DirEnvirPOVPart})$. In the non-ballistic case the walk is directionally transient but the limit law was still unknown ($\cite{BouchetSubbal}$), it was only known that for some explicit $\kappa\in (0,1]$, $\frac{\log(|X_n|)}{\log(n)}\rightarrow \kappa$.   \\
In this paper we give the annealed limit law for the sub-ballistic regime ($\kappa\leq 1$) in dimension $d\geq 3$. In the case $\kappa=1$ we have the limit law of $\frac{1}{n\log(n)}Y_n$ (where $Y$ is the random walk) while for $\kappa<1$ we have the limit law of the process. To the best of our knowledge, this is the first stable limit theorem for non reversible RWRE in iid environment, in dimension $d\geq 2$.
.

\subsection{Definitions and statement of the results}

In all the paper we set $d\geq 3$. Let $(e_1,\dots,e_d)$ be the canonical base of $\Z^d$ and for any $j\in  [\![d+1,2d]\!]$, set $e_j=-e_{j-d}$. For any $z\in\Z^d $, let $||z||:=\sum\limits_{i=1}^d |z_i|$ be the $L_1$-norm of $z$. For any $x,y \in \Z^d$ we will write $x\sim y$ if $||y-x||=1$. Let $E=\{(x,y)\in(\Z^d)^2,x\sim y\}$ be the set of directed edges of $Z^d$ and let $\tilde{E}=\{\{x,y\},(x,y)\in(\Z^d)^2,x\sim y\}$ be the set of non-directed edges. Let $\Omega$ be the set of environments on $\Z^d$:
\[
\Omega=\{\omega = (\omega(x,y))_{x\sim y} \in (0,1]^E\text{ such that }\forall x\in\Z^d,\sum\limits_{i=1}^{2d}\omega(x,x+e_i)=1\}.
\]
For each $\omega\in\Omega$, let $(Y_n)_{n\in\N}$ be the Markov chain on $\Z^d$ defined by $Y_0=0$ almost surely and the following transition probabilities:
\[
\forall y\in\Z^d,\forall i \in [\![1,2d]\!],\ \prob_0^{\omega}(Y_{n+1}=y+e_i|Y_n=y)=\omega(y,y+e_1).
\]
Let $\text{E}_{\prob_0^{\omega}}$ be the expectation with respect to $\prob_0^{\omega}$.\\
Given a family of positive weights $(\alpha_1,\dots,\alpha_{2d})$, we consider the case where the transition probabilities at each site are iid Dirichlet random variables of parameter $\alpha:=(\alpha_1,\dots,\alpha_{2d})$,  that is with density:
\[
\frac{\Gamma\left(\sum\limits_{i=1}^{2d}\alpha_i\right)}{\prod\limits_{i=1}^{2d}\Gamma(\alpha_i)}\left(\prod\limits_{i=1}^{2d}x_i^{\alpha_i-1}\right)\dd x_1\dots \dd x_{2d-1}
\]
on the simplex
\[
\{(x_1,\dots,x_{2d})\in (0,1]^{2d},\sum\limits_{i=1}^{2d}x_i=1\}.
\]
Let $\Prob^{(\alpha)}$ be the law obtained on $\Omega$ this way. Let $\E_{\Prob^{(\alpha)}}$ be the expectation with respect to $\Prob^{(\alpha)}$ and let $\Prob^{(\alpha)}_0[.]:=\E_{\Prob^{(\alpha)}}[\prob^{\omega}_0(.)]$ be the annealed law of the process starting at $0$.
Let $\left(\tau_i\right)_{i\in\N^*}$ be the renewal times, in the direction $e_1$, introduced in $\cite{SznitmanZerner}$:
\begin{defin}
We define $(\tau_i)_{i\in\N^*}$, the renewal times in the direction $e_1$, by:
\[
\tau_1 = \inf \{ n \in \N, \forall i<n, Y_i.e_1<Y_n.e_1 \text{ and } \forall i>n, Y_i.e_1 > Y_n.e_1  \}
\]
and for all $i>1$:
\[
\tau_{i+1} = \inf \{ n > \tau_{i}, \forall i<n, Y_i.e_1<Y_n.e_1 \text{ and } \forall i>n, Y_i.e_1 > Y_n.e_1  \}.
\]
\end{defin}
The renewal times are used to create independence thanks to the following theorem (Theorem 1.4 of $\cite{SznitmanZerner}$).
\begin{prop}\label{prop:1}
For all $k\in\N^*$, let $\mathcal{G}_k$ be the $\sigma$-field defined by:
\[
\mathcal{G}_k:=\sigma(\tau_1,\dots,\tau_k,(Y_n)_{0\leq n \leq \tau_k},(\omega(x,\cdot))_{x.e_1 < Y_{\tau_k}.e_1}).
\] 
We have, for all $k\geq 1$:
\[
\Prob^{(\alpha)}_0\left((Y_{\tau_k+n})_{n\geq 0} \in \cdot, (\omega(Y_{\tau_k}+x,\cdot))_{x.e_1\geq 0}\in \cdot |\mathcal{G}_k\right)\\
=\Prob^{(\alpha)}_0\left((Y_n)_{n\geq 0} \in \cdot, (\omega(x,\cdot))_{x.e_1\geq 0}\in \cdot |\tau_1=0\right).
\]
\end{prop}
This means that the trajectories and the transition probabilities inside slabs between two consecutive renewal times (after the first one) are i.i.d random variables.  
\begin{defin}
We define the drift $d_{\alpha}$ by:
\[
d_{\alpha}:= \sum \alpha_i e_i.
\]
If $d_{\alpha}\not = 0$, we will assume, without loss of generality, that $\alpha_1>\alpha_{1+d}$.
\end{defin}
\begin{defin}
We define the two parameters $\kappa$ and $\kappa^{\prime}$ by:
\[
\kappa = 2 \left(\sum\limits_{i=1}^{2d}\alpha_i\right)- \max\limits_{i=1,\dots,d}(\alpha_i+\alpha_{i+d})
\]
and
\[
\kappa^{\prime}= 3\left(\sum\limits_{i=1}^{2d}\alpha_i\right)-2\max\limits_{i=1,...,d}\left(\alpha_i+\alpha_{i+d}\right).
\]
For any direction $j\in[\![1,d]\!]$ we also define the parameter $\kappa_j$ by:
\[
\kappa_j:=2 \left(\sum\limits_{i=1}^{2d}\alpha_i\right)- (\alpha_j+\alpha_{j+d})
\]
\end{defin}
In $\cite{DirEnvirPOVPart}$, it was proved that, for $d\geq 3$, when $\kappa>1$, there exists an invariant probability measure $\Qprob^{(\alpha)}$ for the environment from the point of view of the particle, absolutely continuous with respect to $\Prob^{(\alpha)}$. From that it is possible to show that directional transience and ballisticity are equivalent when $\kappa>1$. Furthermore, we know for which parameter the walk is directionally transient.
\begin{theoA}[Corollary 1 of \cite{transdir2}] \label{theodule}
If $d\geq 3$ and $d_{\alpha}\not =0$, then for $\Prob^{(\alpha)}$ almost every environment, the walk is directionally transient with asymptotic direction $d_{\alpha}$, that is to say:
\[
\frac{Y_n}{||Y_n||}\rightarrow \frac{d_{\alpha}}{||d_{\alpha}||},\  \prob_0^{\omega} \text{ almost surely.}
\] 
\end{theoA} 
However, when $\kappa\leq 1$, such an invariant probability does not exist because of traps. But, in $\cite{BouchetSubbal}$, it was proved that, by accelerating the walk, we can get an invariant probability for this accelerated walk, absolutely continuous with respect to $\Prob^{(\alpha)}$.\\
This lead to the following limit theorem in $\cite{BouchetSubbal}$:
\begin{prop}
If $\kappa\leq 1$, $d\geq 3$ and $d_{\alpha}\not = 0$. Let $l\in \{e_1,\dots,e_{2d}\}$ be such that $d_{\alpha}.l>0$. then we have the following convergence in probability (for the annealed law): 
\[
\frac{\log(Y_n.l)}{\log(n)}\rightarrow \kappa.
\] 
\end{prop}
We will now give a precise definition of the accelerated walk. We call directed path a sequence of vertices $\sigma = (x_0,\dots,x_n)$ such that $(x_i,x_{i+1})\in E$ for all $i$. To simplify notations, we will write $\omega_{\sigma}:=\prod\limits_{i=0}^{n-1} \omega(x_i,x_{i+1})$. For any positive integer $m$, we define the accelerating function $\gamma^m_{\omega}(x)$ by:
\[
 \gamma^m_{\omega}(x):=\frac{1}{\sum \omega_{\sigma}},
\] 
where the sum is on all finite simple (each vertex is visited at most once) paths $\sigma$ in $x+[\![-m,m ]\!]^d$, starting from $x$, going to the border of $x+[\![-m,m ]\!]^d$ and stopped the first time they reach this border. We will call $X^m_t$ the continuous-time Markov chain whose jump rate from $x$ to $y$ is $\gamma^m_{\omega}(x)\omega(x,y)$, with $X^m_0=0$. This means that $Y_n=X^m_{t^m_n}$ and $X^m_t=\sum\limits_{k}Y_k 1_{t^m_k\leq t < t^m_{k+1}}$, for $t^m_n=\sum\limits_{k=1}^{n}\frac{1}{\gamma^m_{\omega}(Y_k)}\mathcal{E}_k$, where the $\mathcal{E}_i$ are iid exponential random variables of parameter 1. The walk $X^m_t$ can be viewed as an accelerated version of the walk $Y_n$.\\
Now, we need to introduce an other object: the walk seen from the point of view of the particle. First, let $(\theta_x)_{x\in\Z^d}$ be the shift on the environment defined by: $\theta_x \omega(y,z):=\omega(x+y,x+z)$. We call process seen from the point of view of the particle the process defined by $\overline{\omega^m_t}=\theta_{X^m_t}\omega$. Unlike the walk $Y$, under $\Prob^{(\alpha)}_0$, $\overline{\omega^m_t}$ is a Markov process on $\Omega$. Its generator $R$ is given by:
\[
Rf(\omega)=\sum\limits_{i=1}^{2d}\gamma^m_{\omega}(0)\omega(0,e_i)f(\theta_{e_i}\omega),
\]
for all bounded measurable functions $f$ on $\Omega$.
\begin{theoA}(Theorem 2.1 of \cite{BouchetSubbal})\\
In dimension $d\geq 3$, if $m$ is large enough then the process $\left(\overline{\omega^m_t}\right)_{t\in\R^+}$ has a stationary distribution $\Qprob^{m,\alpha}$. For any $\beta>1$ there exists an $m$ such that $\frac{\dd \Qprob^{m,\alpha}}{\dd \Prob^{\alpha}}$ is in $L^{\beta}$.
\end{theoA}
We will write $\Qprob^{m,\alpha}_0(\cdot)$ for $\Qprob^{m,\alpha}\left(\prob^{\omega}_0(\cdot)\right)$
To simplify the notations, we will drop the $(\alpha)$ from $\Prob^{(\alpha)},\Prob^{(\alpha)}_0,\Qprob^{m,\alpha}$ and $\Qprob^{m,\alpha}_0$ when there is no ambiguity. We will also write $X_t$, $\Qprob$ and $\Qprob_0$ instead of $X^m_t$, $\Qprob^m$ and $\Qprob^m_0$ when there is no ambiguity on $m$. \\
We need a last definition to be able to state the limit theorems.
\begin{defin}
For any $\kappa\in (0,1)$ let $\mathcal{S}^{\kappa}$ be the L\'evy process where the increments are completely asymmetric $\kappa$-stable random variables. The increment have the following characterizations:
\[
\forall \lambda\in\R,\forall s\in \R^+, \E\left(\exp\left(i\lambda \mathcal{S}^{\kappa}_s\right)\right)=\exp\left(-s|\lambda |^{\kappa}\left(1-i\text{sgn}(\lambda)\tan\left(\frac{\pi\kappa}{2}\right)\right)\right)
\]
and for any $s\in \R^+$, $\mathcal{S}^{\kappa}_s$ and $s^{\frac{1}{\kappa}}\mathcal{S}^{\kappa}_1$ have the same law. \\
Since this process is non-decreasing and c\`adl\`ag we can define the c\`adl\`ag inverse $\tilde{\mathcal{S}}^{\kappa}$ by:
\[
\tilde{\mathcal{S}}^{\kappa}_t := \inf\{s, \mathcal{S}^{\kappa}_s\geq t\}.
\]
\end{defin}
The following two theorems, which are the main results of this paper, give a full annealed limit theorem: \\
\begin{theo2}
Set $d\geq 3$ and $\alpha\in(0,\infty)^{2d}$. Let $Y^n(t)$ be defined by:
\[
Y^n(t)=n^{-\kappa}Y_{\lfloor nt \rfloor}.
\]
If $\kappa<1$ and $d_{\alpha}\not=0$, there exists positive constants $c_1,c_2,c_3$ such that for the $J_1$ topology and for $\Prob^{(\alpha)}_0$:
\[
\left(t\rightarrow n^{-\frac{1}{\kappa}}\tau_{\lfloor nt \rfloor}\right)\rightarrow c_1 \mathcal{S}^{\kappa},
\]
for the $M_1$ topology and for $\Prob^{(\alpha)}_0$:
\[
\left(t\rightarrow n^{-\frac{1}{\kappa}}\inf\{t\geq 0, Y(t).e_1\geq nx\}\right) \rightarrow c_2 \mathcal{S}^{\kappa}
\]
and for the $J_1$ topology and for $\Prob^{(\alpha)}_0$:
\[
Y^n\rightarrow c_3 \tilde{\mathcal{S}}^{\kappa}d_{\alpha}.
\]
\end{theo2}
\begin{rmq}
We will give a quick explanation on what the $M_1$ and $J_1$ topologies are, for a precise definition see \cite{Skorohod},\cite{Whitt}. They were both introduced as a generalization of the infinite norm for c\`adl\`ag functions. In the $J_1$ topology, a sequence of c\`adl\`ag functions $f_n$ converges to $f$ if there exists a sequence of increasing homomorphisms $\lambda_n:[0,1]\mapsto [0,1]$ such that 
\[
\sup\limits_{t\in[0,1]} |\lambda_n(t)-t|\rightarrow 0,
\]
and
\[
\sup\limits_{t\in[0,1]} |f_n(\lambda_n(t))-f(t)|\rightarrow 0.
\]
It is essentially the same as the infinite norm except that you can "wiggle" the function time-wise. The $M_1$ topology is a topology on the graphs of the functions where we add vertical segments every time there is a jump. The main difference between the $M_1$ and $J_1$ topology is that there is almost no difference between one jump and small consecutive jumps in the $M_1$ topology while the difference is significant in the $J_1$ topology. The reason why we only have a convergence in $M_1$ for the hitting times $n^{-\frac{1}{\kappa}}\inf\{t\geq 0, Y(t).e_1\geq nx\}$ is because there are consecutive jumps. Indeed, if there is a large jump between $\inf\{t\geq 0, Y(t).e_1\geq n\}$ and $\inf\{t\geq 0, Y(t).e_1\geq n+1\}$ it is likely that there is a trap with high strength close-by which means that it is likely that there also is a large jump between $\inf\{t\geq 0, Y(t).e_1\geq n+1\}$ and $\inf\{t\geq 0, Y(t).e_1\geq n+2\}$.
\end{rmq}
\begin{theo2}
If $d\geq 3$ and $\kappa=1$, there exists positive constants $c_1,c_2,c_3$ such that we have the following convergences in probability (for the annealed law):
\[
\frac{1}{n\log(n)}\tau_n \rightarrow c_1,
\]
\[
\frac{1}{n\log(n)}\inf\{i,Y_i.e_1\geq n\}\rightarrow c_2,
\]
\[
\frac{\log(n)}{n}(Y_n)\rightarrow c_3 d_{\alpha}.
\]
\end{theo2}
\begin{rmq}
We cannot replace the convergence in probability by an almost sure convergence. This is because if we look at a sum of iid random variables $Z_i$ with a heavy tail $\Prob(Z_i\geq t)\sim ct^{-1}$ then we do not have an almost sure convergence. In fact, there are infinitely many $i$ such that:
\[
Z_i \geq i\log(i)\log(\log(i)).
\]
\end{rmq}

A tool that will be central in the proof is the study of traps. We now give a precise definition of traps.
\begin{defin}\label{defTrap}
A trap is any undirected edge $\{x,y\}$ such that $\omega(x,y)+\omega(y,x)> \frac{3}{2}$. \\
The strength of a trap is the quantity $\frac{1}{(1-\omega(x,y))+(1-\omega(y,x))}$.
\end{defin}
\begin{rmq}
$\frac{3}{2}$ has been chosen because it ensures that $\omega(x,y),\omega(y,x)>\frac{1}{2}$ which in turn means that for every point $x$, there is at most one point $y$ such that $(x,y)$ is a trap. 
\end{rmq}

\subsection{Sketch of the proof}

The proofs for $\kappa<1$ and $\kappa=1$ are mostly the same and therefore we will explain both at the same time.
\subsubsection{Only the renewal times matter}
We first show that the number of points visited between two renewal times has a finite expectation (lemma \ref{lem:9} ). This means that the walk does not "wander far" between two renewal times. So we only have to know the renewal times and the position of the walk at the renewal times to prove both theorems (lemma \ref{lem:10} ). By proposition \ref{prop:1}, the random variables $(\tau_{i+1}-\tau_i)$ are iid which simplifies the study of the process $i\rightarrow \tau_i $
\subsubsection{The time between renewal times only depends on the strength of the traps}
Then we use the stationary law of the accelerated walk to get two results: firstly, the time spent outside of traps is negligible (lemma \ref{lem:19} ); secondly, the number of time $N$ the walk enters a trap has a finite moment of order $\kappa +\varepsilon$ for some $\varepsilon>0$ if $\kappa<1$. If $\kappa=1$, then $N$ has a finite expectation (lemma \ref{lem:21} ). This means the time spent in a trap mostly depends on its strength.\\
Now we want to show that the number of times the walk enters a trap and the time it stays in the trap each time are approximately independent.\\
We get two different results in this direction:
\subsubsection{The strength of the traps are essentially independent}
The first result (lemma \ref{lem:11}) is that in a way the time spent in traps are independent random variables. These random variables have a tail in $Ct^{-\kappa}$ where the constant $C$ depends on where the walk enters and exits the trap and how many times it does. More precisely, we first set an environment and a path in this environment. Then we forget all the transition probabilities in the traps, this means that if $\{x,y\}$ is a trap, then we only remember the "renormalized" transition probabilities:
\[
\left(\frac{\omega(x,z)}{1-\omega(x,y)}\right)_{z\sim x,z\not = y} \text{ and } \left(\frac{\omega(y,z)}{1-\omega(y,x)}\right)_{z\sim y,z\not = x}.
\]
Then every time the path visits a trap we only remember where it enters the trap and where it exits the trap, we forget the number of back and forths inside the trap. Then, only knowing these information, the strength of the traps are independent.
\subsubsection{The number of times a trap is visited and its strength are essentially independent}
The second result (lemma \ref{lem:13} ) allows us to bound the probability that both the number of times the walk enters a trap and the strength of the trap are high. We use the fact that for an edge $(x,y)$ if we know all the transition probabilities outside of $x,y$ and we know the $\left(\frac{\omega(x,z)}{1-\omega(x,y)}\right)_{z \sim x}$ and the $\left(\frac{\omega(y,z)}{1-\omega(y,x)}\right)_{z\sim y}$ then the number of times the walk enters the trap is essentially independent of the strength of the trap (it depends mostly on $\frac{1-\omega(x,y)}{1-\omega(y,x)}$ and hardly on the strength of the trap). This means that it is unlikely that the traps with a high strength are visited many times.
\subsubsection{Conclusion}
Thanks to these results we get that if we set an integer $A$ and we only look at traps that are entered less than $A$ times then we have a good approximation of the total time spent in traps (lemma \ref{lem:16} ). The higher $A$ is, the better the approximation gets. Now if we only look at the traps the walk enters less than $A$ times, we get a finite sum of sums of iid random variables by lemma \ref{lem:11}. This means that, after renormalization, the time spent in traps entered less than $A$ times converges to a stable distribution if $\kappa<1$. It converges to a constant if $\kappa = 1$ (lemma \ref{lem:18} ). Then the only thing left is to make $A$ go to infinity and we get the first two results of both theorems.\\
Finally to prove the last part of both theorems we just use basic inversion arguments.

\section{The proof}
\subsection{Number of points visited between renewal times}

In this section we show that the expectation of the number of point visited between two renewal times is finite. This means that only knowing the values of the renewal times will be enough to prove theorem 1 and 2.
\begin{lem}\label{lem:8}
For $m$ such that $\Qprob^m$ exists, let $(T_i^m)_{i\in\N^*}$ be the renewal times for the walk $X^m$ i.e $T_i^m := t^m_{\tau_i}$ or to put it another way $X^m_{T_i^m}=Y_{\tau_i}$. There exists a constant $C_m$ such that for all $i\in\N^*$, $\E_{\Prob^{(\alpha)}_0}(T^m_{i+1}-T^m_i)=C_m$ and $\Prob^{(\alpha)}_0$ almost surely:
\[
\frac{1}{n}T^m_n\rightarrow C_m.
\]
\end{lem}
\begin{proof}
Let $D$ be the random distance defined by $D=Y_{\tau_2}-Y_{\tau_1}$. First we will show that $\E_{\Prob_0}(D)<\infty$. \\
Let $(\tau_i)_{i\in \N^*}$ be the different renewal times along the direction $e_1$. Now let $(d_i)_{i\in\N^*}$ be the sequence defined by:
\[
\forall i\in\N^*, d_i=Y_{\tau_i}.e_1.
\]
Let $\tilde{L}^{\tau}(i)$ be the number of renewal times before the walks travels a distance $i$ in the direction $e_1$ ie:
\[
\forall i\in\N^*, \tilde{L}^{\tau}(i)=\inf\{n,d_n\geq i\}.
\]
The sequence of random variables $(d_{i+1}-d_{i})_{i\in\N^*}$ is iid by lemma \ref{prop:1}. Therefore, if the expectation of $D=d_2-d_1$ is infinite then $\frac{d_n}{n}\rightarrow\infty$, $\Prob_0$ almost surely. Now, for every $i\in\N^*$, we have $d_{\tilde{L}^{\tau}(i)} \geq i$ and therefore $\frac{\tilde{L}^{\tau}(i)}{i} \leq \frac{\tilde{L}^{\tau}(i)}{d_{\tilde{L}^{\tau}(i)}}$. If $\Prob_0$ almost surely $\frac{n}{d_n}\rightarrow 0$ we would have $\frac{\tilde{L}^{\tau}(i)}{i}\rightarrow 0$ $\Prob_0$ almost surely. Since $\frac{\tilde{L}^{\tau}(i)}{i+1}\leq 1$ we would get that $\E_{\Prob_0}\left(\frac{\tilde{L}^{\tau}(i)}{i}\right)\rightarrow 0$. However, there is a constant $C>0$ such that every time the walk reaches a new height along $e_1$, it is a renewal time with probability $C$ (independent of the walk up to that time) so $\E_{\Prob_0}\left(\frac{\tilde{L}^{\tau}(i)}{i}\right)\geq C$. Therefore we get that the expectation of the distance the walk travels in the direction $e_1$ between two renewal times is finite.\\
\\
Now we can look at the accelerated walk $X^m$. We would like the sequence $(T^m_{i+1}-T^m_i)_{i\in\N^{*}}$ to be a sequence of iid random variables. Unfortunately, the definition of the accelerated random walk uses vertices in a box of size $m$ around the vertex on which the walk currently is, so we need to wait at least $2m+3$ renewal times to be sure to be at a distance at least $2m+1$ of all the vertices visited before time $T^m_{i+1}-1$. So we only have that for any $j\in\N$, the sequence $\left(T^m_{(2m+3)i+j+1}-T^m_{(2m+3)i+j}\right)_{i\in\N^*}$ is a sequence of iid random variables. Furthermore the sequence $(T^m_{i+1}-T^m_i)_{i\geq m+2}$ is identically distributed.\\
\\
We know that there exists a constant $c>0$ such that $\Prob_0$ almost surely $\frac{X^m_t.e_1}{t}\rightarrow c>0$. If the expectation of the time the accelerated walk spends between two renewal times is infinite then $\frac{T^m_i}{i}\rightarrow \infty$, $\Prob_0$ almost surely since the random variables $\left(T^m_{(2m+3)i+1}-T^m_{(2m+3)i}\right)_{i\in\N^*}$ are iid. Therefore we would have $\frac{X^m_{T^m_i}.e_1}{T^m_i}\frac{T^m_i}{i}\rightarrow \infty$ so $\frac{Y_{\tau_i}.e_1}{i}\rightarrow\infty$ which is absurd because: $\frac{Y_{\tau_i}.e_1}{i}=\frac{d_i}{i}$ and $\frac{d_i}{i}$ satisfies a law of large number. Therefore the expectation of time the accelerated walk spends between two renewal times is finite and there exists a constant $C>0$ such that:
\[
\forall i\geq m+2,\ \E_{\Prob_0}(T^m_{i+1}-T^m_i)=C.
\]
And by the law of large number, $\Prob_0$ almost surely:
\[
\frac{1}{i}T^m_i\rightarrow C.
\]
\end{proof}

\begin{lem}\label{lem:9}
The number of different points the walk visits between two renewal times has a finite expectation (Note that the number of different points visited between two renewal times is the same for the walk $Y$ and the accelerated walks $X^m$). 
\end{lem}
\begin{proof}
We choose $m$ large enough such that $\frac{\dd\Qprob^m}{\dd\Prob}$ is in $L^{\gamma}$ for some $\gamma >1$. In the following we will write $T_i$ instead of $T_i^m$ to simplify the notations. Let $\beta$ be such that $\frac{1}{\gamma}+\frac{1}{\beta}=1$. Let $c_{\infty}$ be the constant such that $\Prob_0$ almost surely: $\frac{1}{i}T_i\rightarrow c_{\infty}$, it exists by lemma \ref{lem:8}. Let $(R_i)_{i\in\N^*}$ be the sequence defined by: $\forall i\in\N^*,R_i= \#\{x,\exists j\leq \tau_i, Y_j=x\}$. The random variables $(R_{i+1}-R_i)_{i\geq 1}$ are iid by proposition \ref{prop:1}. Thus if the number of different points the walk visits between two renewal times has an infinite expectation (for $\Prob_0$) then $\frac{R_i}{i}\rightarrow \infty $, $\Prob_0$ almost surely and therefore $\Qprob^m_0$ almost surely. However we have for any $C>0$: 
\[
\begin{aligned}
\Qprob^m_0(R_n\geq Cn)&\leq\Qprob^m_0(T_n\geq 2c_{\infty}n)+\Qprob(R_n\geq Cn \text{ and } T_n<2c_{\infty}n)\\
&=o(1)+\Qprob^m_0(R_n\geq Cn \text{ and } T_n<2c_{\infty}n)\\
&\leq o(1) + \Qprob^m_0\left(\sum\limits_{0\leq i\leq 2c_{\infty}n}\#\{x,\exists t\in[i,i+1),X_t=x\}\geq Cn\right)\\
&\leq o(1) + \frac{1}{Cn}\E_{\Qprob^m_0}\left(\sum\limits_{0\leq i\leq 2c_{\infty}n}\#\{x,\exists t\in[i,i+1),X_t=x\}\right)\\
&\leq o(1) + \frac{4c_{\infty}}{C}\E_{\Qprob^m_0}\left(\#\{x,\exists t\in[0,1),X_t=x\}\right) \text{ for } n \text{ large enough.}
\end{aligned}
\]
Now we just have to prove that $\E_{\Qprob^m_0}\left(\#\{x,\exists t\in[0,1),X_t=x\}\right)$ is finite. We use the fact that $\frac{\dd\Qprob^m}{\dd\Prob}$ is in $L^{\gamma}$ and therefore $\frac{\dd\Qprob^m_0}{\dd\Prob_0}$ is also in $L^{\gamma}$.
\[
\begin{aligned}
\E_{\Qprob^m_0}\left(\#\{x,\exists t\in[0,1),X_t=x\}\right)&=\E_{\Prob_0}\left(\#\{x,\exists t\in[0,1),X_t=x\}\frac{\dd\Qprob^m_0}{\dd\Prob_0}\right)\\
&\leq \E_{\Prob_0}\left(\#\{x,\exists t\in[0,1),X_t=x\}^{\beta}\right)^{\frac{1}{\beta}}\left(\E_{\Prob_0}\left(\frac{\dd\Qprob^m_0}{\dd\Prob_0}\right)^{\gamma}\right)^{\frac{1}{\gamma}}.
\end{aligned}
\]
So we just need to prove that $\E_{\Prob}\left(\#\{x,\exists t\in[0,1),X_t=x\}^{\beta}\right)$ is finite. This is an immediate consequence of lemma $4$ of $\cite{BouchetSubbal}$. Therefore, for $C$ large enough, we get: 
\[
\Qprob_0^m(R_n\geq Cn)\leq o(1)+\frac{1}{2}.
\]
Therefore, the number of different points the walk visits between two renewal times has a finite expectation.
\end{proof}

Now, we show that the trajectory of the walk cannot deviate too much from a straight line.
\begin{lem}\label{lem:10}
Let $L^{\tau}(n)=\min\{i,\tau_i\geq n\}$. There exists $D\in\R^{d}$ such that $\Prob_0$ almost surely:
\[
\frac{Y_n}{L^{\tau}(n)}\rightarrow D.
\] 
\end{lem}
\begin{proof}
By proposition \ref{prop:1}, $(Y_{\tau_{i+1}}-Y_{\tau_i})_{i\geq 1}$ is a sequence of iid random variables (for $\Prob_0$). Let $R_i:= \#\left\{x\in\Z^d,\exists j< \tau_i,\ Y_j=x \right\}$ be the number of different points visited before time $\tau_i$. By lemma \ref{lem:9} , $R_i-R_{i-1}$ has a finite expectation and since $|Y_{\tau_{i+1}}-Y_{\tau_i}|_{1}\leq R_{i+1}-R_i$, we get that $|Y_{\tau_{i+1}}-Y_{\tau_i}|_{1}$ also has a finite expectation. So there exists $D\in\Z^d$ such that $\Prob_0$ almost surely:
\[
\frac{Y_{\tau_n}}{n}\rightarrow D.
\]
Now we want to show that $\frac{\left|Y_n-Y_{\tau(L^{\tau}_n)}\right|_1}{L^{\tau}(n)}\rightarrow 0$, $\Prob_0$ almost surely.
We clearly have:
\[
\frac{\left|Y_n-Y_{\tau(L^{\tau}(n))}\right|_1}{L^{\tau}(n)}\leq \frac{R_{L^{\tau}(n)}-R_{L^{\tau}(n)-1}}{L^{\tau}(n)}
\]
but since $\E_{\Prob_0}(R_i-R_{i-1})$ is finite, $\frac{R_i-R_{i-1}}{i}\rightarrow 0$, $\Prob_0$ almost surely, so:
\[
\frac{\left|Y_n-Y_{\tau(L^{\tau}(n))}\right|_1}{L^{\tau}(n)}\rightarrow 0, \ \Prob_0 \text{ almost surely }.
\]
So we get that $\Prob_0$ almost surely:
$\frac{Y_n}{L^{\tau}(n)}\rightarrow D.$
\end{proof}

\subsection{Number of visits of traps}

This section is devoted to refining some results of $\cite{BouchetSubbal}$ to get an upper bound on the number of visits of traps. First we must get some results on finite graphs and then we will extend these results on $\Z^d$.\\
\begin{defin}
Let $G=(V,E)$ be a finite, directed graph. A vertex $\delta\in V$ is a cemetery vertex if
\begin{itemize}
\item no edge exits $\delta$, ie $\forall x\in V, (\delta,x)\not in E$,
\item for every vertex $x\in V$ there exists a directed path from $x$ to $\delta$.
\end{itemize}
\end{defin}
In this section we will only consider graphs with no multiple edges, no elementary loops (one edge starting and ending at the same point), and such that for every $x,y\in V\backslash \{\delta\}$, $(x,y)\in E$ if and only if $(y,x \in E)$.\\
We will first extend the definition of $\gamma^m_{\omega}(x)$ for those graphs. Let $G=(V\cup \{\delta\},E)$ be a finite directed graph, $(\alpha(e))_{e\in E}$ be a family of real numbers, and $\Prob^{\alpha}$ be the corresponding Dirichlet distribution (independent at each site).
\begin{defin}
For $x\in G$ and $\Lambda\subset V\cup\{\delta\}$, we define the following generalization of $\gamma^m_{\omega}$:
\[
\gamma_{G,\omega}^{\Lambda}(x):=\frac{1}{\sum\limits_{\sigma}\omega_{\sigma}},
\]
where we sum on simple paths from $x$ to the border of $\Lambda$ (i.e $\{y\in\Lambda,\exists z \not\in\Lambda, \{x,y\}\in V\}$) that stay in $\Lambda$.
\end{defin}
\begin{rmq}
We notice that, in $\Z^d$, for any $m\in\N^{*}$:
\[
\forall x\in\Z^d,\ \gamma_{\omega}^m(x)=\gamma_{\Z^d,\omega}^{x+[\![-m,m]\!]^d}(x).
\] 
\end{rmq}
We will also use the following acceleration function.
\begin{defin}
For any graph $G$ and any environment $\omega$ on $G$ we define the partial acceleration function $\gamma_G^{\omega}$ by: 
\[
\gamma_G^{\omega}(x)=\max\limits_{y\sim x}\left(\frac{1}{1-\omega(x,y)+1-\omega(y,x)}\right).
\] 
When there is no ambiguity we will write $\gamma^{\omega}(x)$ instead of $\gamma_G^{\omega}(x)$
\end{defin}
\begin{rmq}
Let $x$ be a vertex in $\Z^d$. If it is in a trap then $\gamma^{\omega}(x)$ is equal to the strength of the trap. Otherwise $\gamma^{\omega}(x) \leq 2 $.
\end{rmq}
We have the following result, in the case of finite graphs:
\begin{lem}\label{lem:3}
(Proposition A.2 of $\cite{BouchetSubbal}$) \\
Let $n\in\N^*$. Let $G=(V\cup\{\delta\},E)$ be a finite directed graph possessing at most $n$ edges and such that every vertex is connected to $\delta$ by a directed path. We furthermore suppose that $G$ has no multiple edges, no elementary loop, and that if $(x,y)\in E$ and $y\not = \delta$, then $(y,x)\in E$. Let $(a(e))_{e\in E}$ be positive real numbers. Then, for every vertex $x\in V$, there exist real numbers $C,r>0$ such that, for small $\varepsilon>0$,
\[
\Prob^{(a)}\left( \gamma_{G,\omega}^{\{\delta\}}(x)  \geq\frac{1}{\varepsilon}\right)
\leq C \varepsilon^{\beta}(-\ln \varepsilon)^r
\] 
where the value of $\beta$ is explicit and given in $\cite{BouchetSubbal}$ but to simplify the notations we will only use the fact that it is bigger than or equal to $\kappa^{\prime}$ in the case we will look at.
\end{lem}
\begin{lem}\label{lem:4}
(Lemma 8 of $\cite{IntegExitTime}$) \\
Let $(p_i^{(1)})_{1\leq i \leq n_1},\dots,(p_i^{(r)})_{1\leq i \leq n_r}$ be independent Dirichlet random variables with respective parameters $(\alpha_i^{(1)})_{1\leq i \leq n_1},\dots,(\alpha_i^{(r)})_{1\leq i \leq n_r}$. Let $m_1,\dots,m_r$ be integers such that $\forall i\leq r, 1\leq m_i < n_i$, and let $\Sigma=\sum\limits_{j=1}^r\sum\limits_{i=1}^{m_j}p_i^{(j)}$ and $\beta=\sum\limits_{j=1}^r\sum\limits_{i=1}^{m_j}\alpha_i^{(j)}$. There exists positive constants $C,C^{\prime}$ such that, for any positive measurable function $f:\R\times \R^{\sum_j m_j}\mapsto \R$,
\[
\E\left[f\left(\frac{p_{1}^{(1)}}{\Sigma},\dots,\frac{p_{m_1}^{(1)}}{\Sigma},\dots,\frac{p_{1}^{(r)}}{\Sigma},\dots,\frac{p_{m_r}^{(r)}}{\Sigma}\right)\right]
\leq C \tilde{\E}\left[f\left(\tilde{p}_{1}^{(1)},\dots,\tilde{p}_{m_1}^{(1)},\dots,\tilde{p}_{1}^{(r)},\dots,\tilde{p}_{m_r}^{(r)}\right) \right],
\]
where, under the probability $\tilde{\Prob}$, $(\tilde{p}_{1}^{(1)},\dots,\tilde{p}_{m_1}^{(1)},\dots,\tilde{p}_{1}^{(r)},\dots,\tilde{p}_{m_r}^{(r)})$ is sampled from a Dirichlet distribution of parameter $(\tilde{\alpha}_{1}^{(1)},\dots,\tilde{\alpha}_{m_1}^{(1)},\dots,\tilde{\alpha}_{1}^{(r)},\dots,\tilde{\alpha}_{m_r}^{(r)})$.
\end{lem}
The following lemma shows that the value of the acceleration function $\gamma^m_{\omega}(x)$ depends mostly on the strength of the trap that contains $x$ (if there is one). This means that the number of visits to a vertex depends mostly on the strength of the trap containing this vertex.
\begin{lem}\label{lem:5}
Set $\alpha\in (0,\infty)^{2d}$. In $\Z^d$, for any $\beta \in \left[\kappa,\frac{\kappa+\kappa^{\prime}}{2}\right)$, for any $m\geq 2$:
\[
\E_{\Prob^{(\alpha)}_0}\left( \left(\frac{\gamma_{\omega}^m(0)}{\gamma^{\omega}_{\Z^d}(0)}\right)^{\beta}\right) < \infty.
\]
\end{lem}
\begin{proof}
Let $m\geq 2$ be an integer. We will use the results we have on finite graphs for this lemma. First we notice that the value of $\left(\frac{\gamma^m_{\omega}(0)}{\gamma^{\omega}(0)}\right)^{\beta}$ only depends on a finite amount of edges and vertices around $0$. This means that we can look at this quantity on a finite graph and have the same law. The finite graph $G^m=(V^m,E^m)$ we want is obtained by contracting all the points $x\in \Z^d$ such that $||x||_{1}\geq m$ in a single point $\delta$ (the cemetery vertex) and deleting all the edges going from this vertex to the rest of the environment. For any environment $\omega$ on $\Z^d$ we have an equivalent environment $\omega^m$ on $G^m$: if $(x,y)\in E$ and $(x,y)\in E^m$ then $\omega(x,y)=\tilde{\omega}(x,y)$ and for any $x\in V^m \backslash \{\delta\}$, $\tilde{\omega}(x,\delta)=\sum\limits_{y\in\Z^d,||y||_1=m} \omega(x,y)$. Now we have:
\[
\gamma^m_{\omega}(0)= \gamma^{\{\delta\}}_{G^m,\omega^m}(0)
\]  
and
\[
\gamma^{\omega}_{\Z^d}(0)=\gamma^{\omega^m}_{G^m}(0).
\]
So we just have to show that
\[
\E_{\Prob^{(\alpha)}}\left( \left(\frac{\gamma^{\{\delta\}}_{G^m,\omega^m}(0)}{\gamma^{\omega^m}_{G^m}(0)}\right)^{\beta}\right) < \infty.
\]
For any point $y\sim 0$ and any environment $\omega$ we define $\Sigma^{\omega}_y$ by:
\[
\Sigma^{\omega}_y=2-\omega(0,y)-\omega(y,0).
\]
For any point $x\in G^m$ such that $x\sim 0$, we define $G^m_x=(V^m_x,E^m_x)$ by contracting the vertices $0$ and $x$ into a single vertex $0$ and deleting the edges $(0,x)$ and $(x,0)$. The edges $(0,y)$ and $(y,0)$ stay the same for any $y\sim 0$ such that $x\not=y$. However, the edges $(x,y)$ and $(y,x)$ become $(0,y)$ and $(y,0)$ respectively, for any $y\sim x$ such that $0\not=y$. We can also define $\omega^m_x$ by:
\[
\begin{aligned}
&\forall (y,z)\in E^m, y\not\in \{0,x\},\ \omega^m_x(y,z):=\omega^m(y,z)\\
&\forall (y,z)\in E^m, y\in \{0,x\}, (y,z)\in E^m_x,\ \omega^m_x(y,z):=\frac{\omega^m(y,z)}{\Sigma^{\omega}_y}
\end{aligned}
\] 
Let $x\sim 0$ be a vertex of $G^m$. If we think of $\frac{1}{\gamma^{\{\delta\}}_{G^m,\omega^m}}$ as a sum on simple paths, we have:
\[
\frac{1}{\gamma^{\{\delta\}}_{G^m,\omega^m}}\geq \Sigma^{\omega^m}_x \omega^m(0,x)\frac{1}{\gamma^{\{\delta\}}_{G^m,\omega^m}}
\]
Indeed, if we look at $\frac{1}{\gamma^{\{\delta\}}_{G^m_x,\omega^m_x}}$ as a sum on simple paths $\sigma$ from $0$ to $\delta$ ($\sigma_0=0$), either the first vertex $\sigma_1$ visited by the path is such that $(0,\sigma_1)\in E^m$ or $(x,\sigma_1)\in E^m$. We define $\tilde{\sigma}$ by: if $(0,\sigma_1)\in E^m$ then $\tilde{\sigma}:=\sigma$ and we have:
\[
\omega^m(\tilde{\sigma})=\Sigma^{\omega^m}_x \omega^m_x(\sigma)\geq \Sigma^{\omega^m}_x \omega^m(0,x) \omega^m_x(\sigma),
\]
and if $(x,\sigma_1)\in E^m$ then $\tilde{\sigma}_i:=\sigma_{i-1}$ for $i\geq 2$ and $\tilde{\sigma}_0:=0$ and $\tilde{\sigma}_1:=x$ and we get:
\[
\omega^m(\tilde{\sigma})=\Sigma^{\omega^m}_x \omega^m(0,x) \omega^m_x(\sigma).
\]
For any environment $\omega$, let $x(\omega^m)$ be the point that maximises $(y\rightarrow\omega^m(0,y))$. We have $\tilde{\omega}(0,y)\geq \frac{1}{2d}$ and therefore:
\[
\frac{1}{\gamma^{\{\delta\}}_{G^m,\omega^m}}
\geq \frac{1}{2d}\Sigma^{\omega^m}_{x(\omega^m)} \frac{1}{\gamma^{\{\delta\}}_{G^m_{x(\omega^m)},\omega^m_{x(\omega^m)}}}.
\]
So we get, for any $\varepsilon>0$:
\[
\begin{aligned}
\Prob^{(\alpha)}\left( \left(\frac{\gamma^{\{\delta\}}_{G^m,\omega^m}(0)}{\gamma^{\omega^m}_{G^m}(0)}\right)\geq \frac{1}{\varepsilon}\right)
&\leq \Prob^{(\alpha)}\left(\frac{2d\varepsilon}{\gamma^{\omega^m}_{G^m}(0)}\geq \Sigma^{\omega^m}_{x(\omega^m)} \frac{1}{\gamma^{\{\delta\}}_{G^m_{x(\omega^m)},\omega^m_{x(\omega^m)}}}\right)\\
&=\sum\limits_{y\sim 0}\Prob^{(\alpha)}\left(y=x(\omega^m) \text{ and } \frac{2d\varepsilon}{\gamma^{\omega^m}_{G^m}(0)}\geq \Sigma^{\omega^m}_{y} \frac{1}{\gamma^{\{\delta\}}_{G^m_{y},\omega^m_{y}}}\right)\\
&\leq \sum\limits_{y\sim 0}\Prob^{(\alpha)}\left(\frac{2d\varepsilon}{\gamma^{\omega^m}_{G^m}(0)}\geq \Sigma^{\omega^m}_{y} \frac{1}{\gamma^{\{\delta\}}_{G^m_{y},\omega^m_{y}}}\right).
\end{aligned}
\]
by definition of $\gamma^{\omega^m}_{G^m}(0)$:
\[
\forall y\sim 0,\ \gamma^{\omega^m}(0)_{G^m}\Sigma^{\omega^m}_y\geq 1.
\]
Therefore:
\[
\Prob^{(\alpha)}\left( \left(\frac{\gamma^{\{\delta\}}_{G^m,\omega^m}(0)}{\gamma^{\omega^m}_{G^m}(0)}\right)\geq \frac{1}{\varepsilon}\right)
\leq \sum\limits_{y\sim 0}\Prob^{(\alpha)}\left(2d\varepsilon \geq \frac{1}{\gamma^{\{\delta\}}_{G^m_{y},\omega^m_{y}}}\right).
\]
Now we can apply lemma \ref{lem:4} which gives, for any $y\sim 0$:
\[
\Prob^{(\alpha)}\left(2d\varepsilon \geq \frac{1}{\gamma^{\{\delta\}}_{G^m_{y},\omega^m_{y}}}\right)
\leq C \tilde{\Prob}\left(2d\varepsilon \geq \frac{1}{\gamma^{\{\delta\}}_{G^m_{y},\omega^m_{y}}}\right),
\]
where under $\tilde{\Prob}$, $\omega^m_{y}$ are independent Dirichlet random variables (on the graph $G^m_y$ and the parameters of the Dirichlet are the same as in $\Z^d$). Now, according to lemma \ref{lem:3} there exists two constants $C^{\prime}$, $r$ such that:
\[
\forall \varepsilon \text{ \ \ small enough \ \ }, \tilde{\Prob}\left(2d\varepsilon \geq \frac{1}{\gamma^{\{\delta\}}_{G^m_{y},\omega^m_{y}}}\right) \leq C^{\prime}\varepsilon^{\kappa^{\prime}}\left(-\log(\varepsilon)\right)^r.
\]
This means that by changing the constant $C^{\prime}$, we get:
\[
\forall \varepsilon \geq 0, \tilde{\Prob}\left(2d\varepsilon \geq \frac{1}{\gamma^{\{\delta\}}_{G^m_{y},\omega^m_{y}}}\right) \leq C^{\prime}\varepsilon^{\frac{\kappa+\kappa^{\prime}}{2}}.
\]
So there exists a constant $D$ that does not depend on $\varepsilon$ such that:
\[
\Prob^{(\alpha)}\left( \left(\frac{\gamma^{\{\delta\}}_{G^m,\omega^m}(0)}{\gamma^{\omega^m}_{G^m}(0)}\right)\geq \frac{1}{\varepsilon}\right) \leq
D\varepsilon^{\frac{\kappa+\kappa^{\prime}}{2}}.
\]
We have the result we want.
\end{proof}
Unfortunately this statement cannot be efficiently used with the invariant distribution $\Qprob^m$ because we can visit multiple points between times 0 and 1 since the time is continuous. So we need a version of the previous lemma that takes this continuity into account.
\begin{lem}\label{lem:15}
Set $\alpha\in (0,\infty)^{2d}$. For every $\beta < \frac{\kappa+\kappa^{\prime}}{2}$, there exists an integer $m$ such that:
\[
\E_{\Qprob_0^m}\left(\sum\limits_{x\in\Z^d}\left(\int\limits_{t=0}^1 \frac{\gamma^m_{\omega}(x)}{\gamma^{\omega}_{\Z^d}(x)}1_{X^m_t=x}\dd t\right)^{\beta}\right)<\infty.
\]
\end{lem}
\begin{proof}
Let $p\in(1,\infty)$ be a constant such that $\beta p^2<\frac{\kappa+\kappa^{\prime}}{2}$ and let $\gamma$ be such that $\frac{1}{p}+\frac{1}{\gamma}=1$. Now let $m$ be an integer such that $\frac{\dd\Qprob^m}{\dd\Prob}$ is in $L^{\gamma}$. This means that $\frac{\dd\Qprob_0^m}{\dd\Prob_0}$ is also in $L^{\gamma}$. We will only work in $\Z^d$ so we will write $\gamma^{\omega}$ instead of $\gamma^{\omega}_{\Z^d}$.
\[
\begin{aligned}
&\E_{\Qprob_0^m}\left(\sum\limits_{x\in\Z^d}\left(\int\limits_{t=0}^1 \frac{\gamma^m_{\omega}(x)}{\gamma^{\omega}(x)}1_{X^m_t=x}\dd t\right)^{\beta}\right)\\
=&\sum\limits_{x\in\Z^d}\E_{\Prob_0}\left(\left(\int\limits_{t=0}^1 \frac{\gamma^m_{\omega}(x)}{\gamma^{\omega}(x)}1_{X^m_t=x}\dd t\right)^{\beta}\frac{\dd\Qprob_0^m}{\dd\Prob_0}\right)\\
\leq &\sum\limits_{x\in\Z^d} \E_{\Prob_0}\left(\left(\int\limits_{t=0}^1 \frac{\gamma^m_{\omega}(x)}{\gamma^{\omega}(x)}1_{X^m_t=x}\dd t\right)^{p\beta}\right)^{\frac{1}{p}}\E_{\Prob_0}\left(\left(\frac{\dd\Qprob_0^m}{\dd\Prob_0}\right)^{\gamma}\right)^{\frac{1}{\gamma}}.
\end{aligned}
\]
This means we just need to show that $\sum\limits_{x\in\Z^d} \E_{\Prob_0}\left(\left(\int\limits_{t=0}^1 \frac{\gamma^m_{\omega}(x)}{\gamma^{\omega}(x)}1_{X^m_t=x}\dd t\right)^{p\beta}\right)^{\frac{1}{p}}$ is finite. Let $D^m_1$ be the random variable defined by:
\[
D^m_1:= \sum\limits_{i=1}^d \max\limits_{t\in [0,1]}|X^m_t.e_i|.
\]
We have:
\[
\begin{aligned}
\sum\limits_{x\in\Z^d} \E_{\Prob_0}\left(\left(\int\limits_{t=0}^1 \frac{\gamma^m_{\omega}(x)}{\gamma^{\omega}(x)}1_{X^m_t=x}\dd t\right)^{p\beta}\right)^{\frac{1}{p}}
&\leq \sum\limits_{x\in\Z^d} \E_{\Prob_0}\left(\left(\frac{\gamma^m_{\omega}(x)}{\gamma^{\omega}(x)}\right)^{p\beta}1_{\exists t\in[0,1],X^m_t=x}\right)^{\frac{1}{p}}\\
&\leq \sum\limits_{x\in\Z^d} \left(\E_{\Prob_0}\left(\left(\frac{\gamma^m_{\omega}(x)}{\gamma^{\omega}(x)}\right)^{p\beta}1_{D^m_1\geq ||x||_{\infty}}\right)\right)^{\frac{1}{p}}\\
&\leq \sum\limits_{x\in\Z^d} \left(\E_{\Prob_0}\left(\left(\frac{\gamma^m_{\omega}(x)}{\gamma^{\omega}(x)}\right)^{p^2\beta }\right)\right)^{\frac{1}{p^2}}\left(\E_{\Prob_0}\left(1_{D^m_1\geq ||x||_{\infty}}\right)\right)^{\frac{1}{\alpha}}\\
&= \sum\limits_{x\in\Z^d} \left(\E_{\Prob}\left(\left(\frac{\gamma^m_{\omega}(x)}{\gamma^{\omega}(x)}\right)^{p^2\beta }\right)\right)^{\frac{1}{p^2}}\left(\E_{\Prob_0}\left(1_{D^m_1\geq ||x||_{\infty}}\right)\right)^{\frac{1}{\alpha}}
\end{aligned}
\]
Now since the environment for $\Prob$ is iid, $\E_{\Prob}\left(\left(\frac{\gamma^m_{\omega}(x)}{\gamma^{\omega}(x)}\right)^{p^2\beta}\right)$ does not depend on $x$ and we get:
\[
\sum\limits_{x\in\Z^d} \E_{\Prob_0}\left(\left(\int\limits_{t=0}^1 \frac{\gamma^m_{\omega}(x)}{\gamma^{\omega}(x)}1_{X^m_t=x}\dd t\right)^{p\beta}\right)^{\frac{1}{p}}
\leq \left(\E_{\Prob}\left(\left(\frac{\gamma^m_{\omega}(0)}{\gamma^{\omega}(0)}\right)^{p^2\beta}\right)\right)^{\frac{1}{p^2}} \sum\limits_{x\in\Z^d} \left(\E_{\Prob_0}\left(1_{D^m_1\geq ||x||_{\infty}}\right)\right)^{\frac{1}{\gamma}}.
\]
And since there exists a constant $C$ such that for every $i \geq 1$ there are at most $Ci^{d-1}$ points $x$ such that $||x||_{\infty}=i$, we get:
\[
\sum\limits_{x\in\Z^d} \left(\E_{\Prob_0}\left(1_{D_1\geq ||x||_{\infty}}\right)\right)^{\frac{1}{\alpha}}
\leq 1+C\sum\limits_{i\geq 1} i^{d-1}\left(\E_{\Prob_0}\left(1_{D_1\geq i}\right)\right)^{\frac{1}{\alpha}}
\]
which is finite by lemma 4 of $\cite{BouchetSubbal}$. And by lemma \ref{lem:5} we get:
\[
\E_{\Prob}\left(\left(\frac{\gamma^m_{\omega}(0)}{\gamma^{\omega}(0)}\right)^{p^2\beta}\right)< \infty.
\]
So we get the result we want.
\end{proof}

\subsection{Independence of the traps}

This section will be devoted to the precise study of traps. The notion of trap was defined in the introduction in definition\ref{defTrap}. In the previous section we have essentially shown that the total amount of time spent on a trap mostly depends on its strength. Now, we need a way to create independence between the times spent in the different traps. We will do it in two steps. First we will show that the strength of the traps are essentially independent and then we will show that the strength of a trap and the number of times it is visited are essentially independent. However, we first need to introduce a few objects to characterize this independence precisely.  
\begin{defin}
Let $\mathcal{T}^{\omega}$ be the set of traps $\{x,y\}\in \tilde{E}$ for the environment $\omega$.\\
$\tilde{\mathcal{T}}^{\omega}$ is the set of vertices $x\in \Z^d$ such that there exist $y$ such that $\{x,y\}\in\mathcal{T}^{\omega}$. \\
For any subset $J$ of $[\![1,d]\!]$ we define $\mathcal{T}^{\omega}_J$, the traps with direction in $J$ by:
\[
\mathcal{T}^{\omega}_J = \{\{x,y\}\in\mathcal{T},\ \exists j\in J, y=x+e_j \text{ or } y=x-e_{j} \}.
\]
For any subset $J$ of $[\![1,d]\!]$, $\tilde{\mathcal{T}}^{\omega}_J$ is the set of vertices $x\in \Z^d$ such that there exist $y$ such that $\{x,y\}\in\mathcal{T}^{\omega}_J$. \\
In the following we will omit the $\omega$ when there is no ambiguity.
\end{defin}
\begin{defin}
We say that two environments $\omega_1$ and $\omega_2$ are trap-equivalent if:\\
- they have the same traps:
\[
\mathcal{T}^{\omega_1}=\mathcal{T}^{\omega_2},
\]
- at each vertex not in a trap, the transition probabilities are the same for both environment:
\[
\forall x \not\in\tilde{\mathcal{T}}^{\omega_1},\ \forall y\sim x,\ \omega_1(x,y)=\omega_2(x,y),
\]
- at each vertex $x$ in a trap $\{x,y\}$, the transition probabilities conditioned on not crossing the trap are the same:
\[
\forall (x,y)\in E,\ \{x,y\}\in \mathcal{T}^{\omega_1},\ \forall z\sim x, z\not = y,
 \frac{\omega_1(x,z)}{1-\omega_1(x,y)}=\frac{\omega_2(x,y)}{1-\omega_2(x,y)}.
\]
We will denote by $\tilde{\Omega}$ the set of all equivalence classes for the trap-equivalence relation.
\end{defin}
\begin{defin}
Set $\tilde{\omega}\in\tilde{\Omega}$. Let $\mathcal{T}$ be its set of trap and $\sigma$ a path starting at $0$ that only stays a finite amount of time every time it enters a trap. We want to define a path, with the same trajectory as $\sigma$ outside the traps, which does not keep information regarding the time spent in the traps. We essentially want to erase all the back and forths inside traps. To that extent we define the sequences of integer times $(t_i),(s_i)$ by:
\[
\begin{aligned}
&t_0=0,\\
&s_i=\inf\{n\geq t_i, (\sigma_{n}=\sigma_{t_i}\text{ or }\{\sigma_{n},\sigma_{t_i}\}\in\mathcal{T})\text{ and } (\sigma_{n+1}\not =\sigma_{t_i}\text{ and }\{\sigma_{n+1},\sigma_{t_i}\}\not\in\mathcal{T})\},\\
&t_{i+1}=
\begin{cases}
s_i+1 &\text{ if }\sigma_{s_i}=\sigma_{t_i}\\
s_i &\text{ otherwise. }
\end{cases}
\end{aligned}
\] 
If $\sigma_{t_i}$ is in a trap then $[t_i,s_i]$ is the interval of time spent in this trap before leaving it.\\
The partially forgotten path $\tilde{\sigma}$ associated with $\sigma$ in the environment $\tilde{\omega}$ is defined by:
\[
\tilde{\sigma}_i:=\sigma_{t_i}.
\]
Similarly we can define the partially-forgotten walk $(\tilde{Y}_n)_{n\in\N}$ associated with $(Y_n)_{n\in\N}$ 
\end{defin}
\begin{defin}
For all $i\in\N^*$, let $I_i$ be the set defined by:
\[
I_i= [\![1,d]\!] \times \{a,b,c,d\in \N , a\geq1 , a+b+c+d=i\}.
\]
And $I^n$ be defined by:
\[
I^n=\bigcup\limits_{1\leq i\leq n} I_i.
\]
Let $\sigma$ be a path starting at $0$ and $\tilde{e}\in\tilde{E}$ be an undirected edge. We define the sequences $(t^{\text{in}}_i)$ (the times when the path enters $\tilde{e}$) and $(t^{\text{out}}_i)$ (the times when the path exits $\tilde{e}$) by:
\[
\begin{aligned}
t^{\text{in}}_1 =& \inf \{ n, \tilde{Y}_n \in \tilde{e} \}, \\
t^{\text{in}}_{i+1} =& \inf \{ n > t^{\text{in}}_{i} , \tilde{Y}_n \in \tilde{e} \text{ and } \tilde{Y}_{n-1}\not\in\tilde{e}\}, \\
t^{\text{out}}_i =& \inf \{ n \geq t^{\text{in}}_i, \tilde{Y}_n \in \tilde{e} \text{ and } \tilde{Y}_{n+1}\not\in\tilde{e}\} \}. \\
\end{aligned}
\]
Since the walk is almost surely transient by theorem $\ref{theodule}$, we have that for $i$ large enough $t^{\text{in}}_i=t^{\text{out}}_i=\infty$ almost surely.\\
Now let $x:=\sigma_{t^{\text{in}}_1}$ and $y$ be such that $\{x,y\}=\tilde{e}$. Let $j\in [\![1,d]\!]$ be such that either $x = y + e_j$ or $x = y - e_j$ ($j$ is the direction of the edge) and $n$ be such that $t^{\text{in}}_n <\infty$ and $t^{\text{in}}_{n+1}= \infty$. Now we can define $N_{x\rightarrow x},N_{x\rightarrow y},N_{y\rightarrow x},N_{y\rightarrow y}$ by:
\[
\begin{aligned}
N_{x\rightarrow x} =& \#\{i \leq n ,t^{\text{in}}_i = x \text{ and } t^{\text{out}}_i = x \}, \\
N_{x\rightarrow y} =& \#\{i \leq n ,t^{\text{in}}_i = x \text{ and } t^{\text{out}}_i = y \}, \\
N_{y\rightarrow x} =& \#\{i \leq n ,t^{\text{in}}_i = y \text{ and } t^{\text{out}}_i = x \}, \\
N_{y\rightarrow y} =& \#\{i \leq n ,t^{\text{in}}_i = y \text{ and } t^{\text{out}}_i = y \}.
\end{aligned}
\] 
The configuration $p$ of the edge $\tilde{e}$, for the path $\sigma$, is the element of $I_n$ defined by:
\[
p^{\sigma}_{\{x,y\}}:= (j, N_{x\rightarrow x},N_{x\rightarrow y},N_{y\rightarrow x},N_{y\rightarrow y}).
\]
\end{defin}
\begin{rmq}
Set $\tilde{\omega}\in\tilde{\Omega}$. Let $\sigma_1,\sigma_2$ be two paths starting at $0$ with the same partially forgotten path in $\tilde{\omega}$. For any undirected edge $\tilde{e}$, the configuration of $\tilde{e}$ is the same for $\sigma_1$ and $\sigma_2$. Therefore we only need to know the partially forgotten path to know the configuration of an edge.
\end{rmq}
Now we can say in what way the strength of the traps are independent.
\begin{lem}\label{lem:11}
For any environment $\omega\in\Omega$, let $\tilde{\omega}\in\tilde{\Omega}$ be its equivalence class for the trap-equivalent relation. Now let $(\tilde{Y}_i)$ be the partially forgotten walk. We will write $\overline{\alpha}:= \sum\limits_{1\leq i \leq 2d} \alpha_i$ and for any vertex $z$ and integer $i$ we will use the notation $\alpha(z,z+e_i):=\alpha_i$. Knowing $\tilde{\omega}$ and $(\tilde{Y}_i)$, the strength of the various traps are independent. Furthermore, let $\{x,y\}$ be a trap and $p=(j, N_{x\rightarrow x},N_{x\rightarrow y},N_{y\rightarrow x},N_{y\rightarrow y})$ its configuration. To simplify notations we will write $N_x:=N_{x\rightarrow x}+N_{y\rightarrow x}$, $N_y:=N_{x\rightarrow y}+N_{y\rightarrow y}$ and $N:= N_x+ N_y$. Let $(r,k)$ be defined by $(1-\omega(x,y),1-\omega(y,x))=((1+k)r,(1-k)r)$. The density of law of $(r,k)$ (with respect to the Lebesgue measure) knowing $\tilde{\omega}$ and $\tilde{Y}$ is:
\[
C_p r^{\kappa_j-1} (1+k)^{N_x+\overline{\alpha}-\alpha(x,y)-1}(1-k)^{N_y+\overline{\alpha}-\alpha(y,x)-1}h_p(r(1+k),r(1-k))1_{0\leq r\leq\frac{1}{4}}1_{-1\leq k \leq 1},
\]
where $C_p$ is a constant that only depends on $p$ and $\alpha$, and $h_p$ is a function that only depends on $p$ and $\alpha$ and that satisfies the following bound:
\[
\forall r\leq \frac{1}{4},\ \ 
|\log(h_p(r(1+k),r(1-k)))|
\leq 5(N+2\overline{\alpha})r.
\]
And for the law of the strength $s$ of the trap, there exists a constant $D$ that only depends on the configuration of the trap such that for any $A\geq 2$: 
\[
D A^{-\kappa_j} \exp\left( -\frac{5(N + 2\overline{\alpha})}{A}\right)
\leq \Prob_0\left(s \geq A|\tilde{\omega},\tilde{Y}\right) 
\leq D A^{-\kappa_j} \exp\left( \frac{5(N + 2\overline{\alpha})}{A}\right).
\]
\end{lem}
\begin{proof}
In the following, we will write $\overline{\alpha}:=\sum\limits_{i=1}^{2d}\alpha_i$ and if $y=x+e_i$ we will write $\alpha(x,y):=\alpha_i$. First we need to show that the strength of the traps is approximately independent of the trajectory of the walk. We will take an environment $\omega$ and let $\tilde{\omega}$ be the set of all environments that are trap-equivalent to $\omega$. Now for any path $\sigma$ starting at $0$, let $\tilde{\sigma}^{\tilde{\omega}}$ be the set of all path that starts at $0$ and that have the same partially-forgotten path as $\sigma$. We want to see how the law of the environment is changed knowing the partially-forgotten path and the equivalence class of the environment. We get that the density of the environment (we look at an environment of finite size, large enough to contain the path we look at) (for $\Prob^{(\alpha)}$) knowing the equivalence class of the environment is equal to:
\begin{equation}
C\prod\limits_{\{x,y\}\in \mathcal{T}}(\varepsilon_x)^{\overline{\alpha}-\alpha(x,y)-1}(1-\varepsilon_x)^{\alpha(x,y)-1}
(\varepsilon_y)^{\overline{\alpha}-\alpha(y,x)-1}(1-\varepsilon_y)^{\alpha(y,x)-1}
1_{\varepsilon_x+\varepsilon_y<\frac{1}{2}}\dd \varepsilon_x \dd \varepsilon_y,
\label{eqn:11.1}
\end{equation}
where $\varepsilon_x=1-\omega(x,y)$ and $\varepsilon_y=1-\omega(y,x)$. Now, knowing the environment, the probability of having the given partially-forgotten walk is the same in parts of the environment where there is no trap. The only thing that depends on the specific environment is the times when the walk crosses the traps. Let $\{x,y\}$ be a trap, and for any $z_1,z_2\in \{x,y\}$ let $\tilde{p}(z_1,z_2)$ be the probability to exit the path by $z_2$, starting at $z_1$, we get:
\[
\begin{aligned}
&\tilde{p}(x,x)=\frac{\varepsilon_x}{\varepsilon_x+\varepsilon_y-\varepsilon_x\varepsilon_y}, \ \ \ \ 
&\tilde{p}(y,y)=\frac{\varepsilon_y}{\varepsilon_x+\varepsilon_y-\varepsilon_x\varepsilon_y}, \\
&\tilde{p}(x,y)=\frac{\varepsilon_y(1-\varepsilon_x)}{\varepsilon_x+\varepsilon_y-\varepsilon_x\varepsilon_y},
&\tilde{p}(y,x)=\frac{\varepsilon_x(1-\varepsilon_y)}{\varepsilon_x+\varepsilon_y-\varepsilon_x\varepsilon_y}.
\end{aligned}
\]
So for any environment $\omega$, we get that the probability of a partially-forgotten path (for $\prob^{(\alpha)}_0$), is equal to:  
\begin{align}
&C\prod\limits_{\{x,y\}\in \mathcal{T}}\tilde{p}(x,x)^{N_{x\rightarrow x}}\tilde{p}(x,y)^{N_{x\rightarrow y}}\tilde{p}(y,x)^{N_{y\rightarrow x}}\tilde{p}(y,y)^{N_{y\rightarrow y}} \notag\\
=&C\prod\limits_{\{x,y\}\in \mathcal{T}}
\frac{\varepsilon_x^{N_{x\rightarrow x}} (\varepsilon_y(1-\varepsilon_x))^{N_{x\rightarrow y}} (\varepsilon_x(1-\varepsilon_y))^{N_{y\rightarrow x}} \varepsilon_y^{N_{y\rightarrow y}}}{\left(\varepsilon_x+\varepsilon_y-\varepsilon_x\varepsilon_y\right)^{N_{x\rightarrow x}+N_{x\rightarrow y}+N_{y\rightarrow x}+N_{y\rightarrow y}}} \notag\\
=&C\prod\limits_{\{x,y\}\in \mathcal{T}}\frac{\varepsilon_x^{N_{x\rightarrow x}+N_{y\rightarrow x}} \varepsilon_y^{N_{x\rightarrow y}+N_{y\rightarrow y}}}{(\varepsilon_x+\varepsilon_y)^{N_{x\rightarrow x}+N_{x\rightarrow y}+N_{y\rightarrow x}+N_{y\rightarrow y}}}
\frac{(1-\varepsilon_x)^{N_{x\rightarrow y}}(1-\varepsilon_y)^{N_{y\rightarrow x}}}{\left(1-\frac{\varepsilon_x\varepsilon_y}{\varepsilon_x+\varepsilon_y}\right)^{N_{x\rightarrow x}+N_{x\rightarrow y}+N_{y\rightarrow x}+N_{y\rightarrow y}}}. \label{eqn:11.2}
\end{align}
We define $h_{\{x,y\}}$ by:
\[
h_{\{x,y\}}(\varepsilon_x,\varepsilon_y)=\frac{(1-\varepsilon_x)^{N_{x\rightarrow y}}(1-\varepsilon_y)^{N_{y\rightarrow x}}}{\left(1-\frac{\varepsilon_x\varepsilon_y}{\varepsilon_x+\varepsilon_y}\right)^{N_{x\rightarrow x}+N_{x\rightarrow y}+N_{y\rightarrow x}+N_{y\rightarrow y}}}(1-\varepsilon_x)^{\alpha(x,y)-1}(1-\varepsilon_y)^{\alpha(y,x)-1}.
\]
Now we get that the density probability of having a given environment knowing the equivalence class of the environment and the partially forgotten path is equal to the product of \ref{eqn:11.1} and \ref{eqn:11.2} up to a multiplicative constant $C$ that depends on the partially-forgotten path:
\[
C\prod\limits_{\{x,y\}\in \mathcal{T}}\frac{\varepsilon_x^{N_{x\rightarrow x}+N_{y\rightarrow x}+\overline{\alpha}-\alpha(x,y)-1} \varepsilon_y^{N_{x\rightarrow y}+N_{y \rightarrow y}+\overline{\alpha}-\alpha(y,x)-1}}{(\varepsilon_x+\varepsilon_y)^{N_{x\rightarrow x}+N_{x\rightarrow y}+N_{y\rightarrow x}+N_{y\rightarrow y}}}h_{\{x,y\}}(\varepsilon_x,\varepsilon_y)1_{\varepsilon_x+\varepsilon_y< \frac{1}{2}}\dd \varepsilon_x \dd \varepsilon_y.
\]
This means that for $\Prob^{(\alpha)}_0$, knowing the equivalence class of the environment and the partially forgotten path, the transition probabilities for each trap are independent, so we will look at each trap independently. Let's fix a trap $\{x,y\}$ and to simplify notations, we will write $N_x=N_{x\rightarrow x}+N_{y\rightarrow x}$, $N_y=N_{x\rightarrow y}+N_{y\rightarrow y}$ and $N=N_x+N_y$. We define $r$ and $k$ by $r=\frac{\varepsilon_x+\varepsilon_y}{2}$ and $k=\frac{\varepsilon_x-\varepsilon_y}{\varepsilon_x+\varepsilon_y}$ which gives $\varepsilon_x=r(1+k)$ and $\varepsilon_y=r(1-k)$ the law of the transition probabilities becomes:
\[
\begin{aligned}
&C\frac{r}{2}\frac{(r(1+k))^{N_x+\overline{\alpha}-\alpha(x,y)-1}(r(1-k))^{N_y+\overline{\alpha}-\alpha(y,x)-1}}{(2r)^{N_x+N_y+2}}h_{\{x,y\}}(r(1+k),r(1-k))1_{r<\frac{1}{2}} \dd r \dd k\\
=& C^{\prime} r^{\kappa_j-1} (1+k)^{N_x+\overline{\alpha}-\alpha(x,y)-1}(1-k)^{N_y+\overline{\alpha}-\alpha(y,x)-1}h_{\{x,y\}}(r(1+k),r(1-k))1_{r<\frac{1}{2}} \dd r \dd k.
\end{aligned}
\]
Now we want to give bounds on $h_{\{x,y\}}$. Since for all $r\leq \frac{1}{2}$, $|\log(1-r)|\leq 2r$, we get:
\[
\begin{aligned}
&|\log(h_{\{x,y\}}(r(1+k),r(1-k)))| \\
\leq &|(N(x,y)+\alpha(x,y)-1)\log(1-r(1+k))|+|(N(y,x)+\alpha(y,x)-1)\log(1-r(1-k))|\\
&\qquad +|N\log(1-\frac{r(1-k^2)}{2})|\\
\leq &(N(x,y)+\alpha(x,y))4r+(N(y,x)+\alpha(y,x))4r+Nr\\
\leq &5(N_x+N_y+\alpha_x+\alpha_y)r.
\end{aligned}
\]
Let $D=\int\limits_{k= -1}^1 C^{\prime} (1+k)^{N_x+\overline{\alpha}-\alpha(x,y)-1}(1-k)^{N_y+\overline{\alpha}-\alpha(y,x)-1}$, for any $A\geq 2$, we have:
\[
D A^{-\kappa_j} \exp\left( -\frac{5(N + 2\alpha)}{A}\right)
\leq \Prob_0\left(s \geq A|\tilde{\omega},\tilde{Y}\right) 
\leq D A^{-\kappa_j} \exp\left( \frac{5(N + 2\alpha)}{A}\right).
\]
\end{proof}

Now we want to show that there cannot be too many traps that are visited many times.
\begin{lem}\label{lem:14}
Set $\alpha\in (0,\infty)^{2d}$. For any $\beta\in \left[\kappa,\frac{\kappa+\kappa^{\prime}}{2}\right)$ with $\beta\leq 1$ there exists a finite constant $C>0$ such that for every $i\in\N\setminus \{0,1\}$:
\[
\E_{\Prob_0}\left(\sum\limits_{\{x,y\}\in\mathcal{T}}\#\{j\in [\![\tau_i,\tau_{i+1}-1 ]\!],Y_j\in\{x,y\}\text{ and } Y_{j+1}\not\in\{x,y\} \}^{\beta} \right) = C.
\]
\end{lem}
\begin{proof}
We want to show that 
\[
\E_{\Prob_0}\left(\sum\limits_{\{x,y\}\in\mathcal{T}}\#\{j\in [\![\tau_i,\tau_{i+1}-1 ]\!],Y_j\in\{x,y\}\text{ and } Y_{j+1}\not\in\{x,y\} \}^{\beta} \right)
\]
can be bounded away from infinity by using the inequality from lemma \ref{lem:15}:
\[
\E_{\Qprob^m_0}\left(\sum\limits_{x\in\Z^d}\left(\int\limits_{t=0}^1 \frac{\gamma^m_{\omega}(x)}{\gamma^{\omega}(x)}1_{X_t=x}\dd t\right)^{\beta}\right)<\infty,
\]
which is true for any $\beta\in \left[\kappa,\frac{\kappa+\kappa^{\prime}}{2}\right)$, and for any integer $m$ such that $\Qprob^m_0$ exists.\\
To that end we need to introduce the intermediate quantity $S^m_n$:
\[
S^m_n:=\sum\limits_{i=0}^n\sum\limits_{\{x,y\}\in \mathcal{T}}\left(\int\limits_{T^m_{i}}^{T^m_{i+1}}\frac{\gamma^m_{\omega}(x)}{\gamma^{\omega}(x)}1_{X^m_t=x}\dd t\right)^{\beta}+\left(\int\limits_{T^m_{i}}^{T^m_{i+1}}\frac{\gamma^m_{\omega}(y)}{\gamma^{\omega}(y)}1_{X^m_t=y}\dd t\right)^{\beta},
\]
where $(T^m_i)$ are the renewal times for the walk $(X^m_t)$, with the convention that $T^m_0:=0$.
By definition of $X^m$, the time the walk $X^m$ spends in a vertex $x$ is a sum of $\ell_x$ iid exponential random variables of expectation $\frac{1}{\gamma^m_{\omega}(x)}$, where $\ell_x$ is the number of times the walk $Y$ visits the point $x$. Therefore the quantity
\[
\int_{0}^{\infty}\gamma^m_{\omega}(x)1_{X_t=x}\dd t
\]
should be close to $\ell_x$. Then, every time the walk $Y$ enters the trap $\{x,y\}$ is stays a time of order $\gamma^{\omega}(x)$. This means that $\frac{\ell_x}{\gamma^{\omega}(x)}$ should be almost equal to the number of times the trap is entered. Finally, we get that for every trap the quantities
\[
\sum\limits_{\{x,y\}\in\mathcal{T}}\#\{j\in [\![\tau_i,\tau_{i+1}-1 ]\!],Y_j\in\{x,y\}\text{ and } Y_{j+1}\not\in\{x,y\} \}^{\beta} 
\]
and
\[
\sum\limits_{\{x,y\}\in \mathcal{T}}\left(\int\limits_{T^m_{i}}^{T^m_{i+1}}\frac{\gamma^m_{\omega}(x)}{\gamma^{\omega}(x)}1_{X^m_t=x}\dd t\right)^{\beta}+\left(\int\limits_{T^m_{i}}^{T^m_{i+1}}\frac{\gamma^m_{\omega}(y)}{\gamma^{\omega}(y)}1_{X^m_t=y}\dd t\right)^{\beta}
\]
should be of the same order. Then we just need to bound the second quantity with lemma \ref{lem:15} and a law of large number.\\
For any $k\in [\![0,2m+3]\!]$ the random variables $(S^m_{(2m+3)i+k+1}-S^m_{(2m+3)i+k})_{i\geq 1}$ are iid (the definition of $\gamma^m_{\omega}(x)$ depends on a box f size $m$ around $x$ and traps span over 2 vertices that's why we cannot consider the sequence $(S^m_{i+1}-S^m_i)_{i\geq 1}$). This means that there is a positive constant $C_0$ that can be infinite such that $\E_{\Prob_0}\left(S^m_{2m+3}-S^m_{2m+2}\right)=C_0$ and
\[
\frac{1}{n}S^m_{n}\rightarrow C_0 \ \Prob_0 \text{ a.s and therefore } \Qprob_0 \text{ a.s.}
\]
For any $x\in \Z^d$ there is at most one integer $i$ such that $\left(\int\limits_{T^m_{i}}^{T^m_{i+1}}\frac{\gamma^m_{\omega}(x)}{\gamma^{\omega}(x)}1_{X^m_t=x}\dd t\right)$ is non-zero and therefore:
\[
S^m_n = \sum\limits_{\{x,y\}\in \mathcal{T}}\left(\int\limits_{0}^{T^m_{n+1}}\frac{\gamma^m_{\omega}(x)}{\gamma^{\omega}(x)}1_{X^m_t=x}\dd t\right)^{\beta}+\left(\int\limits_{0}^{T^m_{n+1}}\frac{\gamma^m_{\omega}(y)}{\gamma^{\omega}(y)}1_{X^m_t=y}\dd t\right)^{\beta}.
\]
By lemma \ref{lem:8} there is a finite constant $D^m$ such that $\frac{1}{n}T_n^m\rightarrow D^m$ $\Prob_0$ and $\Qprob_0$ almost surely. We get:
\[
\frac{1}{n}\sum\limits_{\{x,y\}\in \mathcal{T}}\left(\int\limits_{0}^{D^m n}\frac{\gamma^m_{\omega}(x)}{\gamma^{\omega}(x)}1_{X^m_t=x}\dd t\right)^{\beta}+\left(\int\limits_{0}^{D^m n}\frac{\gamma^m_{\omega}(y)}{\gamma^{\omega}(y)}1_{X^m_t=y}\dd t\right)^{\beta} \rightarrow C_0 \ \Qprob_0 \text{ a.s.}
\]
Therefore,
\[
\liminf \frac{1}{n} \E_{\Qprob_0}\left(\sum\limits_{\{x,y\}\in \mathcal{T}}\left(\int\limits_{0}^{D^m n}\frac{\gamma^m_{\omega}(x)}{\gamma^{\omega}(x)}1_{X^m_t=x}\dd t\right)^{\beta}+\left(\int\limits_{0}^{D^m n}\frac{\gamma^m_{\omega}(y)}{\gamma^{\omega}(y)}1_{X^m_t=y}\dd t\right)^{\beta}\right) \geq C_0.
\]
Since $\beta\leq 1$ we have:
\[
\begin{aligned}
&\frac{1}{n} \E_{\Qprob_0}\left(\sum\limits_{\{x,y\}\in \mathcal{T}}\left(\int\limits_{0}^{D^m n}\frac{\gamma^m_{\omega}(x)}{\gamma^{\omega}(x)}1_{X^m_t=x}\dd t\right)^{\beta}+\left(\int\limits_{0}^{D^m n}\frac{\gamma^m_{\omega}(y)}{\gamma^{\omega}(y)}1_{X^m_t=y}\dd t\right)^{\beta}\right)\\
\leq& \frac{1}{n} \sum\limits_{i=0}^{\lfloor D^m n \rfloor}\E_{\Qprob_0}\left(\sum\limits_{\{x,y\}\in \mathcal{T}}\left(\int\limits_{i}^{i+1}\frac{\gamma^m_{\omega}(x)}{\gamma^{\omega}(x)}1_{X^m_t=x}\dd t\right)^{\beta}+\left(\int\limits_{i}^{i+1}\frac{\gamma^m_{\omega}(y)}{\gamma^{\omega}(y)}1_{X^m_t=y}\dd t\right)^{\beta}\right)\\
=& \frac{\lfloor D^m n \rfloor +1}{n} \E_{\Qprob_0}\left(\sum\limits_{\{x,y\}\in \mathcal{T}}\left(\int\limits_{0}^{1}\frac{\gamma^m_{\omega}(x)}{\gamma^{\omega}(x)}1_{X^m_t=x}\dd t\right)^{\beta}+\left(\int\limits_{0}^{1}\frac{\gamma^m_{\omega}(y)}{\gamma^{\omega}(y)}1_{X^m_t=y}\dd t\right)^{\beta}\right)\\
&<\infty \text{ by lemma \ref{lem:15}}.
\end{aligned}
\]
So $C_0$ is finite. \\
Now we want to get a bound on $Y$ from a bound on $X^m$. For any trap $\{x,y\}\in\mathcal{T}$ let $N_{\{x,y\}}$ be the number of times the trap $\{x,y\}$ is entered. Let $\mathcal{T}^{\omega,n}$ be the subset of $\mathcal{T}^{\omega}$ defined by:
\[
\mathcal{T}^{\omega,n}:=\left\{ \{x,y\}\in\mathcal{T}^{\omega},\ Y_{\tau_1}.e_1\leq x.e_1\leq Y_{\tau_1}.e_1+n \text{ and } Y_{\tau_1}.e_1\leq y.e_1\leq Y_{\tau_1}.e_1+n \right\}.
\] 
We chose a partially-forgotten path $\sigma$ and we look at the law of the total time the walk $X$ spends in a trap $\{x,y\}\in\mathcal{T}^{\omega}$ knowing $Y_{\tau_1}$ and $\tilde{Y}=\sigma$, where $\tilde{Y}$ is the partially forgotten walk. We now have two sources of randomness: the number of back and forth the walk does every time it visits a trap and the time the continuous speed-walk $X^m$ spends for every step. \\
Knowing the partially-forgotten walk, $N_{\{x,y\}}$ is deterministic. Let $t^j_{\{x,y\}}$ be the $j^{\text{th}}$ time the walk $Y$ enters the trap $\{x,y\}$ and $\tilde{t}^j_{\{x,y\}}$ be the $j^{\text{th}}$ time the walk $Y$ exits the trap $\{x,y\}$. We define $H^j_{\{x,y\}}$ by $H^j_{\{x,y\}}:=\left\lfloor\frac{\tilde{t}^j_{\{x,y\}}-t^j_{\{x,y\}}}{2}\right\rfloor$, the number of back and forths in the trap $\{x,y\}$ during the $j^{\text{th}}$ visit to the trap. For any integer $n$ and for any trap $\{x,y\}\in\mathcal{T}^{\omega,n}$ we have that knowing the environment, $Y_{\tau_1}$ and the partially forgotten walk, $\left(H^j_{\{x,y\}}\right)_{j\in\N,\{x,y\}\in\mathcal{T}}$ is a sequence of independent geometric random variables of parameter $(1-\omega(x,y))(1-\omega(y,x))$. Finally, for every $x\in\tilde{\mathcal{T}}$, let $\ell^j_x$ be the number of time $x$ is visited between times $t^j_{\{x,y\}}$ and $\tilde{t}^j_{\{x,y\}}$. We define $\varepsilon^j_{x}$ by $\varepsilon^j_{x}:=\ell^j_x-H^j_{\{x,y\}} $. Knowing the partially forgotten walk, $\varepsilon^j_{x}$ is deterministic (it is equal to $0$ iff the walk enters and leaves the trap by $y$ during the $j^{\text{th}}$ visit) and $\varepsilon^j_{x}\in\{0,1\}$. We have:
\[
\int\limits_{0}^{\infty}\frac{\gamma^m_{\omega}(x)}{\gamma^{\omega}(x)}1_{X^m_t=x}\dd t=\sum\limits_{j=1}^{N_{\{x,y\}}}\sum\limits_{k=1}^{\varepsilon^j_{x}+H^j_{\{x,y\}}}\mathcal{E}^{k,j}_{m,x}\frac{\gamma^m_{\omega}(x)}{\gamma^{\omega}(x)},
\]
where the $(\mathcal{E}_{m,x}^{k,j})_{x\in\Z^d,k,j\in\N}$ are independent exponential random variables of parameter $\gamma^m_{\omega}(x)$, they correspond to the time the accelerated walk spends on each vertex. By technical lemma \ref{lem:23} (the proof of which is in the annex) we get that there exists a constant $C_1>0$ such that for any integer $n$ and any trap $\{x,y\}\in\mathcal{T}^{\omega,n}$:
\begin{equation} \label{eqn:14.2}
C_1 (N_{\{x,y\}})^{\beta} 
\leq \E_{\prob^{\omega}_0}\left(\left(\int\limits_{0}^{\infty}\frac{\gamma^m_{\omega}(x)}{\gamma^{\omega}(x)}1_{X^m_t=x}\dd t\right)^{\beta} + \left(\int\limits_{0}^{\infty}\frac{\gamma^m_{\omega}(y)}{\gamma^{\omega}(y)}1_{X^m_t=y}\dd t\right)^{\beta}|\tilde{Y},Y_{\tau_1}\right).
\end{equation}
Unfortunately, we cannot directly use this inequality to conclude because it does not behave nicely with the renewal times. Indeed if you know that a trap spans over two renewal blocks, it means that you cannot do any back and forth inside the trap and the previous inequality becomes false. Instead we will have to first consider traps in  $\mathcal{T}^{\omega,n}$. First, by definition of the renewal times, no trap in $\mathcal{T}^{\omega,n}$ can be visited before time $\tau_1$ or after time $\tau_{n+2}$ since $Y_{\tau_{n+2}}.e_1\geq Y_{\tau_1}.{e_1}+n+1$. Therefore:
\[
\begin{aligned}
&\sum\limits_{\{x,y\}\in\mathcal{T}^{\omega,n}}\left(\int\limits_{0}^{\infty}\frac{\gamma^m_{\omega}(x)}{\gamma^{\omega}(x)}1_{X^m_t=x}\dd t\right)^{\beta}+\left(\int\limits_{0}^{\infty}\frac{\gamma^m_{\omega}(y)}{\gamma^{\omega}(y)}1_{X^m_t=y}\dd t\right)^{\beta}\\
\leq &\sum\limits_{\{x,y\}\in\mathcal{T}^{\omega}}\left(\int\limits_{T^m_1}^{T^m_{n+2}}\frac{\gamma^m_{\omega}(x)}{\gamma^{\omega}(x)}1_{X^m_t=x}\dd t\right)^{\beta}+\left(\int\limits_{T^m_1}^{T^m_{n+2}}\frac{\gamma^m_{\omega}(y)}{\gamma^{\omega}(y)}1_{X^m_t=y}\dd t\right)^{\beta}
\end{aligned}
\]
Therefore we get:
\[
\frac{1}{n+1}\E_{\Prob_0}\left(\sum\limits_{\{x,y\}\in\mathcal{T}^{\omega,n}}\left(\int\limits_{0}^{\infty}\frac{\gamma^m_{\omega}(x)}{\gamma^{\omega}(x)}1_{X^m_t=x}\dd t\right)^{\beta}+\left(\int\limits_{0}^{\infty}\frac{\gamma^m_{\omega}(y)}{\gamma^{\omega}(y)}1_{X^m_t=y}\dd t\right)^{\beta}\right)\leq C_0<\infty
\]
This in turns gives:
\[
\frac{1}{n+1}\E_{\Prob_0}\left(\sum\limits_{\{x,y\}\in\mathcal{T}^{\omega,n}}\left(N_{\{x,y\}}\right)^{\beta}\right)\leq \frac{C_0}{C_1}<\infty
\]
Now let $C_{2}:=\E_{\Prob_0}\left(\sum\limits_{\{x,y\}\in\mathcal{T}}\#\{j\in [\![\tau_i,\tau_{i+1}-1 ]\!],Y_j\in\{x,y\}\text{ and } Y_{j+1}\not\in\{x,y\} \}^{\beta} \right) $ be the quantity we want to bound. By the law of large number, we have that $\Prob_0$ a.s and therefore $\Qprob_0$ a.s:
\[
\frac{1}{n}\sum\limits_{i=1}^n\sum\limits_{\{x,y\}\in\mathcal{T}}\#\{j\in [\![\tau_i,\tau_{i+1}-1 ]\!],Y_j\in\{x,y\}\text{ and } Y_{j+1}\not\in\{x,y\} \}^{\beta}\rightarrow C_2
\]
Now, as a consequence of lemma \ref{lem:9} and the law of large number, there exists a finite constant $D>0$ such that $\Prob_0$ a.s and therefore $\Qprob_0$ a.s, $\frac{1}{n}Y_{\tau_n}.e_1\rightarrow D$. Furthermore, a trap spans over at most two renewal blocks so for any trap $\{x,y\}$:
\[
\sum\limits_{i\geq 1} \#\{j\in [\![\tau_i,\tau_{i+1}-1 ]\!],Y_j\in\{x,y\}\text{ and } Y_{j+1}\not\in\{x,y\} \}^{\beta} \leq 2 (N_{\{x,y\}})^{\beta}.
\]
As a consequence, $\Prob_0$ a.s:
\[
\liminf \frac{1}{n}\frac{1}{n+1}\E_{\Prob_0}\left(\sum\limits_{\{x,y\}\in\mathcal{T}^{\omega,Dn}}\left(N_{\{x,y\}}\right)^{\beta}\right) \geq \frac{C_2}{2}.
\]
Finally we get:
\[
\frac{C_2}{2} \leq D\frac{C_0}{C_1}
\]
so $C_2$ is finite.
\end{proof}

The next lemma is just a variation of the previous one, with the difference that the sum has a deterministic number of terms instead of a random one which makes it simpler to use.
\begin{lem}\label{lem:21}
For any $j\in [\![ 1,d ] \!]$ let $(x_i^j,y_i^j)$ be the $i^{\text{th}}$ trap in the direction $j$ the walk encounters after $\tau_2$. Let $N_i^j$ be the number of times the walk enters this trap. \\
If $\kappa \leq 1$, for any $\beta\in [\kappa,\frac{\kappa+\kappa^{\prime}}{2})$ with $\beta\leq 1$ there is a constant $C$ such that for any $j\in [\![ 1,d  ]\!]$:
\[
\E_{\Prob_0}\left(\sum\limits_{i=1}^{n}(N_i^j)^{\beta}\right) \leq C n.
\] 
If $\kappa=1$ there exists a positive concave function $\phi$ defined on $[0,\infty)$ such that $\phi(t)$ goes to infinity when $t$ goes to infinity. And such that if $\Phi(t)=\int\limits_{x=0}^t \phi(x) \dd x$ then there exists a constant $C$ such that for any $n\in\N$:
\[
\E_{\Prob_0}\left(\sum\limits_{i=1}^{n}\Phi(N_i^j)\right) \leq C n.
\] 
Those results are also true if $(x_i^j,y_i^j)$ is the $i^{\text{th}}$ trap in the direction $j$ the walk encounters after $\tau_2$ such that $x_i.e_1,y_i.e_1\geq Y_{\tau_2}.e_1$ .
\end{lem}
\begin{proof}
Let $p>0$ be the probability, for ${\Prob_0}$, that there is at least one trap in the direction $j$ between times $\tau_2$ and $\tau_{3}-1$. Let $\mathcal{T}_j$ be the set of traps in the direction $j$. Now let the sequence $(n_i)$ be defined by:
\[
\begin{aligned}
n_0=&1,\\
n_{i+1}=& \min \{k> n_i, \exists \{x,y\}\in \mathcal{T}_j, \exists n\in[\![\tau_k,\tau_{k+1}-1]\!], Y_n\in\{x,y\}    \}.
\end{aligned}
\]
Now, if $\kappa \leq 1$, let $Z^j_i=\sum\limits_{\{x,y\}\in\mathcal{T}_j}\#\{m\in [\![\tau_{n_i},\tau_{n_i+1}-1 ]\!],Y_m\in\{x,y\}\text{ and } Y_{m+1}\not\in\{x,y\} \}^{\beta}$. The $(Z^j_i)_{i\geq 1}$ are clearly identically distributed and we have:
\[
\E_{\Prob_0}(Z^j_i) = \frac{1}{p}\E_{\Prob_0}\left(\sum\limits_{\{x,y\}\in\mathcal{T}_j}\#\{m\in [\![\tau_{i},\tau_{i+1}-1 ]\!],Y_m\in\{x,y\}\text{ and } Y_{m+1}\not\in\{x,y\} \}^{\beta}\right).
\]
So let $C_j=\E_{\Prob_0}(Z^j_i)$ which is finite by lemma \ref{lem:14}. We clearly have:
\[
\sum\limits_{i=1}^m (N_i^j)^{\beta}\leq \sum\limits_{i=1}^{2m} Z^j_i.
\]
The sum has to go up to $2m$ because in the second sum some traps can appear twice if they are in between two renewal slabs. Indeed, in this case they can be visited before and after the renewal time (if they are in the direction $e_1$). We now have:
\[
\E_{\Prob_0}\left(\sum\limits_{i=1}^m (N_i^j)^{\beta}\right) \leq 2C_j m.
\]
Similarly, if $\{\overline{x}_i,\overline{y}_i\}$ is the $i^{\text{th}}$ trap in the direction $j$ the walk encounters after $\tau_2$ such that $x_i.e_1,y_i.e_1\geq Y_{\tau_2}.e_1$ and $\overline{N}_i^j$ the number of times the walk enters this trap then we have:
\[
\sum\limits_{i=1}^m (\overline{N}_i^j)^{\beta} \leq \sum\limits_{i=1}^{2m+1} Z_i.
\]
If $\kappa=1$, by lemma \ref{lem:14},
\[
\E_{\Prob_0}\left(\sum\limits_{\{x,y\}\in\mathcal{T}_j}\#\{m\in [\![\tau_{2},\tau_{3}-1 ]\!],Y_m\in\{x,y\}\text{ and } Y_{m+1}\not\in\{x,y\} \}^{\beta}\right)<\infty.
\]
Therefore, by forthcoming technical lemma \ref{lem:22} there exists a positive, concave function $\phi$ defined on $[0,\infty)$ such that $\phi(t)$ goes to infinity when $t$ goes to infinity and such that, if $\Phi(t):=\int\limits_{x=0}^t \phi(x) \dd x$ then:
\[
\E_{\Prob_0}\left(\Phi\left(2\sum\limits_{\{x,y\}\in\mathcal{T}_j}\#\{m\in [\![\tau_{2},\tau_{3}-1 ]\!],Y_m\in\{x,y\}\text{ and } Y_{m+1}\not\in\{x,y\} \}^{\beta}\right)\right)<\infty,
\]
where $\Phi(t):=\int\limits_{x=0}^t \phi(x) \dd x$. We have that $x\rightarrow \frac{\Phi(x)}{x}$ is increasing and therefore, by writing $g(x)=\frac{\Phi(x)}{x}$, for any non-negative sequence $(a_i)_{1\leq i \leq n}$:
\[
\begin{aligned}
\sum\limits_{1\leq i\leq n} \Phi(a_i)
&= \sum\limits_{1\leq i\leq n} a_i g(a_i)\\
&\leq \sum\limits_{1\leq i\leq n} a_i g\left(\sum\limits_{1\leq j\leq n} a_j\right)\\
&= \left(\sum\limits_{1\leq i\leq n} a_i\right) g\left(\sum\limits_{1\leq i\leq n} a_i\right) \\
&= \Phi\left(\sum\limits_{1\leq i\leq n} a_i\right).
\end{aligned}
\]
So we get:
\[
\begin{aligned}
&\E_{\Prob_0}\left(\sum\limits_{\{x,y\}\in\mathcal{T}_j}\Phi\left(2\#\{m\in [\![\tau_{2},\tau_{3}-1 ]\!],Y_m\in\{x,y\}\text{ and } Y_{m+1}\not\in\{x,y\} \}^{\beta}\right)\right)\\
\leq & \E_{\Prob_0}\left(\Phi\left(2\sum\limits_{\{x,y\}\in\mathcal{T}_j}\#\{m\in [\![\tau_{2},\tau_{3}-1 ]\!],Y_m\in\{x,y\}\text{ and } Y_{m+1}\not\in\{x,y\} \}^{\beta}\right)\right)<\infty.
\end{aligned}
\]
Let $Z^j_i:=\sum\limits_{\{x,y\}\in\mathcal{T}_j}\Phi\left(\#\{m\in [\![\tau_{n_i},\tau_{n_i+1}-1 ]\!],Y_m\in\{x,y\}\text{ and } Y_{m+1}\not\in\{x,y\} \}^{\beta}\right)$. The $(Z^j_i)_{i\geq 1}$ are clearly identically distributed and we have:
\[
\E_{\Prob_0}(Z_i) = \frac{1}{p}\E_{\Prob_0}\left(2\sum\limits_{\{x,y\}\in\mathcal{T}_j}\Phi\left(\#\{m\in [\![\tau_{2},\tau_{3}-1 ]\!],Y_m\in\{x,y\}\text{ and } Y_{m+1}\not\in\{x,y\} \}^{\beta}\right)\right)<\infty.
\]
So let $C_j=\E_{\Prob_0}(Z^j_i)$, which is finite. We clearly have:
\[
\sum\limits_{i=1}^m \Phi(N_i^j) \leq \sum\limits_{i=1}^{2m}  Z_i.
\]
Once again, the sum has to go up to $2m$ because in the second sum some traps can appear twice if they are in between two renewal slabs. Indeed, in this case they can be visited before and after the renewal time (if they are in the direction $e_1$).
so:
\[
\E_{\Prob_0}\left(\sum\limits_{i=1}^m \Phi(N_i^j)\right) \leq 2C_j m.
\]
Similarly, if $\{\overline{x}_i,\overline{y}_i\}$ is the $i^{\text{th}}$ trap in the direction $j$ the walk encounters after $\tau_2$ such that $x_i.e_1,y_i.e_1\geq Y_{\tau_2}.e_1$ and $\overline{N}_i^j$ the number of times the walk enters this trap then we have:
\[
\sum\limits_{i=1}^m \Phi(\overline{N}_i^j) \leq \sum\limits_{i=1}^{2m+1} Z^j_i.
\]
and we get the result we want.
\end{proof}

The following lemma gives us some independence between the strength of a trap and the number of times the walk enters this trap.
\begin{lem}\label{lem:13}
Let $j\in [\![ 1,d ]\!]$ be an integer that represents the direction of the traps we will consider. Let $\{x^j_i,y^j_i\}$ be the $i^{\text{th}}$ trap in the direction $j$ (ie $x^j_i-y^j_i\in \{e_j,-e_{j}\}$) to be visited after time $\tau_2$ and such that $x^j_i.e_1 \geq Y_{\tau_2}.e_1$ and $y^j_i.e_1 \geq Y_{\tau_2}.e_1$. Now let $s^j_i$ be the strength of the trap. Let $N_i^j$ be the number of times the trap $\{x_i^j,y_i^j\}$ is exited. Let $\kappa_j=2\sum\limits_{i=1}^{2d} \alpha_i-\alpha_j-\alpha_{j+d}$. For any $\gamma\in[0,1]$, there exists a constant $C$ that does not depend on $i$ such that:
\[
\forall A\geq 2,\ \E_{\Prob_0}\left((N^j_i)^{\gamma}1_{s^j_i\geq A}\right)\leq\frac{C}{A^{\kappa_j}}\E_{\Prob_0}((N^j_i)^{\gamma}).
\]
We also have that for any positive concave function $\phi$ such that $\phi (0)=1$ with $\Phi(t)=\int\limits_{x=0}^t \phi(x) \dd x$ we get:
\[
\forall A\geq 2,\ \E_{\Prob_0}\left(\Phi(N^j_i)1_{s^j_i\geq A}\right)\leq\frac{C}{A^{\kappa_j}}\E_{\Prob_0}(\Phi(N^j_i)).
\]  
\end{lem}
\begin{proof}
First if $H$ is a geometric random variable of parameter $p$ then for any $\gamma \in [0,1]$ we have the following three inequalities:
\begin{equation}
\E((1+H)^{\gamma})\geq 1=p^{\gamma}\frac{1}{p^{\gamma}},
\label{eqn: 13.1}
\end{equation}
\begin{equation}
\E((1+H)^{\gamma})\geq\Prob\left(Z\geq \frac{1}{p}\right)\frac{1}{p^{\gamma}}\geq (1-p)^{\frac{1}{p}-1}\frac{1}{p^{\gamma}},
\label{eqn: 13.2}
\end{equation}
\begin{equation}
\E((1+H)^{\gamma})\leq\E((1+H))^{\gamma}=\frac{1}{p^{\gamma}}.
\label{eqn: 13.3}
\end{equation}
Inequalities \ref{eqn: 13.1} and \ref{eqn: 13.2} give us that there is a constant $C_{\gamma}$ such that $\E((1+H)^{\gamma})\geq C_{\gamma}\frac{1}{p^{\gamma}}$, inequality \ref{eqn: 13.1} gives us the result for $p\geq\frac{1}{2}$ and since $(1-p)^{\frac{1}{p}-1}$ converges to $\exp(-1)$ when $p$ goes to $0$, inequality \ref{eqn: 13.2} gives us the result for $p\leq\frac{1}{2}$.\\
By lemma \ref{lem:24} we get that there is a constant $C_{\phi}$ such that:
\begin{equation}
\frac{1}{2} \frac{1}{p}\phi\left(\frac{1}{p}\right) 
\leq \E(\Phi(1+H)) 
\leq C_{\phi} \frac{1}{p}\phi\left(\frac{1}{p}\right).
\label{eqn: 13.ineq}
\end{equation}
Let $t\in \N$ be an integer. In the following we will call renewal hyperplan the set of vertices $\{x,x.e_1=Y_t.e_1\}$. We look at the $n^{\text{th}}$ time, after time $t$, that the walk encounters a vertex that touches a trap $\{x,y\}$ in the direction $j$ that has never been visited before and such that $x.e_1,y.e_1\geq Y_t.e_1$. We want to show that the strength of the trap is basically independent from the number of times the walk leaves the trap and from the random variable $1_{\tau_2=t}$. Let ${x,y}$ be the corresponding trap with $x$ being the first vertex visited.\\
Now we look at the trap $\{x,y\}$. Let $i$ be such that $y=x+e_i$, we will write $\alpha_x:=\alpha_i$, $\alpha_y:=\alpha_{i+d}$ and $\overline{\alpha}:=\sum\limits_{k=1}^{2d}\alpha_k$. The density probability (for $\Prob^{(\alpha)}$) for the transition probabilities $\omega(x,y)$ and $\omega(y,x)$, knowing all the transition probabilities $(\omega(z_1,z_2))_{z_1\in\Z^d \backslash \{x,y\}}$, the renormalized transition probabilities $(\frac{\omega(x,z)}{1-\omega(x,y)})_{z\not = y},(\frac{\omega(y,z)}{1-\omega(y,x)})_{z\not = x}$ and that $\{x,y\}$ is a trap is:
\[
C\omega(x,y)^{\alpha_x-1}(1-\omega(x,y))^{\overline{\alpha}-\alpha_x-1}\omega(y,x)^{\alpha_y-1}(1-\omega(y,x))^{\overline{\alpha}-\alpha_y-1}1_{\omega(x,y)+\omega(y,x)\geq \frac{3}{2}}.
\]
Now we make the change of variables:
\[
1-\omega(y,x)=r(1-k),\ 1-\omega(x,y)=r(1+k),
\]
which gives a density probability of:
\[
2rCr^{\kappa_j-2}(1-k)^{\overline{\alpha}-\alpha_y-1}(1+k)^{\overline{\alpha}-\alpha_x-1}(1-r(1+k))^{\alpha_x-1}(1-r(1-k))^{\alpha_y-1}1_{r\leq\frac{1}{4}}\dd r \dd k.
\]
Let $h(r,k)$ be defined by:
\[
h(r,k)=(1-r(1+k))^{\alpha_x-1}(1-r(1-k))^{\alpha_y-1}.
\]
For $0\leq r \leq \frac{1}{4}$ and $-1\leq k \leq 1$ we have:
\[
\log(h(r,k))\leq |\alpha_x-1|\left|\log\left(\frac{1}{2}\right)\right|+|\alpha_y-1|\left|\log\left(\frac{1}{2}\right)\right|\leq (\alpha_x+\alpha_y+2)\log(2).
\]
So for $0\leq r \leq \frac{1}{4}$ and $-1\leq k \leq 1$ we have:
\[
2^{-(\alpha_x+\alpha_y+2)}\leq h(r,k) \leq 2^{\alpha_x+\alpha_y+2}.
\]
Now the density probability is:
\[
2Ch(r,k)r^{\kappa_j-1}(1-k)^{\overline{\alpha}-\alpha_y-1}(1+k)^{\overline{\alpha}-\alpha_x-1}1_{r\leq\frac{1}{4}}\dd r\dd k.
\]
Now we look at a specific environment $\omega$ and an edge $\{x^{\prime},y^{\prime}\}$ in that environment. To simplify the notation we will write $\varepsilon_{x^{\prime}}=1-\omega(x^{\prime},y^{\prime})$ and $\varepsilon_{y^{\prime}}=1-\omega(y^{\prime},x^{\prime})$. When the walk leaves the trap there are three possibilities: \\
-the walk goes to infinity before going back to the trap or the renewal hyperplan\\
-the walk goes to the renewal hyperplan before it goes back to the trap (this does not necessarily mean that the walk will go back to the trap after going to the renewal hyperplan)\\
-the walk goes back to the trap before it goes to the renewal hyperplan (this does not necessarily mean that the walk will eventually go to the renewal hyperplan).\\
\\
If the walk is in $x^{\prime}$ let $\beta^{\infty}_{x^{\prime}}$ be the probability, knowing that the next step isn't crossing the trap, that the walk goes to infinity without going to the renewal hyperplan or the trap. Similarly, let $\beta^{0}_{x^{\prime}}$ be the probability, knowing that the next step isn't crossing the trap, that the walk goes to the renewal hyperplan before it goes back to the trap (this does not mean that the walk necessarily goes back to the trap). We will also define $\beta_{x^{\prime}}$ by $\beta_{x^{\prime}}:=\beta^{\infty}_{x^{\prime}}+\beta^{0}_{x^{\prime}}$. Similarly we will define $\beta_{y^{\prime}},\beta^{\infty}_{y^{\prime}},\beta^{0}_{y^{\prime}}$. \\
\\
Now, if the walk is in $x^{\prime}$, the probability that when the walk leaves the trap it either never comes back to the trap or goes to the renewal hyperplan before it goes back to the trap is:
\[
\frac{\varepsilon_{x^{\prime}}}{\varepsilon_{x^{\prime}}+\varepsilon_{y^{\prime}}-\varepsilon_{x^{\prime}}\varepsilon_{y^{\prime}}}\beta_{x^{\prime}}
+\frac{\varepsilon_{y^{\prime}}(1-\varepsilon_{x^{\prime}})}{\varepsilon_{x^{\prime}}+\varepsilon_{y^{\prime}}-\varepsilon_{x^{\prime}}\varepsilon_{y^{\prime}}}\beta_{y^{\prime}}
=\frac{\varepsilon_{x^{\prime}}\beta_{x^{\prime}}+\varepsilon_{y^{\prime}}(1-\varepsilon_{x^{\prime}})\beta_{y^{\prime}}}{\varepsilon_{x^{\prime}}+\varepsilon_{y^{\prime}}-\varepsilon_{x^{\prime}}\varepsilon_{y^{\prime}}}.
\]
Similarly, if the walk is in $y^{\prime}$, this probability is:
\[
\frac{\varepsilon_{x^{\prime}}(1-\varepsilon_{y^{\prime}})\beta_{x^{\prime}}+\varepsilon_{y^{\prime}}\beta_{y^{\prime}}}{\varepsilon_{x^{\prime}}+\varepsilon_{y^{\prime}}-\varepsilon_{x^{\prime}}\varepsilon_{y^{\prime}}}.
\]
Now we want to show that that both these quantities are almost equal to:
\[
\frac{\varepsilon_{x^{\prime}}\beta_{x^{\prime}}+\varepsilon_{y^{\prime}}\beta_{y^{\prime}}}{\varepsilon_{x^{\prime}}+\varepsilon_{y^{\prime}}}.
\]
We will only show it for the first quantity, the proof is the same for the second one. We recall that $\varepsilon_{x^{\prime}},\varepsilon_{y^{\prime}}\leq\frac{1}{2}$, therefore:
\[
0\leq\varepsilon_{x^{\prime}}\varepsilon_{y^{\prime}}\leq\frac{1}{2}(\varepsilon_{x^{\prime}}+\varepsilon_{y^{\prime}})
\]
and
\[
0\leq \varepsilon_{x^{\prime}}\varepsilon_{y^{\prime}}\beta_{y^{\prime}}\leq\frac{1}{2}(\varepsilon_{x^{\prime}}\beta_{x^{\prime}}+\varepsilon_{y^{\prime}}\beta_{y^{\prime}}).
\]
So we get:
\[
\frac{1}{2}\frac{\varepsilon_{x^{\prime}}\beta_{x^{\prime}}+\varepsilon_{y^{\prime}}\beta_{y^{\prime}}}{\varepsilon_{x^{\prime}}+\varepsilon_{y^{\prime}}}
\leq \frac{\varepsilon_{x^{\prime}}\beta_{x^{\prime}}+\varepsilon_{y^{\prime}}(1-\varepsilon_{x^{\prime}})\beta_{y^{\prime}}}{\varepsilon_{x^{\prime}}+\varepsilon_{y^{\prime}}-\varepsilon_{x^{\prime}}\varepsilon_{y^{\prime}}}
\leq 2\frac{\varepsilon_{x^{\prime}}\beta_{x^{\prime}}+\varepsilon_{y^{\prime}}\beta_{y^{\prime}}}{\varepsilon_{x^{\prime}}+\varepsilon_{y^{\prime}}}.
\]
Similarly, if the walk is in ${x^{\prime}}$, the probability that the walk goes to infinity knowing that the walk either goes to infinity or the renewal hyperplan before coming to the trap is:
\[
\frac{\varepsilon_{x^{\prime}}\beta^{\infty}_{x^{\prime}}+\varepsilon_{y^{\prime}}(1-\varepsilon_{x^{\prime}})\beta^{\infty}_{y^{\prime}}}{\varepsilon_{x^{\prime}}+\varepsilon_{y^{\prime}}-\varepsilon_{x^{\prime}}\varepsilon_{y^{\prime}}}\frac{\varepsilon_{x^{\prime}}+\varepsilon_{yx^{\prime}}-\varepsilon_{x^{\prime}}\varepsilon_{y^{\prime}}}{\varepsilon_{x^{\prime}}\beta_{x^{\prime}}+\varepsilon_{y^{\prime}}(1-\varepsilon_{x^{\prime}})\beta_{y^{\prime}}}
=\frac{\varepsilon_{x^{\prime}}\beta^{\infty}_{x^{\prime}}+\varepsilon_{y^{\prime}}(1-\varepsilon_{x^{\prime}})\beta^{\infty}_{y^{\prime}}}{\varepsilon_{x^{\prime}}\beta_{x^{\prime}}+\varepsilon_{y^{\prime}}(1-\varepsilon_{x^{\prime}})\beta_{y^{\prime}}}.
\]
And if it is in $y^{\prime}$ this probability is:
\[
\frac{\varepsilon_{x^{\prime}}(1-\varepsilon_{y^{\prime}})\beta^{\infty}_{x^{\prime}}+\varepsilon_{y^{\prime}}\beta^{\infty}_{y^{\prime}}}{\varepsilon_{x^{\prime}}(1-\varepsilon_{y^{\prime}})\beta_{x^{\prime}}+\varepsilon_{y^{\prime}}\beta_{y^{\prime}}}.
\]
We want to show that both these probabilities are almost equal to $\frac{\varepsilon_{x^{\prime}}\beta^{\infty}_{x^{\prime}}+\varepsilon_{y^{\prime}}\beta^{\infty}_{y^{\prime}}}{\varepsilon_{x^{\prime}}\beta_{x^{\prime}}+\varepsilon_{y^{\prime}}\beta_{y^{\prime}}}$. We will only show it for the first one:
\[
\begin{aligned}
\frac{\varepsilon_{x^{\prime}}\beta^{\infty}_{x^{\prime}}+\varepsilon_{y^{\prime}}(1-\varepsilon_{x^{\prime}})\beta^{\infty}_{y^{\prime}}}{\varepsilon_{x^{\prime}}\beta_{x^{\prime}}+\varepsilon_{y^{\prime}}(1-\varepsilon_{x^{\prime}})\beta_{y^{\prime}}}
\leq &\frac{\varepsilon_{x^{\prime}}\beta^{\infty}_{x^{\prime}}+\varepsilon_{y^{\prime}}\beta^{\infty}_{y^{\prime}}}{\varepsilon_{x^{\prime}}\beta_{x^{\prime}}+\varepsilon_{y^{\prime}}(1-\varepsilon_{x^{\prime}})\beta_{y^{\prime}}}\\
\leq & \frac{1}{(1-\varepsilon_x)}\frac{\varepsilon_{x^{\prime}}\beta^{\infty}_{x^{\prime}}+\varepsilon_{y^{\prime}}\beta^{\infty}_{y^{\prime}}}{\varepsilon_{x^{\prime}}\beta_{x^{\prime}}+\varepsilon_{y^{\prime}}\beta_{y^{\prime}}}\\
\leq & 2 \frac{\varepsilon_{x^{\prime}}\beta^{\infty}_{x^{\prime}}+\varepsilon_{y^{\prime}}\beta^{\infty}_{y^{\prime}}}{\varepsilon_{x^{\prime}}\beta_{x^{\prime}}+\varepsilon_{y^{\prime}}\beta_{y^{\prime}}}.
\end{aligned}
\] 
And we also get, the same way:
\[
\frac{\varepsilon_{x^{\prime}}\beta^{\infty}_{x^{\prime}}+\varepsilon_{y^{\prime}}(1-\varepsilon_{x^{\prime}})\beta^{\infty}_{y^{\prime}}}{\varepsilon_{x^{\prime}}\beta_{x^{\prime}}+\varepsilon_{y^{\prime}}(1-\varepsilon_{x^{\prime}})\beta_{y^{\prime}}}
\geq \frac{1}{2}\frac{\varepsilon_{x^{\prime}}\beta^{\infty}_{x^{\prime}}+\varepsilon_{y^{\prime}}\beta^{\infty}_{y^{\prime}}}{\varepsilon_{x^{\prime}}\beta_{x^{\prime}}+\varepsilon_{y^{\prime}}\beta_{y^{\prime}}}.
\]
Now we get back to the trap $\{x,y\}$. Let $N$ be the number of times the walks leaves the trap $\{x,y\}$ before going to the renewal hyperplan (so if the walk never goes to the renewal hyperplan, N is just the number of times the walk leaves the trap $\{x,y\}$). We get that knowing $\varepsilon_x,\varepsilon_y$ and $N$, the probability (for $\prob_0^{\omega}$) that the walk never goes to the renewal hyperplan is between $\frac{1}{2}\frac{\varepsilon_x\beta^{\infty}_x+\varepsilon_y\beta^{\infty}_y}{\varepsilon_x\beta_x+\varepsilon_y\beta_y}$ and $2\frac{\varepsilon_x\beta^{\infty}_x+\varepsilon_y\beta^{\infty}_y}{\varepsilon_x\beta_x+\varepsilon_y\beta_y}$. \\
We also have that there exists two geometric random variables $N^{-}$ and $N^{+}$ respectively of parameter $\frac{1}{2}\frac{\varepsilon_x\beta_x+\varepsilon_y\beta_y}{\varepsilon_x+\varepsilon_y}$ and $2\frac{\varepsilon_x\beta_x+\varepsilon_y\beta_y}{\varepsilon_x+\varepsilon_y}$ such that $\prob_0^{\omega}$ almost surely:
\[
1+N^{-}\leq N \leq 1+N^{+}.
\]
Therefore, by equations \ref{eqn: 13.1}, \ref{eqn: 13.2}, \ref{eqn: 13.3} and \ref{eqn: 13.ineq} there exists two positive constants $C_1$ and $C_2$ (that depend on $\gamma$ and $\Phi$) such that for $f$ equal to either $x\rightarrow x^{\gamma}$ or $\Phi$:
\begin{equation}
C_1f\left(\frac{\varepsilon_x+\varepsilon_y}{\varepsilon_x\beta_x+\varepsilon_y\beta_y}\right)
\leq \E_{\prob_0^{\omega}}(f(N))
\leq C_2f\left(\frac{\varepsilon_x+\varepsilon_y}{\varepsilon_x\beta_x+\varepsilon_y\beta_y}\right). \label{eqn:13.4}
\end{equation}
Now let $f$ be either $x\rightarrow x^{\gamma}$ or $\Phi$. We need to show that $N$ is almost independent from $1_{\tau_2=t}$. Let $t_{xy}$ be the first time the walk is in $x$ or $y$ and let $B$ be the event that ``$\tau_2$ can be equal to $t$" ie there exists $t^{\prime}<t$ ($t^{\prime}$ plays the role of $\tau_1$) such that: \\
- $\forall i<t^{\prime},\ X_i.e_1<X_{t^{\prime}}.e_1$, \\
- $\forall i \in [\![t^{\prime},t-1]\!],\ X_{t^{\prime}}.e_1\leq X_i.e_1 < X_{t}.e_1$, \\
- $\forall i \in [\![t,t_{xy}]\!],\ X_i.e_1 \geq X_{t}.e_1$, \\
- $\forall i \in [\![0,t^{\prime}-1]\!]\cup [\![t^{\prime}+1,t-1]\!], \left(\exists j <i,\ X_j.e_1\geq X_i.e_1\right) \text{ or } \left( \exists j\in [\![i+1,t-1]\!] X_j.e_1<X_i.e_1\right)$. \\
\\
We have that if $B$ isn't true then $\tau_2$ cannot be equal to $t$. If $B$ is true the $\tau_2=t$ iff the walk never crosses the renewal hyperplan after time $t_{xy}$. So, for any environment $\omega$:
\begin{equation}
\frac{1}{2}\frac{\varepsilon_x\beta^{\infty}_x+\varepsilon_y\beta^{\infty}_y}{\varepsilon_x\beta_x+\varepsilon_y\beta_y} \prob_0^{\omega}(B)
\leq \prob_0^{\omega}\left(\tau_2=t|N\right)
\leq 2\frac{\varepsilon_x\beta^{\infty}_x+\varepsilon_y\beta^{\infty}_y}{\varepsilon_x\beta_x+\varepsilon_y\beta_y} \prob_0^{\omega}(B) \label{eqn:13.5}
\end{equation}
To simplify notations we will write 
\[
\overline{h}(k):=\frac{(1+k)\beta^{\infty}_x+(1-k)\beta^{\infty}_y}{(1+k)\beta_x+(1-k)\beta_y}.
\]
We have (in the following, the constant $C$ will depend on the line):
\[
\begin{aligned}
&\E_{\Prob_0}\left(f(N)1_{\varepsilon_x+\varepsilon_y\leq\frac{1}{A}}1_{\tau_2=t}\right)\\
\leq & 2\E_{\Prob_0}\left(f(N)\frac{\varepsilon_x\beta^{\infty}_x+\varepsilon_y\beta^{\infty}_y}{\varepsilon_x\beta_x+\varepsilon_y\beta_y}1_{\varepsilon_x+\varepsilon_y\leq\frac{1}{A}}\right) \text{ by } \ref{eqn:13.5}\\
\leq& C \E_{\Prob_0}\left(f\left(\frac{\varepsilon_x+\varepsilon_y}{\varepsilon_x\beta_x+\varepsilon_y\beta_y}\right)1_{\varepsilon_x+\varepsilon_y\leq\frac{1}{A}}\frac{\varepsilon_x\beta^{\infty}_x+\varepsilon_y\beta^{\infty}_y}{\varepsilon_x\beta_x+\varepsilon_y\beta_y}\right) \text{ by } \ref{eqn:13.4}.
\end{aligned}
\]
Now we use the fact that the various $\beta$ only depend on the trajectory of the walk up to the time it encounters the $n^{\text{th}}$ trap in the direction $j$ after time $t$, the transition probabilities $(\omega(z_1,z_2))_{z_1\in\Z^d \backslash \{x,y\}}$, the renormalized transition probabilities $(\frac{\omega(x,z)}{1-\omega(x,y)})_{z\not = y},(\frac{\omega(y,z)}{1-\omega(y,x)})_{z\not = x}$ and that $\{x,y\}$ is a trap. But the law of $(\omega(x,y),\omega(x,y))$ is independent of this so we get:
\[
\begin{aligned} 
&\E_{\Prob_0}\left(f(N)1_{\varepsilon_x+\varepsilon_y\leq\frac{1}{A}}1_{\tau_2=t}\right)\\
\leq & C \E_{\Prob_0}\left(\int\limits_{r=0}^{\frac{1}{2A}}\int\limits_{k=-1}^{1}f\left(\frac{r}{r(1+k)\beta_x+r(1-k)\beta_y}\right)2Ch(r,k)r^{\kappa_j-1}(1-k)^{\alpha_y}(1+k)^{\alpha_x}\overline{h}(k)\dd k\dd r\right) \\
\leq& C \E_{\Prob_0}\left(\int\limits_{r=0}^{\frac{1}{2A}}r^{\kappa_j-1}\dd r\int\limits_{k=-1}^{1}f\left(\frac{1}{(1+k)\beta_x+(1-k)\beta_y}\right)(1-k)^{\alpha_y}(1+k)^{\alpha_x}\overline{h}(k)\dd k\right)\\
=& C \left(\frac{2}{A}\right)^{\kappa_j} \E_{\Prob_0}\left(\int\limits_{r=0}^{\frac{1}{4}}r^{\kappa_j-1}\dd r\int\limits_{k=-1}^{1}f\left(\frac{1}{(1+k)\beta_x+(1-k)\beta_y}\right)(1-k)^{\alpha_y}(1+k)^{\alpha_x}\overline{h}(k)\dd k\right)\\
\leq& \frac{C}{A^{\kappa_j}}\E_{\Prob_0}\left(\int\limits_{r=0}^{\frac{1}{4}}\int\limits_{k=-1}^{1}f\left(\frac{r}{r(1+k)\beta_x+r(1-k)\beta_y}\right)2Ch(r,k)r^{\kappa_j-1}(1-k)^{\alpha_y}(1+k)^{\alpha_x}\overline{h}(k)\dd k\dd r\right) \\
\leq& \frac{C}{A^{\kappa_j}} \E_{\Prob_0}\left(f(N)1_{\varepsilon_x+\varepsilon_y\leq\frac{1}{2}}\frac{\varepsilon_x\beta^{\infty}_x+\varepsilon_y\beta^{\infty}_y}{\varepsilon_x\beta_x+\varepsilon_y\beta_y}\right)\\
\leq& \frac{C}{A^{\kappa_j}}\E_{\Prob_0}\left(f(N)1_{\varepsilon_x+\varepsilon_y<\frac{1}{2}}1_{\tau_2=t}\right).
\end{aligned}
\]
Then, by summing on all $t$ we get the result. 
\end{proof}

\subsection{The time the walk spends in trap}
Now that we have some independence, we can start to look at the precise behaviour of the time spent in the traps. First we want to show that the number of times the walk enters a trap times the strength of said trap is a good approximation of the total time spent in this trap.
\begin{lem}\label{lem:17}
Let $j\in [\![ 1,d ]\!]$ be a direction. Now let $\{x^j_i,y^j_i\}$ be the $i^{\text{th}}$ trap in the direction $j$ entered after time $\tau_2$ and such that $x_i^j.e_1,y_i^j.e_1\geq Y_{\tau_2}.e_1$. Let $s^j_i$ be the strength of this trap, $N^j_i$ the number of times the walk enters this trap and $\ell^j_i=\#\{n,Y_n\in \{x^j_i,y^j_i\} \}$ the time spent in the trap. We have for any environment $\omega$, for any $A, B\geq 0$, for any integer $m$ and for any $C\in \R^+ \cup \{\infty\}$: 
\[
\prob_0^{\omega}\left(\sum\limits_{i=1}^n \ell^j_i1_{N^j_i\geq m}1_{s^j_i\leq C}\geq A \text{ and } \sum\limits N^j_i1_{N^j_i\geq m} s^j_i1_{s^j_i\leq C}\leq B\right)\leq \frac{5B}{A}.
\]
\end{lem}
\begin{proof}
Let $\omega$ be an environment, $(\tilde{Y}_i)_{i\in\N}$ be the partially forgotten walk on this environment. Let $p^j_i=\omega(x^j_i,y^j_i)\omega(y^j_i,x^j_i)$. Now the number of back and forths inside the trap $(x^j_i,y^j_i)$ during its $k^{\text{th}}$ visit is equal to $H^j_{i,k}$ where $H^j_{i,k}$ is a geometric random variable of parameter $p^j_i$. Knowing the partially-forgotten walk and $p^j_i$, the $H^j_{i,k}$ are independent and we get for any $j$:
\[
\begin{aligned}
\E_{\prob_0^{\omega}}\left(\sum\limits_{i=1}^{n}1_{N^j_i\geq m}1_{s^j_i\leq C}\sum\limits_{k=1}^{N^j_i}2H^j_{i,k}|\tilde{Y}\right)
=&\sum\limits_{i=1}^{n}1_{N^j_i\geq m}1_{s^j_i\leq C}\sum\limits_{k=1}^{N^j_i}2\frac{p^j_i}{1-p^j_i}\\
\leq & 2\sum\limits_{i=1}^{n} 1_{N^j_i\geq m}1_{s^j_i\leq C}N^j_i\frac{1}{1-p^j_i}.
\end{aligned}
\]
Now we use the fact that $\omega(x^j_i,y^j_i)\geq \frac{1}{2}$ to show that $1-p^j_i\geq \frac{1}{2s^j_i}$:
\[
\begin{aligned}
1-p^j_i
=& 1-(1-(1-\omega(x^j_i,y^j_i)))(1-(1-\omega(y^j_i,x^j_i))\\
=& (1-\omega(x^j_i,y^j_i))+(1-\omega(y^j_i,x^j_i))-(1-\omega(x^j_i,y^j_i))(1-\omega(y^j_i,x^j_i))\\
\geq & (1-\omega(x^j_i,y^j_i))+(1-\omega(y^j_i,x^j_i)) -\frac{1}{2}(1-\omega(y^j_i,x^j_i))\\
\geq &\frac{(1-\omega(x^j_i,y^j_i))+(1-\omega(y^j_i,x^j_i))}{2}\\
=&\frac{1}{2s^j_i}.
\end{aligned}
\]
So we get:
\[
\E_{\prob_0^{\omega}}\left(\sum\limits_{i=1}^{n}1_{N^j_i\geq m}1_{s^j_i\leq C}\sum\limits_{k=1}^{N^j_i}2H^j_{i,k}|\tilde{Y}\right)\leq 4\sum\limits_{i=1}^{n} N^j_i1_{N^j_i\geq m}s^j_i1_{s^j_i\leq C}.
\]
The actual value of $\ell_i^j$ can be slightly larger than $\sum\limits_{k=1}^{N^j_i}2H^j_{i,k}$ because this only counts the back-and-forths, so we miss the correct amount by $1$ every time the walks crosses the trap an even number of times and by $2$ every time the walks crosses the trap an odd number of times. So we get that the time $\ell_i^j$ the walk spends in the $i^{\text{th}}$ trap is smaller than $2N^j_i+\sum\limits_{j=1}^{N^j_i}2H^j_{i,k}$. For any positive constants $A,B>0$, let $E^n(B)$ be the event $\sum\limits_{i=1}^{n} N^j_i1_{N^j_i\geq m}1_{s^j_i\leq C}s^j_i\leq B$, we have:
\[
\begin{aligned}
&\prob_0^{\omega}\left(\sum\limits_{i=1}^n \ell^j_i1_{N^j_i\geq m}1_{s^j_i\leq C}\geq A \text{ and } \sum\limits N^j_i1_{s^j_i\leq C}1_{N^j_i\geq m}s^j_i\leq B\right)\\
=&\E_{\prob_0^{\omega}}\left(\prob_0^{\omega}\left(\sum\limits_{i=1}^n \ell^j_i1_{s^j_i\leq C}1_{N^j_i\geq m}\geq A |\tilde{Y}\right)1_{E^n(B)}\right)\\
\leq & \E_{\prob_0^{\omega}}\left(\frac{\sum\limits_{i=1}^n N^j_i1_{N^j_i\geq m}1_{s^j_i\leq C}(4s^j_i+2)}{A}1_{E^n(B)}\right)\\
\leq & \E_{\prob_0^{\omega}}\left(\frac{\sum\limits_{i=1}^n 5N^j_i1_{N^j_i\geq m}s^j_i1_{s^j_i\leq C}}{A}1_{E^n(B)}\right) \text{ since } s_i^j>2 \\
\leq & \E_{\prob_0^{\omega}}\left(\frac{5B}{A}1_{E^n(B)}\right)
\leq \frac{5B}{A}.
\end{aligned}
\]
\end{proof}

Now we want to show that we can neglect the time spent in traps in directions such that $\kappa_j\not=\kappa$ and in traps that are visited a lot of times. This will allow us to have traps that are rather similar so that the time spent in those traps are almost identically distributed.
\begin{lem}\label{lem:16}
Let $j\in [\!|1,d|\!]$ be an integer that represents the direction of the trap we will consider.
Let $\{x_i,y_i\}$ be the $i^{\text{th}}$ trap in the direction $j$ visited by the walk after time $\tau_2$ and such that $x_i.e_1 \geq Y_{\tau_2}.e_1$ and $y_i.e_1 \geq Y_{\tau_2}.e_1$. Let $\kappa_j=2\alpha-\alpha_j-\alpha_{j+d}\geq \kappa$.\\
If $\kappa <1$ there are two cases:
If $\kappa_j= \kappa$, for any $\varepsilon>0$ there exists an integer $m_{\varepsilon}$ such that for $n$ large enough:
\[
\Prob_0\left(\sum\limits_{i=1}^n \ell_i^j 1_{N^j_i\geq m_{\varepsilon}} \geq \varepsilon n^{\frac{1}{\kappa}}\right)
\leq \varepsilon.
\]
If $\kappa_j > \kappa$, for any $\varepsilon>0$ there exists an integer $n_{\varepsilon}$ such that for $n\geq n_{\varepsilon}$:
\[
\Prob_0\left(\sum\limits_{i=1}^n \ell_i^j \geq \varepsilon n^{\frac{1}{\kappa}} \right)
\leq \varepsilon.
\]
If $\kappa = 1$ there are two cases:
If $\kappa_j= \kappa$, for any $\varepsilon>0$ there exists an integer $m_{\varepsilon}$ such that for $n$ large enough:
\[
\Prob_0\left(\sum\limits_{i=1}^n \ell_i^j1_{N^j_i\geq m_{\varepsilon}} \geq \varepsilon n\log(n)\right)
\leq \varepsilon.
\]
If $\kappa_j > 1$, for any $\varepsilon>0$ there exists an integer $n_{\varepsilon}$ such that for $n\geq n_{\varepsilon}$:
\[
\Prob_0\left(\sum\limits_{i=1}^n \ell_i^j \geq \varepsilon n\log(n) \right)
\leq \varepsilon.
\]
\end{lem}
\begin{proof}
For all $i\geq 0$ let $t_i$ be the time at which the walk $Y$ enters its $i^{\text{th}}$ trap ($\{x_i,y_i\}$) in the direction $j$ after $\tau_2$ and such that $x_i.e_1 \geq Y_{\tau_2}.e_1$ and $y_i.e_1 \geq Y_{\tau_2}.e_1$. We will write $x_i$ the vertex such that $x_i=Y_{t_i}$. Let $s_i^j$ be the strength of the trap $\{x^j_i,y^j_i\}$. For any $A,B>0$:
\begin{equation}
\begin{aligned}[b]
\Prob_0(\exists i\leq n, s_i^j\geq A\text{ and } N^j_i \geq B)
&\leq \Prob_0\left(\left(\sum\limits_{i=1}^n N^j_i1_{s_i^j\geq A}\right)^{\kappa}\geq B^{\kappa}\right)\\
&\leq \Prob_0\left(\sum\limits_{i=1}^n (N^j_i)^{\kappa}1_{s_i^j\geq A}\geq B^{\kappa}\right)\\
&\leq \frac{1}{B^{\kappa}}\E_{\Prob_0}\left(\sum\limits_{i=1}^n (N^j_i)^{\kappa}1_{s_i^j\geq A}\right)\\
&\leq \frac{c}{B^{\kappa}}\frac{1}{A^{\kappa_j}}\E_{\Prob_0}\left(\sum\limits_{i=1}^n (N^j_i)^{\kappa}\right) \text{ \ \ \ by lemma \ref{lem:13}} \\
&\leq \frac{c}{B^{\kappa}}\frac{1}{A^{\kappa_j}}Cn \text{ \ \ \ by lemma \ref{lem:21}}. \label{eqn:16.11}
\end{aligned}
\end{equation}
We will first look at the case $\kappa<1$. \\
Now, we want to show that we can neglect traps with a high $N^j_i$ or a low $s^j_i$. We get that for any positive integer $M$, any real $A\geq 2$ and any $\beta\in [\kappa,1]$ and $\eta>0$ such that $\beta+\eta\leq \min\left(\frac{\kappa+\kappa^{\prime}}{2} ,1\right)$:
\[
\begin{aligned}
&\Prob_0\left(\sum\limits_{i=1}^n N^j_i s_i^j1_{s_i^j<A}1_{N^j_i\geq M}\geq (an)^{\frac{1}{\kappa}}\right)\\
\leq & \Prob_0\left(\sum\limits_{i=1}^n(N^j_i)^{\beta} (s_i^j)^{\beta} 1_{s_i^j<A}1_{N^j_i\geq M}\geq (an)^{\frac{\beta}{\kappa}}\right)\\
\leq & (an)^{-\frac{\beta}{\kappa}}\E_{\Prob_0}\left(\sum\limits_{i=1}^n(N^j_i)^{\beta} (s_i^j)^{\beta} 1_{s^j_i<A}1_{N^j_i\geq M}\right)\\
\leq & (an)^{-\frac{\beta}{\kappa}}M^{-\eta}\E_{\Prob_0}\left(\sum\limits_{i=1}^n(N^j_i)^{\beta+\eta} (s_i^j)^{\beta} 1_{s_i^j<A}\right)\\
\leq & (an)^{-\frac{\beta}{\kappa}}M^{-\eta}\E_{\Prob_0}\left(\int\limits_{t=0}^{A^{\beta}}\sum\limits_{i=1}^n(N^j_i)^{\beta+\eta}1_{(s_i^j)^{\beta}\geq t}\dd t\right)\\
\leq & (an)^{-\frac{\beta}{\kappa}}M^{-\eta}\sum\limits_{i=1}^n\int\limits_{t=0}^{A^{\beta}}\E_{\Prob_0}\left((N^j_i)^{\beta+\eta}1_{(s_i^j)^{\beta}\geq t}\right)\dd t\\
\leq & (an)^{-\frac{\beta}{\kappa}}M^{-\eta}\sum\limits_{i=1}^n \left(2\E_{\Prob_0}\left((N^j_i)^{\beta+\eta}\right)+\int\limits_{t=2^{\beta}}^{A^{\beta}}\E_{\Prob_0}\left((N^j_i)^{\beta+\eta}1_{s_i\geq t^{\frac{1}{\beta}}}\right)\dd t \right).
\end{aligned}
\]
By lemma \ref{lem:13}, there exists a constant $c$ such that $\E_{\Prob_0}\left((N^j_i)^{\beta+\eta}1_{s^j_i\geq t^{\frac{1}{\beta}}}\right)\leq\E_{\Prob_0}\left((N^j_i)^{\beta+\eta}\right) ct^{-\frac{\kappa}{\beta}}$, for $t\geq 2^{\beta}$ so:
\begin{equation}
\begin{aligned}  
&\Prob_0\left(\sum\limits_{i=1}^n N^j_i s_i^j1_{s_i^j<A}1_{N^j_i\geq M}\geq (an)^{\frac{1}{\kappa}}\right)\\
\leq & (an)^{-\frac{\beta}{\kappa}}M^{-\eta}\sum\limits_{i=1}^n \left(2\E_{\Prob_0}((N^j_i)^{\beta+\eta})+\E_{\Prob_0}\left((N^j_i)^{\beta+\eta}\right)\int\limits_{t=2^{\beta}}^{A^{\beta}} ct^{-\frac{\kappa}{\beta}}\dd t \right)\\
\leq & (an)^{-\frac{\beta}{\kappa}}M^{-\eta} \sum\limits_{i=1}^n\left(2+c\int\limits_{t=2^{\beta}}^{A^{\beta}} t^{-\frac{\kappa}{\beta}}\dd t \right)E_{\Prob_0}((N^j_i)^{\beta+\eta})\\
\leq & dn(an)^{-\frac{\beta}{\kappa}}M^{-\eta} \left(2+c\int\limits_{t=2^{\beta}}^{A^{\beta}} t^{-\frac{\kappa}{\beta}}\dd t \right) \text{ by lemma \ref{lem:21}} \label{eqn:16.1}.
\end{aligned}
\end{equation}
Now for $\kappa_j=\kappa$ if we take $\beta\in (\kappa,1]$  such that $\beta<\frac{\kappa+\kappa^{\prime}}{2}$, $\eta=0$ and $A=bn^{\frac{1}{\kappa}}$ we get:
\begin{equation}
\begin{aligned}
\Prob_0\left(\sum\limits_{i=1}^n N^j_i s_i^j1_{s_i^j<bn^{\frac{1}{\kappa}}}1_{N^j_i\geq M}\geq (an)^{\frac{1}{\kappa}}\right)
\leq& \frac{d}{a}n^{1-\frac{\beta}{\kappa}}\left(2+\frac{\beta c}{\beta-\kappa}\left(bn^{\frac{1}{\kappa}}\right)^{\beta-\kappa}\right)\\
\leq &\frac{d}{a}n^{1-\frac{\beta}{\kappa}}\left(2+\frac{\beta c}{\beta-\kappa}b^{\beta-\kappa}n^{\frac{\beta-\kappa}{\kappa}}\right)\\
= &2\frac{d}{a}n^{1-\frac{\beta}{\kappa}}+\frac{d}{a}\frac{\beta c}{\beta-\kappa}b^{\beta-\kappa}. \label{eqn:16.2}
\end{aligned}
\end{equation}
Now, we get by lemma \ref{lem:17} that for any positive constants $A,B$ and any positive integer $m$:
\begin{equation}
\begin{aligned}
&\Prob_0\left(\sum\limits_{i=1}^n \ell_i^j 1_{N^j_i\geq m}\geq A\right)\\
\leq &\Prob_0\left(\sum\limits_{i=1}^n \ell_i^j 1_{N^j_i\geq m}\geq A \text{ and } \sum\limits_{i=1}^n N^j_i1_{N^j_i\geq m}s_i^j \leq B\right) + \Prob_0\left(\sum\limits_{i=1}^n N^j_i1_{N^j_i\geq m}s_i^j\geq B\right)\\
\leq &\frac{5B}{A}+ \Prob_0\left(\sum\limits_{i=1}^n N^j_i1_{N^j_i\geq m} s_i^j\geq B\right). \label{eqn:16.3}
\end{aligned}
\end{equation}
So for any $\varepsilon>0$, for any $a>0$, by taking $B=\varepsilon^2 n^{\frac{1}{\kappa}}$ and $A=\varepsilon n^{\frac{1}{\kappa}}$ in \ref{eqn:16.3}, we have for any positive integer $m$:
\[
\Prob_0\left(\sum\limits_{i=1}^n \ell_i^j 1_{N^j_i\geq m}\geq \varepsilon n^{\frac{1}{\kappa}}\right) 
\leq 5\varepsilon + \Prob_0\left(\sum\limits_{i=1}^n N^j_i1_{N^j_i\geq m} s_i^j\geq \varepsilon^2 n^{\frac{1}{\kappa}}\right).
\] 
And we have for any $b>0$:
\[
\begin{aligned}
&\Prob_0\left(\sum\limits_{i=1}^n N^j_i1_{N^j_i\geq m} s_i^j\geq \varepsilon^2 n^{\frac{1}{\kappa}}\right)\\
\leq & \Prob_0\left(\sum\limits_{i=1}^n N^j_i1_{N^j_i\geq m} s_i^j1_{s_i^j\leq b n^{\frac{1}{\kappa}}}\geq \varepsilon^2 n^{\frac{1}{\kappa}}\right)
+ \Prob_0\left(\exists i\leq n, N^j_i\geq m \text{ and } s_i^j\geq b n^{\frac{1}{\kappa}}\right).
\end{aligned}
\]
We have by \ref{eqn:16.11}:
\[
\Prob_0\left(\exists i\leq n, N^j_i\geq m \text{ and } s_i^j\geq b n^{\frac{1}{\kappa}}\right)
\leq \frac{cdn}{(mb)^{\kappa}n}=\frac{cd}{(mb)^{\kappa}}.
\] 
And by \ref{eqn:16.2}, taking $b=\varepsilon^{\frac{2\kappa+1}{\beta-\kappa}}$:
\[
\begin{aligned}
\Prob_0\left(\sum\limits_{i=1}^n N^j_i1_{N^j_i\geq m} s_i^j1_{s_i^j\leq \varepsilon^{\frac{2\kappa+1}{\beta-\kappa}}n^{\frac{1}{\kappa}}}\geq \varepsilon^2 n^{\frac{1}{\kappa}}\right)
\leq &\frac{d}{\varepsilon^{2\kappa}}\left(2n^{1-\frac{\beta}{\kappa}}+\frac{\beta c}{\beta-\kappa}\varepsilon^{\frac{2\kappa+1}{\beta-\kappa}(\beta-\kappa)}\right)\\
=&\frac{d}{\varepsilon^{2\kappa}}\left(2n^{1-\frac{\beta}{\kappa}}+\frac{\beta c}{\beta-\kappa}\varepsilon^{2\kappa+1}\right).
\end{aligned}
\]
So for $n$ large enough:
\[
\Prob_0\left(\sum\limits_{i=1}^n N^j_i1_{N^j_i\geq m} s_i^j1_{s_i^j\leq \varepsilon^{\frac{2\kappa+1}{\beta-\kappa}}n^{\frac{1}{\kappa}}}\geq \varepsilon^2 n^{\frac{1}{\kappa}}\right)
\leq \frac{2d\beta c}{\beta-\kappa}\varepsilon
\]
which means that for $n$ large enough and $m_{\varepsilon}$ such that $m_{\varepsilon}\varepsilon^{\frac{2\kappa+1}{\beta-\kappa}} \geq \varepsilon^{-\frac{1}{\kappa}}$ we have:
\[
\Prob_0\left(\sum\limits_{i=1}^n \ell^j_i 1_{N^j_i\geq m_{\varepsilon}} \geq \varepsilon n^{\frac{1}{\kappa}}\right)
\leq 5 \varepsilon + cd\varepsilon + \frac{2d\beta c}{\beta-\kappa}\varepsilon.
\]
And we have the result we want.\\
If $\kappa_j>\kappa$ there exists $\beta \in (\kappa,\kappa_j) $ such that $\beta \leq 1$ and $\beta\leq \frac{\kappa+\kappa^{\prime}}{2}$ we get by taking $M=1$ and $A=\infty$ in \ref{eqn:16.1}:
\[
\begin{aligned}
\Prob_0\left(\sum\limits_{i=1}^nN^j_i s_i^j \geq (an)^{\frac{1}{\kappa}}\right)
\leq & da^{-\frac{\beta}{\kappa}}n^{1-\frac{\beta}{\kappa}}\left(2+\int\limits_{t=2}^{\infty}t^{-\frac{\kappa_j}{\beta}}\dd t\right)\\
= &  da^{-\frac{\beta}{\kappa}}n^{1-\frac{\beta}{\kappa}}\left(2+\frac{\beta}{\kappa_j-\beta}2^{1-\frac{\kappa_j}{\beta}}\right) \\
=& Ca^{-\frac{\beta}{\kappa}}n^{1-\frac{\beta}{\kappa}} \text{ for some constant } C.
\end{aligned}
\] 
And then lemma \ref{lem:17} gives us the result we want.\\
Now we can look at the case $\kappa=1$. \\
Let $\phi$ be a positive concave function such that $\phi(t)$ goes to infinity when $t$ goes to infinity. We define $\Phi$ by $\Phi(x):=\int\limits_{t=0}^x \phi\left(t\right) \dd t$. Let $f$ be defined by $f(0):=\phi(0)>0$ and $\forall x>0,\ f(x):=\frac{\Phi(x)}{x}$, we clearly have that $f(x)\geq f(0)$ and we have for any $y>x>0$:
\[
\begin{aligned}
f(y)
= & \frac{1}{y} \int\limits_{t=0}^y \phi(t) \dd t \\
= & \frac{1}{y} \frac{y}{x}\int\limits_{t=0}^x \phi\left(\frac{y}{x} t\right) \dd t \\
\geq & \frac{1}{x} \int\limits_{t=0}^x \phi\left(t\right) \dd t \\
= & f(x).
\end{aligned}
\]
We get that for any positive integer $M$ and any real $A\geq 2$:
\[
\begin{aligned}
&\Prob_0\left(\sum\limits_{i=1}^n N^j_i s_i^j1_{s_i^j<A}1_{N^j_i\geq M}\geq an\log(n)\right)\\
\leq & \frac{1}{an\log(n)}\E_{\Prob_0}\left(\sum\limits_{i=1}^n N^j_i s_i^j 1_{s_i^j<A}1_{N^j_i\geq M}\right)\\
\leq & \frac{1}{an\log(n)f(M)}\E_{\Prob_0}\left(\sum\limits_{i=1}^n N^j_if(N^j_i) s_i^j 1_{s_i^j<A}\right)\\
\leq & \frac{1}{an\log(n)f(M)}\sum\limits_{i=1}^n\int\limits_{t=0}^{A}\E_{\Prob_0}\left(\Phi(N^j_i)1_{s_i^j\geq t}\right)\dd t\\
\leq & \frac{1}{an\log(n)f(M)}\sum\limits_{i=1}^n \left(2\E_{\Prob_0}(\Phi(N^j_i))+\int\limits_{t=2}^{A}\E_{\Prob_0}\left(\Phi(N^j_i)1_{s_i^j\geq t}\right)\dd t \right).
\end{aligned}
\]
Now, by lemma \ref{lem:13} we get:
\begin{equation}
\begin{aligned}[b]
&\Prob_0\left(\sum\limits_{i=1}^n N^j_i s_i^j1_{s_i^j<A}1_{N^j_i\geq M}\geq an\log(n)\right) \\
\leq & \frac{1}{an\log(n)f(M)}\sum\limits_{i=1}^n \left(2\E_{\Prob_0}(\Phi(N^j_i))+\E_{\Prob_0}\left(\Phi(N^j_i)\right)\int\limits_{t=2}^{A} ct^{-\kappa_j}\dd t \right)\\
\leq & \frac{1}{an\log(n)f(M)} \sum\limits_{i=1}^n\left(2+c\int\limits_{t=2}^{A} t^{-\kappa_j}\dd t \right)E_{\Prob_0}(\Phi(N^j_i))\\
 \leq & \frac{dn}{an\log(n)f(M)}\left(2+c\int\limits_{t=2}^{A} t^{-\kappa_j}\dd t \right) \text{    by lemma \ref{lem:21}}. \label{eqn:16.10}
\end{aligned}
\end{equation}
If $\kappa_j=1$, we get, by taking $A= n^2$ (for $n\geq 2$) in \ref{eqn:16.10}:
\[
\Prob_0\left(\sum\limits_{i=1}^n N^j_i s_i^j1_{s_i^j\leq A}1_{N^j_i\geq M}\geq an\log(n)\right) 
\leq \frac{d}{a\log(n)f(M)}(2+2c\log(n)) \leq \frac{C}{af(M)}.
\]
And by taking $A=n^2$ and $B=1$ in equation \ref{eqn:16.11} we have for some constant $c$:
\[
\Prob_0\left(\exists i \leq n, s_i^j \geq n^2 \right)\leq \frac{c}{n}.
\]
So for any $\varepsilon>0$ we get, by taking $m_{\varepsilon}$ such that $f(m_{\varepsilon})\geq \frac{1}{\varepsilon^3}$ and using lemma \ref{lem:17}:
\[
\begin{aligned}
\Prob_0\left(\sum\limits_{i=1}^n \ell^j_i 1_{N^j_i\geq m_{\varepsilon}}\geq \varepsilon n\log(n)\right) 
\leq & 5\varepsilon + \Prob_0\left(\sum\limits_{i=1}^n N^j_i s_i^j 1_{N^j_i\geq M}\geq \varepsilon^2 n\log(n)\right) \\
\leq & 5\varepsilon + \frac{c}{n} + C\varepsilon.
\end{aligned}
\]
So there exists a constant $C$ such that for any $\varepsilon>0$ there exists $m_{\varepsilon}$ such that:
\[
\Prob_0\left(\sum\limits_{i=1}^n \ell_i^j 1_{\ell_i^j \geq m_{\varepsilon}}\geq \varepsilon n\log(n)\right)
\leq C\varepsilon.
\]
If $\kappa_j >1$, we take $M=0$ and $A=\infty$ in \ref{eqn:16.10} we get for some constant $C$:
\[
\Prob_0\left(\sum\limits_{i=1}^n N^j_i s_i^j\geq an\log(n)\right) 
\leq  \frac{d}{a\log(n)f(0)}\left(2+c\int\limits_{t=2}^{\infty} t^{-\kappa_j}\dd t \right) 
= \frac{C}{a\log(n)}.
\]
And therefore by lemma \ref{lem:17}, for any $\varepsilon >0$
\[
\begin{aligned}
\Prob_0\left(\sum\limits_{i=1}^n \ell_i^j \geq \varepsilon n\log(n)\right)
\leq & 5\varepsilon + \Prob_0\left(\sum\limits_{i=1}^n N^j_i s_i^j\geq \varepsilon^2 n\log(n)\right) \\
\leq & 5\varepsilon + \frac{C}{\varepsilon^2 \log(n)}.
\end{aligned}
\]
So we have the result we want
\end{proof}

Now we have all the tools to get a first limit theorem on the time spent in traps.
\begin{lem}\label{lem:18}
Set $\alpha\in(0,\infty)^{2d}$ and let $\overline{\alpha}:=\sum\limits_{i=1}^{2d}\alpha_i$. Let $J=\{ j\in [\!| 1,d |\!], 2\overline{\alpha}-\alpha_j-\alpha_{j+d}=\kappa \}$ and $\tilde{\mathcal{T}}_{j}$ be the set of vertices $x$ such that there exists $j\in J$ such that either $(x,x+e_j)\in\mathcal{T}$ or $(x,x-e_j)\in\mathcal{T}$. Let $\{x_i^j,y_i^j\}$ be the $i^{\text{th}}$ trap in the direction $j$ encountered after time $\tau_2$. \\ 
For $\kappa<1$, for any $m$ there exists a constant $C_m$ such that: 
\[
n^{-\frac{1}{\kappa}}\sum\limits_{j\in J}\sum\limits_{i\geq 0} \ell_i^j1_{N_i^j\leq m}1_{\exists k \leq \tau_{n+1}-1, Y_k\in \{x_i^j,y_i^j\} }\rightarrow C_m \mathcal{S}^{\kappa}_1 \text{ in law for } \Prob_0.
\]
For $\kappa=1$, for any $m$ there exists a constant $C_m>0$ such that:
\[
\frac{1}{n\log(n)}\sum\limits_{j\in J}\sum\limits_{i\geq 0} \ell_i^j1_{N_i^j\leq m}1_{\exists k \leq \tau_{n+1}-1, Y_k\in \{x_i^j,y_i^j\} }\rightarrow C_m \text{ in probability for }\Prob_0.
\]
\end{lem}
\begin{proof}
For every configuration $p\in\bigcup I_n$ let $C_p$ be the expectation of the number of traps of configuration $p$ encountered between times $\tau_2$ and $\tau_3-1$ (it is also the expectation of the number of traps of configuration $p$ encountered between times $\tau_i$ and $\tau_{i+1}-1$ for any $i \geq 2)$. We clearly have:
\[
C_p\leq \E_{\Prob_0}\left(\sum\limits_{x\in\Z^d}1_{\exists i\in[\tau_2,\tau_3-1],Y_i=x}\right)<\infty.
\] 
Once we know that a trap is in a direction $j\in J$ and has a configuration $p$ for some partially forgotten random walk, the exact number of back and forth the walk does in this trap is still random, because the exact number of back and forths knowing the transition probabilities of the trap is random and because the transition probabilities of the trap are still random, following the law (cf lemma \ref{lem:11}):
\[
C\frac{\varepsilon_x^{p_x}\varepsilon_y^{p_y}}{(\varepsilon_x+\varepsilon_y)^{p^s}}h(\varepsilon_x,\varepsilon_y)1_{\varepsilon_x+\varepsilon_y\leq \frac{1}{2}},
\] 
where $\varepsilon_x:=1-\omega(x,y)$, $\varepsilon_y:=1-\omega(y,x)$ and the value of $p_x,p_y,p^s$ are explicit but irrelevant, except for the fact that $p_x+p_y-p^s=\kappa-2$. Let $N$ be such that $p\in I_N$ (ie the walks exits the trap $N$ times) we also have that there exists a constant $C_{\alpha}$ that only depends on $\alpha$ such that: 
\[
|\log(h(\varepsilon_x,\varepsilon_y))|\leq C_{\alpha} N(\varepsilon_x+\varepsilon_y). \label{eqn:18.1}
\] 
Now if we make the change of variable $2r=\varepsilon_x+\varepsilon_y$, $k=\frac{\varepsilon_x-\varepsilon_y}{\varepsilon_x+\varepsilon_y}$, we get that the law of the transition probabilities becomes:
\[
\begin{aligned}
&2rC r^{p_x+p_y-p^s}\frac{(1+k)^{p_x}(1-k)^{p_y}}{(2r)^{p^s}}h(r(1+k),r(1-k))1_{r\leq \frac{1}{4}}\dd r \dd k\\
=&2^{1-p^s}Cr^{\kappa-1}(1+k)^{p_x}(1-k)^{p_y}h(r(1+k),r(1-k))1_{r\leq \frac{1}{4}}\dd r \dd k.
\end{aligned}
\]
The number of back and forths is the sum of $N$ iid geometric random variable $(H_1,\dots,H_N)$ of parameter $q=1-\varepsilon_x-\varepsilon_y+\varepsilon_x\varepsilon_y=1-2r+r^2(1-k^2)$ . This gives us the following bound:
\[
\begin{aligned}
\Prob\left(\sum\limits_{i=1}^N H_i\geq a|q\right)
&\leq N\Prob\left(H_1\geq\frac{a}{N}|q\right)\\
&\leq N(1-q)q^{\frac{a}{N}}\\
&\leq N\exp\left(\log(1-2r+r^2(1-k^2))\frac{a}{N}\right)\\
&\leq N\exp\left((-2r+r^2)\frac{a}{N}\right).
\end{aligned}
\] 
For $r\in\left[\frac{2\kappa n\log(a)}{a},\frac{1}{2}\right]$ we have $-2r+r^2\leq -r$ and
\[
\begin{aligned}
N\exp\left((-2r+r^2)\frac{a}{N}\right)
&\leq N\exp\left(-r\frac{a}{N}\right)\\
&\leq N\exp\left(-\frac{2\kappa N\log(a)}{a}\frac{a}{N}\right)\\
&= N a^{-2\kappa}.
\end{aligned}
\]
Now let $\ell^-$ be equal to twice the number of back-and-forths: $\ell^-:=2\sum\limits_{i=1}^N H_i$. Now we look at $\Prob\left(\ell^-\geq a \text{ and } r \leq \frac{2\kappa N\log(a)}{a}\right)$, we want to show that it is equivalent to $Ca^{-\kappa}$ for some constant $C$. First we want to have a good approximation of $\Prob\left(2\sum\limits_{i=1}^N H_i\geq a|q\right)$ for large $q$. Now let $\tilde{H}_1,\dots,\tilde{H}_n$ be iid exponential random variables of parameter $-\log(q)$ such that for every $i$, $H_i=\lfloor \tilde{H}_i \rfloor$. And we define $\tilde{\ell}^-=2\sum\limits_{i=1}^n \tilde{H}_i$. Now it is easy to show by induction on $n$ that:
\[
\Prob_0\left(\tilde{\ell}^-\geq 2a|q\right)=\sum\limits_{j=0}^{N-1}\frac{(-a\log(q))^j}{j!}\exp(\log(q)a).
\]
Now we clearly have:
\[
\ell^- \leq \tilde{\ell}^- \leq \ell^- +2N 
\]
so 
\[
\Prob_0\left(\ell^- \geq a|q\right) \leq\Prob_0\left(\tilde{\ell}^-\geq a|q\right)
\]
and 
\[
\Prob_0\left(\ell^-\geq a|q\right) \geq \Prob_0\left(\tilde{\ell}^-\geq a-2N|q\right).
\]
We want to show that $\Prob\left(\tilde{\ell}^-\geq a|q\right)$ and $\Prob\left(\tilde{\ell}^-\geq a-2N|q\right)$ are more or less equal. We clearly have:
\[
\Prob_0\left(\tilde{\ell}^-\geq a-2N|q\right) \leq \Prob_0\left(\tilde{\ell}^- \geq a|q\right)
\] 
and we also have:
\[
\begin{aligned}
\Prob_0\left(\tilde{\ell}^-\geq 2a-2N|q\right) 
&=\sum\limits_{j=0}^{N-1}\frac{(-a\log(p))^j}{j!}\left(1-\frac{N}{a}\right)^j\exp(\log(q)a)\exp(-\log(p)N)\\
&\geq \exp(-\log(q)N)\left(1-\frac{N}{a}\right)^N \sum\limits_{j=0}^{N-1}\frac{(-a\log(q))^j}{j!}\exp(\log(q)a).\\
\end{aligned}
\] 
First we want to show that we can replace $\log(q)$ by $-2r$. We clearly have $\log(q)\leq -2r +r^2$. We also have $\log(q)\geq \log(1-2r)$ and for $r\in[0,\frac{1}{4}]$, there exists a constant C that does not depend on $r$ such that $\log(1-2r)\geq -2r -C r^2$. So we get:
\[
2r-r^2\leq -\log(q) \leq 2r +Cr^2.
\]  
So 
\[
\exp(-2ar)\exp(-Car^2) \leq\exp(a\log(q))\leq \exp(-2ar)\exp(ar^2).
\]
So we get:
\[
\forall j, \ \frac{(-a\log(q))^j}{j!}\exp(\log(q)a)\leq \frac{(2ar)^j}{j!}\exp(-2ar)\left(1+\frac{Cr}{2}\right)^{j}\exp(ar^2)
\]
and
\[
\frac{(-a\log(q))^j}{j!}\exp(\log(q)a)\geq \frac{(2ar)^j}{j!}\exp(-2ar)\left(1-\frac{r}{2}\right)^{j}\exp(-Car^2).
\]
Now we will define $g^+(a,r)$ and $g^-(a,r)$ by:
\[
\begin{aligned}
g^+(a,r)&=\left(1+\frac{Cr}{2}\right)^{j}\exp(ar^2)\exp\left(2rC_{\alpha}N\right)\\
g^-(a,r)&=\left(1-\frac{r}{2}\right)^{j}\exp(-Car^2)\exp\left(-2rC_{\alpha}N\right)\exp(\left(2r-r^2\right)N)\left(1-\frac{N}{a}\right)^N,
\end{aligned}
\]
where $C$ is the same constant as in the previous inequality and $C_{\alpha}$ is the same as in $\ref{eqn:18.1}$. And for every $r\leq\frac{1}{4},k\in[-1,1]$ we have:
\[
\frac{(-a\log(q))^j}{j!}\exp(\log(q)a)h(r(1-k),r(1+k))\leq \frac{(2ar)^j}{j!}\exp(-2ar)g^+(a,r)
\]
and
\[
\frac{(-a\log(q))^j}{j!}h(r(1-k),r(1+k))\exp(\log(q)a)\geq \frac{(2ar)^j}{j!}\exp(-2ar)g^-(a,r).
\]
We clearly have that $g^+(a,r)$ is increasing in $r$ while $g^-(a,r)$ is decreasing in $r$ and $g^+(a,0)=1$ and $g^-(a,0)=\left(1-\frac{N}{a}\right)^N$.\\
So, for any $c>0$, we have the following $2$ inequalities:
\[
\begin{aligned}
&\Prob_0(\ell^-\geq 2a \text{ and } 1-q\leq c)\\
\leq &\Prob_0(\tilde{\ell}^-\geq 2a \text{ and } 1-q\leq c)\\
\leq &\Prob_0(\tilde{\ell}^-\geq 2a \text{ and } r\leq c)\text{\ \ \ \ \ \ since }1-q\geq 2r-r^2\geq r\\
= &\int\limits_{r=0}^{c}\int\limits_{k=-1}^1 2^{1-p^s}Cr^{\kappa-1}(1+k)^{p_x}(1-k)^{p_y}h(r(1+k),r(1-k))\Prob_0(\tilde{\ell}^-\geq 2a|q)\dd k\dd r\\
\leq & \int\limits_{r=0}^{c}\int\limits_{k=-1}^1 2^{1-p^s}Cr^{\kappa-1}(1+k)^{p_x}(1-k)^{p_y}\sum\limits_{j=0}^{N-1}\frac{(2ar)^j}{j!}\exp(-2ar) g^+(a,r)\dd k\dd r\\
\leq & g^+(a,c) \int\limits_{k=-1}^1 (1+k)^{p_x}(1-k)^{p_y}\dd k \int\limits_{r=0}^{c}2^{1-p^s}Cr^{\kappa-1}\sum\limits_{j=0}^{N-1}\frac{(2ar)^j}{j!}\exp(-2ar)\dd r,
\end{aligned}
\]
and
\[
\begin{aligned}
&\Prob_0(\ell^-\geq 2a \text{ and } 1-q\leq c)\\
\geq &\Prob_0(\tilde{\ell}^-\geq 2a-2N \text{ and } 1-q\leq c)\\
\geq &\Prob_0(\tilde{\ell}^-\geq 2a-2N \text{ and } 2r\leq c)\text{\ \ \ \ \ \ since }1-q\leq 2r\\
= &\int\limits_{r=0}^{\frac{c}{2}}\int\limits_{k=-1}^1 2^{1-p^s}Cr^{\kappa-1}(1+k)^{p_x}(1-k)^{p_y}h(r(1+k),r(1-k))\Prob_0(\tilde{\ell}^-\geq 2a-2N|q)\dd k\dd r\\
\geq & \int\limits_{r=0}^{\frac{c}{2}}\int\limits_{k=-1}^1 2^{1-p^s}Cr^{\kappa-1}(1+k)^{p_x}(1-k)^{p_y}\sum\limits_{j=0}^{N-1}\frac{(a2r)^j}{j!}\exp(-2ar) g^-(a,r)\dd k\dd r\\
\geq & g^-\left(a,\frac{c}{2}\right) \int\limits_{k=-1}^1 (1+k)^{p_x}(1-k)^{p_y}\dd k \int\limits_{r=0}^{\frac{c}{2}}2^{1-p^s}Cr^{\kappa-1}\sum\limits_{j=0}^{N-1}\frac{(a2r)^j}{j!}\exp(-2ar)\dd r.
\end{aligned}
\]
If we take $c=a^{-\frac{3}{4}}$ we clearly get when $a\rightarrow\infty$, $g^-(a,a^{-\frac{3}{4}})\rightarrow 1$ and $g^+(a,a^{-\frac{3}{4}})\rightarrow 1$. Furthermore, for any constant $c^{\prime}$:
\[
\begin{aligned}
&\int\limits_{r=0}^{c^{\prime}a^{-\frac{3}{4}}}2^{1-p^s}Cr^{\kappa-1}\sum\limits_{j=0}^{N-1}\frac{(a2r)^j}{j!}\exp(-2ar)\dd r\\
=& (2a)^{-\kappa} \int\limits_{r=0}^{2c^{\prime}a^{\frac{1}{4}}}2^{1-p^s}Cr^{\kappa-1}\sum\limits_{j=0}^{N-1}\frac{r^j}{j!}\exp(-r)\dd r\\
\sim & (2a)^{-\kappa} \sum\limits_{j=0}^{N-1}\frac{\Gamma(j+1)}{j!}\\
=& (2a)^{-\kappa}N.
\end{aligned}
\]
Therefore we get:
\[
\Prob_0(\ell^-\geq 2a \text{ and } 1-q\leq a^{\frac{3}{4}})\sim N\left(\int\limits_{k=-1}^1 (1+k)^{p_x}(1-k)^{p_y}\dd k \right)2^{1-p^s}C(2a)^{-\kappa}.
\]
So there exist a constant $C$ that only depends on $\alpha$ such that:
\[
\Prob_0(\ell^-\geq 2a \text{ and } 1-q\leq a^{-\frac{3}{4}})\sim CNa^{-\kappa}.
\]
So we get for some constant $C^{\prime}$:
\[
\Prob_0(\ell^-\geq a)\sim C^{\prime}Na^{-\kappa}.
\]
Now let $\ell$ be the total time spent in the trap. It is equal to $\ell^-$ plus the number of time the walk enters and exits the trap by the same vertex plus twice the number of times the walk enters and exits the trap by different vertices. This means there exists a constant $\delta_p$ that only depends on the configuration such that $\ell=\ell^-+\delta_p$. This, in turn, means that we have also the asymptotic equality:
\[
\Prob_0(\ell \geq a)\sim C^{\prime}Na^{-\kappa}.
\]
Now, let $\ell^p_i$ be the time spent in the $i^{\text{th}}$ trap with configuration $p$.\\
First, if $\kappa<1$, by Theorem $3.7.2$ of $\cite{Durrett}$ we get that for some constant $c_p$: 
\[
n^{-\frac{1}{\kappa}}\sum\limits_{i=1}^n \ell_i^p \rightarrow c_p \mathcal{S}_1^{\kappa} \text{ in law for } \Prob_0.
\]
Now we use the fact that the number of trap of configuration $p$ between two renewal times has a finite expectation $C_p$ to show that we have the convergence we want. Let $M_{n,p}$ be the number of traps of configuration $p$ the walk has entered before the $n^{\text{th}}$ renewal time. For any $\varepsilon>0$ and any $p$ we have:
\[
\Prob_0(M_{n,p}\in[(C_p-\varepsilon)n,(C_p+\varepsilon)n])\rightarrow 1.
\]
Therefore for any configuration $p$:
\[
n^{-\frac{1}{\kappa}}\sum\limits_{i=(C_p-\varepsilon)n}^{(C_p+\varepsilon)n}\ell^p_i\rightarrow (2\varepsilon)^{\frac{1}{\kappa}}c_pS_{\kappa} \text{ in law for } \Prob_0.
\] 
And for any $m\in\N$:
\[
n^{-\frac{1}{\kappa}}\sum\limits_{p\in I^m}\sum\limits_{i=(C_p-\varepsilon)n}^{(C_p+\varepsilon)n}\ell^p_i\rightarrow (2\varepsilon)^{\frac{1}{\kappa}}\left(\sum\limits_{p\in I^m} (c_p)^{\kappa}\right)^{\frac{1}{\kappa}}S_{\kappa} \text{ in law for } \Prob_0.
\]
We write $I^m(J)$ all the configuration of $I^m$ that are in a direction $j\in J$. Now, using the fact that the $\ell_p^i$ are non negative, for any $n\in\N$ and any $\varepsilon>0$ small enough, we have:
\[
\begin{aligned}
&\Prob_0\left(n^{-\frac{1}{\kappa}}\left|\sum\limits_{p\in I^m(J)}\sum\limits_{i=1}^{M_{n,p}}\ell^p_i-\sum\limits_p\sum\limits_{i=1}^{C_pn}\ell_p^i\right|\geq \eta \right)\\
\leq &\Prob_0(\exists p\in I^m(J),N_{n,p}\not\in[(C_p-\varepsilon)n,(C_p+\varepsilon)n])+\Prob\left(n^{-\frac{1}{\kappa}}\sum\limits_{p\in I^m(J)}\sum\limits_{i=(C_p-\varepsilon)n}^{(C_p+\varepsilon)n}\ell^p_i\geq \eta\right)\\
= & o(1) + \Prob_0\left((2\varepsilon)^{\frac{1}{\kappa}}\left(\sum\limits_{p\in I^m(J)} (c_p)^{\kappa}\right)^{\frac{1}{\kappa}}\mathcal{S}^{\kappa}_1\geq \eta\right).
\end{aligned}
\] 
Since it is true for all $\varepsilon$, we get that 
\[
n^{-\frac{1}{\kappa}}\left|\sum\limits_{p\in I^m(J)}\sum\limits_{i=1}^{M_{n,p}}\ell^p_i-\sum\limits_{p\in I^m(J)}\sum\limits_{i=1}^{C_pn}\ell_p^i\right|\rightarrow 0 \text{ in probability for } \Prob_0.
\]
And since 
\[
n^{-\frac{1}{\kappa}}\sum\limits_{p\in I^m(J)}\sum\limits_{i=1}^{C_pn}\ell^p_i\rightarrow\left(\sum\limits_{p\in I^m(J)} (c_p)^{\kappa}\right)^{\frac{1}{\kappa}}\mathcal{S}^{\kappa}_1 \text{ in probability for } \Prob_0,
\]
we get:
\[
n^{-\frac{1}{\kappa}}\sum\limits_{p\in I^m(J)}\sum\limits_{i=1}^{M_{n,p}}\ell^p_i\rightarrow\left(\sum\limits_{p\in I^m(J)} (c_p)^{\kappa}\right)^{\frac{1}{\kappa}}\mathcal{S}^{\kappa}_1\text{ in law for } \Prob_0
\]
Now if $\kappa=1$, we first want to show that we can neglect the values larger than $n\log(n)$.
Let $p$ be a configuration, $\ell^p_i$ the total time spent in the $i^{\text{th}}$ trap in the configuration $p$ encountered, $C_p$ the constant such that the number of trap encountered before time $\tau_{n+1}-1$ is equivalent to $C_p n$, $M_{n,p}$ the number of traps in the configuration $p$ encountered before the time $\tau_{n+1}-1$ and $c_p$ the constant such that $\Prob_0(\ell^p_i \geq t) \sim c_p n^{-1}$. We get:
\[
\begin{aligned}
\Prob_0(\exists i \leq M_{n,p} , \ell^p_i \geq n\log(n) )
\leq& \Prob_0(\exists i \leq 2C_p n , \ell^p_i \geq n\log(n) ) + \Prob_0(M_{n,p} \geq 2C_p n)\\
\leq&  2C_p n \frac{c_p}{n\log(n)} +o(1).\\
=& o(1)
\end{aligned}
\]
Now we can compute the expectation and variance of $\ell^p_i \wedge n\log(n)$:
\[
\begin{aligned}
\E_{\Prob_0}(\ell^p_i \wedge n\log(n)) 
\sim &\int\limits_{t=1}^{n\log(n)} \frac{c_p}{t} \dd t \\
\sim & c_p \log(n). \label{eqn:18.2}
\end{aligned}
\]
Now for the variance we get:
\[
\begin{aligned}
\text{Var}_{\Prob_0}(\ell^p_i \wedge n\log(n))
\leq & \E_{\Prob_0} ((\ell^p_i \wedge n\log(n))^2)\\
\sim & \int\limits_{t=1}^{n\log(n)} 2 t \frac{c_p}{t} \dd t \\
\sim & 2 c_p n \log(n).
\end{aligned} 
\]
So  for $n$ large enough:
\[
\text{Var}_{\Prob_0}(\ell^p_i \wedge n\log(n)) \leq 4 c_p n\log(n).
\]
First, for any constant $c$, for $n$ big enough:
\[
\begin{aligned}
&\Prob_0\left(\left|\sum\limits_{i=1}^{cn}\ell_i^p\wedge n\log(n)- cnc_p\log(n)\right| \geq \varepsilon n\log(n)\right) \\
\leq & \Prob_0\left(\left|\sum\limits_{i=1}^{cn}\ell_i^p\wedge n\log(n)-  cn\E\left(\ell_1^p\wedge n\log(n)\right)\right| \geq \frac{1}{2}\varepsilon n\log(n)\right) \text{ for } n \text{ big enough, by }\ref{eqn:18.2}\\
\leq & cn\frac{4\text{Var}_{\Prob_0}\left(\ell_1^p\wedge n\log(n)\right)}{(\varepsilon n\log(n))^2}\\
\leq & cn \frac{16 c_p n\log(n)}{(\varepsilon n\log(n))^2} \\
= & \frac{16c c_p}{\log(n)} =o(1)
\end{aligned}
\]
This means that we have the following results:
\[
\Prob_0\left(\sum\limits_{i=1}^{(C_p+\varepsilon)n}\ell_i^p\wedge n\log(n)- (C_p+\varepsilon)nc_p\log(n) \geq \varepsilon n\log(n)\right) \rightarrow 0
\] 
and
\[
\Prob_0\left(\sum\limits_{i=1}^{(C_p-\varepsilon)n}\ell_i^p\wedge n\log(n)- (C_p+\varepsilon)nc_p\log(n) \leq -\varepsilon n\log(n)\right) \rightarrow 0
\] 
Then, by definition of $C_p$ we get, for any $\varepsilon \geq 0$: 
\[
\Prob_0\left(|M(n,p)-C_p n|\geq \varepsilon n\right)\rightarrow 0.
\]
Then, using the fact that $\sum\limits_{i=1}^n \ell_i^p \wedge a$ is increasing in $n$ for any $a$, we get:
\[
\begin{aligned}
&\Prob_0\left(\sum\limits_{i=1}^{M(n,p)}\ell_i^p\geq (C_p+\varepsilon)(c_p +\varepsilon)n\log(n)\right)\\
\leq & \Prob_0(M(n,p)\geq (C_p + \varepsilon) n) + \Prob_0\left(\sum\limits_{i=1}^{(C_p + \varepsilon) n}\ell_i^p\geq (C_p+\varepsilon)(c_p+\varepsilon)n\log(n)\right)\\
=& o(1).
\end{aligned}
\]
Similarly, we have:
\[
\begin{aligned}
&\Prob_0\left(\sum\limits_{i=1}^{M(n,p)}\ell_i^p\leq (C_p-\varepsilon)(c_p -\varepsilon)n\log(n)\right)\\
\leq & \Prob_0(M(n,p)\geq (C_p - \varepsilon) n) + \Prob_0\left(\sum\limits_{i=1}^{(C_p - \varepsilon) n}\ell_i^p\geq (C_p-\varepsilon)(c_p-\varepsilon)n\log(n)\right)\\
=& o(1).
\end{aligned}
\]
Therefore,
\[
\frac{1}{n\log(n)}\sum\limits_{i=1}^{M(n,p)}\ell_i^p \rightarrow C_p c_p \text{ in probability for } \Prob_0.
\]
Now we just have to sum on all configurations $p\in I^m$ that are in a direction $j\in J$ to get the result we want.

\end{proof}

\subsection{Only the time spent in traps matter}

Now to properly show the result we want, we have to show that some quantities and some events are negligible, this is what this section is devoted to.
\begin{lem}\label{lem:31}
Let $j$ be in $[\!|1,d|\!]$. Let $\{x_i^j,y_i^j\}$ be the $i^{\text{th}}$ trap visited by the walk in the direction $j$ after time $\tau_2$, $s_i^j$ its strength, $\ell^j_i$ the time spent in this trap and $N^j_i$ the number of times the walk enters the trap:
\[
\begin{aligned}
\ell^j_i&=\sum\limits_{k\geq 0} 1_{Y_k\in \{x^j_i,y^j_i\}},\\
N^j_i&=\sum\limits_{k\geq 0} 1_{Y_k\in \{x_i,y_i\}\text{ and }Y_{k+1}\not\in \{x^j_i,y^j_i\}}.
\end{aligned}
\]
Let $\kappa_j=2\sum\limits_{i=1}^{2d}\alpha_i-\alpha_j-\alpha_{j+d}\geq \kappa$. Let $M(n,j)$ be the number of traps in the direction $j$ encountered between times $\tau_2$ and $\tau_n - 1$.\\
If $\kappa <1$ and $\kappa_j=\kappa$, for any $\varepsilon>0$ there exists $\varepsilon^{\prime}>0$ such that for $n$ large enough:
\[
\Prob_0\left(\sum\limits_{i=1}^{M(n,j)} \ell_i^j1_{s^j_i \leq \varepsilon^{\prime} n^{\frac{1}{\kappa}}} \geq \varepsilon n^{\frac{1}{\kappa}} \right) \leq \varepsilon.
\]
\end{lem}
\begin{proof}
Let $\gamma \in \left(\kappa,\frac{\kappa+\kappa^{\prime}}{2}\right)$ be such that $\gamma \leq 1$. Let $\beta$ be a positive real. Let $\{\overline{x}^j_i,\overline{y}^j_i\}$ be the $i^{\text{th}}$ trap visited by the walk in the direction $j$ after time $\tau_2$ such that $\{\overline{x}^j_i.e_1,\overline{y}^j_i.e_1\geq Y_{\tau_2}.e_1$. Let $\overline{s}^j_i$ be its strength $\overline{\ell}^j_i$ the time spent in this trap and $\overline{N}^j_i$ the number of times the trap is visited. \\
By lemma \ref{lem:9} the number of traps encountered between $2$ renewal times has a finite expectation and since the $(M(2i+1,j)-M(2i,j))_{i\in \N^*}$ are iid and so are the $(M(2i+2,j)-M(2i+1,j))_{i\in \N^*}$, there exists a constant $C_j$ such that $\Prob_0$ almost surely:
\[
\frac{1}{n}M(n,j) \rightarrow C_j.
\]
So for any $\varepsilon>0$, for $n$ large enough: 
\[
\Prob_0(M(n,j) \geq 2C_j n)\leq \frac{\varepsilon}{4}.
\]
We have for $n$ large enough: 
\[
\begin{aligned}
&\Prob_0\left(\sum\limits_{i=1}^{M(n,j)} \ell_i^j1_{s^j_i\leq\varepsilon^{\prime} n^{\frac{1}{\kappa}}}\geq \varepsilon n^{\frac{1}{\kappa}}\right)\\
\leq & \Prob_0\left(\sum\limits_{i=1}^{M(n,j)} \overline{\ell}^j_i1_{ \overline{s}^j_i\leq\varepsilon^{\prime} n^{\frac{1}{\kappa}}}\geq \frac{1}{2}\varepsilon n^{\frac{1}{\kappa}}\right) + \Prob_0\left(\sum\limits_{i=1}^{M(3,j)} \ell^j_i1_{s_i\leq\varepsilon^{\prime} n^{\frac{1}{\kappa}}}\geq \frac{1}{2}\varepsilon n^{\frac{1}{\kappa}}\right)\\
\leq & \Prob_0\left(\sum\limits_{i=1}^{M(n,j)} \overline{\ell}^j_i1_{ \overline{s}^j_i\leq\varepsilon^{\prime} n^{\frac{1}{\kappa}}}\geq \frac{1}{2}\varepsilon n^{\frac{1}{\kappa}}\right) + \Prob_0\left(\tau_3 \geq \frac{1}{2}\varepsilon n^{\frac{1}{\kappa}}\right)\\
\leq & \Prob_0\left(\sum\limits_{i=1}^{M(n,j)} \overline{\ell}^j_i1_{ \overline{s}^j_i\leq\varepsilon^{\prime} n^{\frac{1}{\kappa}}}\geq \frac{1}{2}\varepsilon n^{\frac{1}{\kappa}}\right) + \frac{\varepsilon}{4} \text{ for } n \text{ large enough } \\
\leq & \Prob_0\left(\sum\limits_{i=1}^{2C_j n} \overline{\ell}^j_i1_{ \overline{s}^j_i\leq\varepsilon^{\prime} n^{\frac{1}{\kappa}}}\geq \frac{1}{2}\varepsilon n^{\frac{1}{\kappa}}\right) + 2\frac{\varepsilon}{4} \text{ for } n \text{ large enough }.
\end{aligned}
\]
Then by lemma \ref{lem:17} we have:
\[
\Prob_0\left(\sum\limits_{i=1}^{2C_j n} \overline{\ell}^j_i1_{ \overline{s}^j_i\leq\varepsilon^{\prime} n^{\frac{1}{\kappa}}} \geq \frac{1}{2}\varepsilon n^{\frac{1}{\kappa}}\right)\leq \frac{\varepsilon}{4}+ \Prob_0\left(\sum\limits_{i=1}^{2C_j n} \overline{N}^j_i\overline{s}^j_i1_{\overline{s}^j_i\leq\varepsilon^{\prime} n^{\frac{1}{\kappa}}} \geq \frac{\varepsilon^2}{40} n^{\frac{1}{\kappa}}\right).
\]
And finally we have:
\[
\begin{aligned}
\Prob_0\left(\sum\limits_{i=1}^{2C_j n} \overline{N}^j_i\overline{s}^j_i1_{\overline{s}^j_i\leq\varepsilon^{\prime} n^{\frac{1}{\kappa}}} \geq \frac{\varepsilon^2}{40} n^{\frac{1}{\kappa}}\right)
\leq & \Prob_0\left(\sum\limits_{i=1}^{2C_j n} (\overline{N}^j_i)^{\gamma}(\overline{s}_i^j)^{\gamma}1_{\overline{s}_i^j\leq\beta n^{\frac{1}{\kappa}}} \geq \left(\frac{\varepsilon^2}{40} n^{\frac{1}{\kappa}}\right)^{\gamma}\right)\\
\leq & \left(\frac{\varepsilon^2}{40} n^{\frac{1}{\kappa}}\right)^{-\gamma}\E_{\Prob_0}\left(\sum\limits_{i=1}^{2C_j n} (\overline{N}^j_i)^{\gamma}(\overline{s}^j_i)^{\gamma}1_{\overline{s}^j_i\leq\beta n^{\frac{1}{\kappa}}}\right)\\
= & \left(\frac{\varepsilon^2}{40} n^{\frac{1}{\kappa}}\right)^{-\gamma}\sum\limits_{i=1}^{2C_j n} \E_{\Prob_0}\left( (\overline{N}^j_i)^{\gamma}(\overline{s}^j_i)^{\gamma}1_{\overline{s}^j_i\leq\beta n^{\frac{1}{\kappa}}}\right).
\end{aligned}
\]
Then by lemma \ref{lem:13} we get, for some constant $c$ that does not depend on $\beta$:
\[
\begin{aligned}
\left(\frac{\varepsilon^2}{40} n^{\frac{1}{\kappa}}\right)^{-\gamma}\sum\limits_{i=1}^{2C_j n} \E_{\Prob_0}\left( (\overline{N}^j_i)^{\gamma}(\overline{s}^j_i)^{\gamma}1_{\overline{s}_i\leq\beta n^{\frac{1}{\kappa}}}\right)
\leq & c\left(\frac{\varepsilon^2}{40} n^{\frac{1}{\kappa}}\right)^{-\gamma}\sum\limits_{i=1}^{2C_j n} \E_{\Prob_0}\left( (\overline{N}_i^j)^{\gamma}\right) \left(\varepsilon^{\prime} n^{\frac{1}{\kappa}}\right)^{\gamma-\kappa}\\
= & c\left(\frac{\varepsilon^2}{40}\right)^{-\gamma}\left(\varepsilon^{\prime}\right)^{\gamma-\kappa}  n^{-1}\sum\limits_{i=1}^{2C_j n} \E_{\Prob_0}\left( (\overline{N}^j_i)^{\gamma}\right).
\end{aligned}
\]
And by lemma \ref{lem:21} there exists a constant $c$ that does not depend on $\beta$ such that:
\[
\left(\frac{\varepsilon^2}{40}\right)^{-\gamma}\left(\varepsilon^{\prime}\right)^{\gamma-\kappa}  n^{-1}\sum\limits_{i=1}^{2C_j n} \E_{\Prob_0}\left( (\overline{N}^j_i)^{\gamma}\right)
\leq c\left(\frac{\varepsilon^2}{40}\right)^{-\gamma}\left(\varepsilon^{\prime}\right)^{\gamma-\kappa}.
\]
So by taking $\beta$ small enough we get the result we wanted.
\end{proof}

\begin{lem}\label{lem:20}
Let $J=\{ j\in[\![1,d]\!],\kappa_j>\kappa\}$. \\
If $\kappa=1$ there exists a constant C such that $\Prob_0$ almost surely:
\[
\frac{1}{n}\sum\limits_{i=0}^{\tau_n-1}1_{Y_i\in\tilde{\mathcal{T}}_J} \rightarrow C.
\]
If $\kappa<1$ there exists a constant $C>0$ and a constant $\gamma \in (\kappa,1]$ such that $\Prob_0$ almost surely:
\[
\limsup n^{-\frac{1}{\gamma}}\sum\limits_{k=0}^{\tau_{n}-1}1_{Y_k\in \tilde{\mathcal{T}}_J} \leq C.
\]
\end{lem}
\begin{proof}
For any $j\in J$ we define $\kappa_j=2\sum\limits_{i=1}^{2d}\alpha_i -\alpha_j-\alpha_{j+d}>\kappa$. Let $\{x^j_i,y^j_i\}$ be the $i^{\text{th}}$ trap in the direction $j$ the walk enters after time $\tau_2$ and such that $x^j_i.e_1,y^j_i.e_1\geq Y_{\tau_2}.e_1$. Let $N^j_i$ be the number of times the walk exits $\{x^j_i,y^j_i\}$ and $\ell^j_i$ the time the walk spends in this trap. Let $M(i,j)$ be the number of traps in the direction $j$ entered before time $\tau_{i}$. The $(M(2i+2,j)-M(2i+1,j))_{i\in\N^*}$ are iid and so are the $(M(2i+1,j)-M(2i,j))_{i\in\N^*}$, they also all have the same law (the only issue is that since a trap span over two vertices, there might be a slight overlap between traps of two different 'renewal slabs'). Now, since the number of different vertices the walk encounters between two renewal times has a finite expectation, the $(M(i+1,j)-M(i,j))$ have a finite expectation and therefore there exists a constant $C_j$ such that $\Prob_0$ almost surely:
\[
M(n,j) -C_j n\rightarrow - \infty.
\]
Now let $\tilde{Y}$ be the partially forgotten walk associated with $Y$. We get that knowing the environment, the partially forgotten walk and the renewal position $Y_{\tau_2}$ the time spend in the $\{x^j_i,y^j_i\}$, the $k^{\text{th}}$ time the walk enters this trap is equal to $\varepsilon_{i,k}^j +2H_{i,k}^j$ where $\varepsilon_{i,k}^j$ is $1$ if the walk enters the trap by the same vertex it leaves it and $2$ otherwise and $H_{i,k}^j$ is a geometric random variable that counts the number of back and forths. The parameter of $H_{i,k}^j$ is $p_i^j:=\omega(x_i^j,y_i^j)\omega(y^j_i,x^j_i)$.\\
First, lets look at the case $\kappa=1$. Since the $\left(\sum\limits_{j=\tau_{2i}}^{\tau_{2i+1}-1}1_{Y_{i}\in \tilde{\mathcal{T}}_J}\right)_{i\in\N^*}$ are iid and so are the $\left(\sum\limits_{j=\tau_{2i+1}}^{\tau_{2i+2}-1}1_{Y_{i}\in \tilde{\mathcal{T}}_J}\right)_{i\in\N^*}$, we just have to prove that their expectation is not infinite to have the result we want. If their expectation were infinite, then we would have that $\Prob_0$ almost surely:
\[
\frac{1}{n}\sum\limits_{j\in J}\sum\limits_{i=1}^{M(n,j)}\ell^j_i \rightarrow \infty.
\]
Therefore we would have $\Prob_0$ almost surely:
\[
\frac{1}{n}\sum\limits_{j\in J}\sum\limits_{i=1}^{C_j n}\ell^j_i \rightarrow \infty.
\]
But
\[
\begin{aligned}
\E_{\prob_0^{\omega}}\left(\frac{1}{n}\sum\limits_{j\in J}\sum\limits_{i=1}^{C_j n}\ell^j_i |\tilde{Y} \right)
= & \frac{1}{n}\sum\limits_{j\in J}\sum\limits_{i=1}^{C_j n}\sum\limits_{k=1}^{N_i^j}\E_{\prob_0^{\omega}}\left(\varepsilon_{i,k}^j+2H_{i,k}^j|\tilde{Y}\right)  \\
= & \frac{1}{n}\sum\limits_{j\in J}\sum\limits_{i=1}^{C_j n}\sum\limits_{k=1}^{N_i^j}\left(\varepsilon_{i,k}^j+2\frac{p_i^j}{1-p_i^j}\right) \\
\leq & 2\frac{1}{n}\sum\limits_{j\in J}\sum\limits_{i=1}^{C_j n}\sum\limits_{k=1}^{N_i^j}\frac{1}{1-p_i^j}\\
\leq & C\frac{1}{n} \sum\limits_{j\in J}\sum\limits_{i= 1 }^{nC_j} N_i^j s_i^j,
\end{aligned}
\]
where $s_i^j$ is the strength of the trap $\{x_i^j,y_i^j\}$. Now we get:
\[
\begin{aligned}
\E_{\Prob_0}\left(\frac{1}{n}\sum\limits_{j\in J}\sum\limits_{i=1}^{C_j n}\ell_i^j  \right) 
\leq & \frac{1}{n}\E_{\Prob_0}\left(C \sum\limits_{j\in J}\sum\limits_{i= 1 }^{nC_j} N_i^j  s_i^j \right) \\
= & C\frac{1}{n} \E_{\Prob_0}\left(\sum\limits_{j\in J}\sum\limits_{i= 1}^{nC_j} N_i^j  \int\limits_{t=0}^{\infty}1_{(s_i^j)\geq t}\dd t \right)\\
\leq  & C\frac{1}{n}\sum\limits_{j\in J}\sum\limits_{i= 1}^{nC_j}\E_{\Prob_0}\left( N_i^j  \left(2+\int\limits_{t=2}^{\infty}1_{s_i^j\geq t}\dd t\right) \right)\\
\leq & C\frac{1}{n}\sum\limits_{j\in J}\sum\limits_{i= 1}^{nC_j} \left( 2\E_{\Prob_0}(N_i^j ) + C \int_{t=2}^{\infty} \E_{\Prob_0}(N_i^j1_{s_i^j\geq t})\dd t\right).
\end{aligned}
\]
Now by lemma \ref{lem:13} we know that there exists a constant $C$ such that for any $t\geq 2$: 
\[
\E_{\Prob_0}(N_i^j 1_{s_i^j\geq t})\leq C t^{-\kappa_j}\E_{\Prob_0}(N_i^j).
\]
So there exists a constant $C^{\prime}$ (the value of this constant will change depending on the line) such that:
\[ 
\begin{aligned}
\E_{\Prob_0}\left(\frac{1}{n}\sum\limits_{j\in J}\sum\limits_{i=1}^{C_j n}\ell_i^j  \right) 
\leq & C^{\prime}\frac{1}{n} \sum\limits_{j\in J}\sum\limits_{i= 1}^{nC_j} \E_{\Prob_0}(N_i^j)\\
\leq & C^{\prime}\sum\limits_{j\in J} C_j \text{\ \ \ by lemma \ref{lem:21}}\\
\leq & C^{\prime}.
\end{aligned}
\]
This means that we cannot have $\frac{1}{n}\sum\limits_{j\in J}\sum\limits_{i=1}^{C_j n}N_i^j\rightarrow \infty$ $\Prob_0$ almost surely. Therefore the random variables $\left(\sum\limits_{j=\tau_{2i}}^{\tau_{2i+1}-1}1_{Y_{i}\in \tilde{\mathcal{T}}_J}\right)_{i\in\N^*}$ have finite expectation and so have the random variables $\left(\sum\limits_{j=\tau_{2i+1}}^{\tau_{2i+2}-1}1_{Y_{i}\in \tilde{\mathcal{T}}_J}\right)_{i\in\N^*}$. So we have the result we want. \\
If $\kappa<1$, we will basically use the same method. First there exists $\gamma\in(\kappa,1]$ such that $\gamma<\frac{\kappa+\kappa^{\prime}}{2}$ and for every $j\in J,\ \gamma<\kappa_j$.\\
We have that:
\[
\limsup n^{-\frac{1}{\gamma}}\sum\limits_{k=0}^{\tau_{n}-1}1_{Y_k\in \tilde{\mathcal{T}}_J}
=\limsup n^{-\frac{1}{\gamma}}\sum\limits_{i=2}^{n-1}\sum\limits_{k=\tau_i}^{\tau_{i+1}-1}1_{Y_k\in \tilde{\mathcal{T}}_J}.
\]
And since:
\[
\left(n^{-\frac{1}{\gamma}}\sum\limits_{i=2}^{n}\sum\limits_{k=\tau_i}^{\tau_{i+1}-1}1_{Y_k\in \tilde{\mathcal{T}}_J}\right)^{\gamma}
\leq \frac{1}{n}\sum\limits_{j\in J}\sum\limits_{i=1}^{n}\left(\sum\limits_{k=\tau_i}^{\tau_{i+1}-1}1_{Y_k\in \tilde{\mathcal{T}}_J} \right)^{\gamma}
\]
we also have:
\[
\limsup n^{-\frac{1}{\gamma}}\sum\limits_{i=2}^{n}\sum\limits_{k=\tau_i}^{\tau_{i+1}-1}1_{Y_k\in \tilde{\mathcal{T}}_J}
\leq \left(\limsup \frac{1}{n}\sum\limits_{i=2}^{n}\left(\sum\limits_{k=\tau_i}^{\tau_{i+1}-1}1_{Y_k\in \tilde{\mathcal{T}}_J} \right)^{\gamma}\right)^{\frac{1}{\gamma}}.
\]
Now, since the random variables $\left(\left(\sum\limits_{k=\tau_{2i}}^{\tau_{2i}-1}1_{Y_k\in \tilde{\mathcal{T}}_J}\right)^{\gamma}\right)_{i\in\N^*}$ are iid and so are the random variables $\left(\left(\sum\limits_{k=\tau_{2i+1}}^{\tau_{2i+1}-1}1_{Y_k\in \tilde{\mathcal{T}}_J}\right)^{\gamma}\right)_{i\in\N^*}$ we have that there exists a constant $C_{\infty}\in [0,\infty]$ such that $\Prob_0$ almost surely:
\[
\frac{1}{n}\sum\limits_{i=2}^n\left(\sum\limits_{k=\tau_{i}}^{\tau_{i}-1}1_{Y_k\in \tilde{\mathcal{T}}_J}\right)^{\gamma}\rightarrow C_{\infty}.
\]
Now, by definition of the $C_j$ and since $(a+b)^{\gamma}\leq a^{\gamma}+b^{\gamma}$ we have that if $C_{\infty}=\infty$ then $\Prob_0$ almost surely:
\[
\frac{1}{n}\sum\limits_{j\in J}\sum\limits_{i=1}^{C_j n}\left(N_i^j\right)^{\gamma}\rightarrow \infty.
\]
However we have (using the same techniques and notations as in the case $\kappa=1$):
\[
\begin{aligned}
\E_{\prob_0^{\omega}}\left(\frac{1}{n}\sum\limits_{j\in J}\sum\limits_{i=1}^{C_j n}(\ell_i^j)^{\gamma}|\tilde{Y}\right)
=& \frac{1}{n}\sum\limits_{j\in J}\sum\limits_{i=1}^{C_j n}\E_{\prob_0^{\omega}}\left(\left(\sum\limits_{k=1}^{N_i^j}\varepsilon_i^j+2H_{i,k}^j\right)^{\gamma}|\tilde{Y}\right)\\
\leq & \frac{1}{n}\sum\limits_{j\in J}\sum\limits_{i=1}^{C_j n}\E_{\prob_0^{\omega}}\left(\sum\limits_{k=1}^{N_i^j}\varepsilon_i^j+2H_{i,k}^j|\tilde{Y}\right)^{\gamma}\\
\leq & \frac{1}{n}\sum\limits_{j\in J}\sum\limits_{i=1}^{C_j n}\left(N_i^j\frac{2}{p_i^j}\right)^{\gamma}\\
\leq & C \frac{1}{n}\sum\limits_{j\in J}\sum\limits_{i=1}^{C_j n}\left(N_i^js_i^j\right)^{\gamma}.
\end{aligned}
\]
Now by the same method as the one for $\kappa =1$, by using lemma \ref{lem:13} and lemma \ref{lem:21} we get:
\[
\E_{\Prob_0}\left(\frac{1}{n}\sum\limits_{j\in J}\sum\limits_{i=1}^{C_j n}(\ell_i^j)^{\gamma}\right)
\leq C. 
\]
This means that $C_{\infty}<\infty$ and therefore:
\[
\limsup n^{-\frac{1}{\gamma}}\sum\limits_{k=0}^{\tau_{n}-1}1_{Y_k\in \tilde{\mathcal{T}}_J} \leq (C_{\infty})^{\frac{1}{\gamma}}<\infty.
\]
\end{proof}

\begin{lem}\label{lem:12}
Let $A^{i_1,i_2}_{\varepsilon,n}$(i) be the event that the walk visits at least two trap of strength at least $\varepsilon n^{\frac{1}{\kappa}}$ between times $\tau_{i}$ and $\tau_{i+i_1}-1$ and that it enters these traps at most $i_2$ times. We have that for any $i_1\geq 1$:
\[
\Prob_0 \left( \bigcup\limits_{2\leq i \leq n} A^{i_1,i_2}_{\varepsilon,n}(i) \right) \rightarrow 0.
\]
\end{lem}
\begin{proof}
Let $\overline{\alpha}:=\sum\limits_{i=1}^{2d}\alpha_i$. Let $M(i)$ be the number of traps visited before time $\tau_i$. We know by lemma \ref{lem:9} that the number $M(i+i_1)-M(i)$ of traps visited between times $\tau_i$ and $\tau_{i+i_1}-1$ has a finite expectation (for $\Prob_0$) and by proposition \ref{prop:1} the $((M(2i+2)-M(2i+1))_{i\geq 1}$ are iid and so are the $(M(2i+1)-M(2i))_{i\geq 1}$. This means that there is a positive constant $C$ such that $\Prob_0$ almost surely:
\[
\frac{1}{n}M(n) \rightarrow C.
\]
Now let $M^{i_2}(i)$ be the number of traps visited at most $i_2$ times before time $\tau_{i}$. We know that:
\[
\Prob_0(M^{i_2}(n+i_1) \geq 2Cn)\rightarrow 0.
\]
Now, for any $\eta>0$ we have:
\[
\begin{aligned}
\Prob_0\left(\exists i\leq n, M(i+i_1)-M(i) \geq \eta n\right)
\leq & \sum\limits_{i\leq n} \Prob_0\left(M(i+i_1)-M(i) \geq \eta n\right) \\
=& o(1) + \sum\limits_{2\leq i\leq n} \Prob_0\left(M(i+i_1)-M(i) \geq \eta n\right) \\
=& o(1) + (n-1)\Prob_0\left(M(i+i_1)-M(i) \geq \eta n\right)\\
=& o(1) \text{\ \ since } M(i+i_1)-M(i)  \text{ has a finite expectation}.
\end{aligned}
\]
Now let $A_i$ be the event "the $i^{\text{th}}$ trap visited by the walk is of strength at least $\varepsilon n^{\frac{1}{\kappa}}$ and that the walk enters this trap at most $i_2$ times". We have:
\[
\begin{aligned}
&\Prob_0\left(\exists i\leq 2Cn , \exists j\leq \eta n, A_i \text{ and } A_{i+j}\right) \\
\leq & \Prob_0 \left(\exists i \leq \frac{2C}{\eta}, \exists j_1,j_2\in [\![ i \eta n, i\eta n + 2\eta n ]\!], j_1\not = j_2 \text{ and } A_{j_1} \text{ and } A_{j_2}\right)\\
\leq & \sum\limits_{i=0}^{\frac{2C}{\eta}} \Prob_0 \left(\exists j_1,j_2\in [\![ i \eta n, i\eta n + 2\eta n ]\!], j_1\not = j_2 \text{ and } A_{j_1} \text{ and } A_{j_2}\right)\\
\leq & \sum\limits_{i=0}^{\frac{2C}{\eta}}\sum\limits_{j_1=i \eta n}^{i\eta n + 2\eta n}\sum\limits_{j_2=i \eta n}^{i\eta n + 2\eta n} \Prob_0 \left(A_{j_1} \text{ and } A_{j_2}\right)1_{j_1\not=j_2}.
\end{aligned}
\]
Now let $(\tilde{Y}_n)_{n\in\N}$ be the partially forgotten walk, by lemma \ref{lem:11} if $s_j$ is the strength of the $j^{\text{th}}$ trap visited and $N_j$ is the number of times the walk enters the $j^{\text{th}}$ trap,  there exists a constant $D_j$ that only depends on its configuration such that for any $B>2$,
\[
\Prob_0\left(s_j\geq B|\tilde{Y},\tilde{\omega}\right) \leq D_j B^{-\kappa} \exp\left(\frac{5(N_i+2\overline{\alpha})}{B}\right).
\]
Let $D^{i_2}$ be the maximum value of $D_j \exp\left(\frac{5(Z_i+2\overline{\alpha})}{2}\right)$ we can get for configuration of traps entered at most $i_2$ times. We get that for any $j$:
\[
\Prob_0(s_j\geq B \text{ and } N_j\leq i_2|\tilde{Y},\tilde{\omega}) \leq D^{i_2} B^{-\kappa}.
\] 
We also know that the strength of the traps are independent, knowing the partially forgotten walk and the equivalence class of the environment for the trap-equivalent relation. Therefore we have, for any $\eta>0$:
\[
\begin{aligned}
& \sum\limits_{i=0}^{\frac{2C}{\eta}}\sum\limits_{j_1=i \eta n}^{i\eta n + 2\eta n}\sum\limits_{j_2=i \eta n}^{i\eta n + 2\eta n} \Prob_0 \left(A_{j_1} \text{ and } A_{j_2}\right)1_{j_1\not=j_2} \\
\leq & \sum\limits_{i=0}^{\frac{2C}{\eta}}\sum\limits_{j_1=i \eta n}^{i\eta n + 2\eta n}\sum\limits_{j_2=i \eta n}^{i\eta n + 2\eta n} (D^{i_2})^2 \left(\varepsilon n^{\frac{1}{\kappa}}\right)^{-2\kappa}\\
\leq & 2\frac{2C}{\eta} (\eta n)^2 (D^{i_2})^2 \varepsilon^{-2\kappa} n^{-2} \text{ for } \eta \text{ small enough}\\
=& 4C \eta (D^{i_2})^2 \varepsilon^{-2\kappa}.
\end{aligned}
\]
Now, by taking a sequence $(\eta_n)_{n\in\N^*}$ of positive reals such that $\eta_n \rightarrow 0$ and such that:
\[
\Prob_0\left(\exists i\leq n, M(i+i_1)-M(i) \geq \eta_n n\right)\rightarrow 0,
\]
we get:
\[
\begin{aligned}
 \Prob_0 \left( \bigcup\limits_{2\leq i \leq n} A^{i_1,i_2}_{\varepsilon,n}(i) \right)
\leq & \Prob_0 \left(M(n+i_1) \leq 2Cn\right) \text { or } \left(\exists i\leq n, M(i+i_1)-M(i) \geq \eta_n n\right)\\
&\ \ \ \ \ +\Prob_0 \left(\exists i\leq 2Cn , \exists j\leq \eta_n n, A_i +\Prob_0 A_{i+j}\right).
\end{aligned}
\]
Therefore:
\[
\Prob_0 \left( \bigcup\limits_{2\leq i \leq n} A^{i_1,i_2}_{\varepsilon,n}(i) \right) \rightarrow 0
\]
\end{proof}

\begin{lem}\label{lem:19}
If $\kappa=1$ there exists a constant $C$ such that $\Prob_0$ almost surely:
\[
\frac{1}{n}\sum\limits_{i=0}^{\tau_n-1}1_{Y_i\not\in\tilde{\mathcal{T}}}\rightarrow C.
\]
If $\kappa<1$, there exists a constant $C>0$ and a constant $\beta< \frac{1}{\kappa}$ such that $\Prob_0$ almost surely, for $n$ large enough:
\[
\sum\limits_{x\in\Z^d}\sum\limits_{i=0}^{\tau_n-1} 1_{Y_i=x}1_{x\not\in\tilde{\mathcal{T}}} \leq Cn^{\beta}.
\]
\end{lem}
\begin{proof}
Let $m$ be such that $\Qprob^m$ is well defined. Let $(t^m_i)_{i\in\N}$ be the times at which $X^m$ changes position, with $t_0:=0$. We have $X^m_{t^m_i}=Y_i$ for all $i\in\N$. Let $(\mathcal{E}_i)_{i\in\N}$ be a sequence of random variables defined by $\mathcal{E}_i=(t^m_{i+1}-t^m_i)\gamma^m_{\omega}(Y_i)$. By definition of $X$ and $Y$, $(\mathcal{E}_i)_{i\in\N} $ is a sequence of iid exponential random variables of parameter 1, independent of the walk and the environment.\\
\\
We will first look at the case $\kappa=1$. \\ 
If $\sum\limits_{i=\tau_2}^{\tau_3-1}1_{Y_i\not\in\tilde{\mathcal{T}}}$ has a finite expectation for $\Prob_0$, since the $\left(\sum\limits_{i=\tau_{2i}}^{\tau_{2i+1}-1}1_{Y_i\not\in\tilde{\mathcal{T}}}\right)_{i\in\N^*}$ are iid and so are the $\left(\sum\limits_{i=\tau_{2i+1}}^{\tau_{2i+2}-1}1_{Y_i\not\in\tilde{\mathcal{T}}}\right)_{i\in\N^*}$ then we have the result we want. On the other hand, if $\sum\limits_{i=\tau_1}^{\tau_2-1}1_{Y_i\not\in\tilde{\mathcal{T}}}$ has an infinite expectation then, since the random variables $\left(\sum\limits_{i=\tau_{i}}^{\tau_{i+1}-1}1_{Y_i\not\in\tilde{\mathcal{T}}}\right)_{i\geq 2}$ are non negative, $n^{-1}\sum\limits_{i=\tau_1}^{\tau_n-1}1_{Y_i\not\in\tilde{\mathcal{T}}}\rightarrow\infty$ $\Prob_0$ almost surely.\\
By the law of large number, we get that $\Prob_0$ almost surely:
\[
\exists k\in\N,\forall n\geq k, \sum\limits_{i=0}^{\tau_{n}-1}\mathcal{E}_i 1_{Y_i\not\in\tilde{\mathcal{T}}} \geq \frac{1}{2}\sum\limits_{i=0}^{\tau_{n}-1}1_{Y_i\not\in\tilde{\mathcal{T}}}.
\]
For any point $x$, if $x$ is not in a trap then, by definition of traps:
\[
\frac{1}{\gamma^{\omega}(x)}\geq \frac{1}{2}.
\]
This yields:
\[
\sum\limits_{i=0}^{\tau_{n}-1}\mathcal{E}_i 1_{Y_i\not\in\tilde{\mathcal{T}}}
\leq  2\sum\limits_{i=0}^{\tau_{n}-1}\mathcal{E}_i 1_{Y_i\not\in\tilde{\mathcal{T}}} \frac{1}{\gamma^{\omega}(Y_i)}.
\]
And by writing $T^m_n=t^m_{\tau_n}$ we have:
\[
\sum\limits_{i=0}^{\tau_{n}-1}\mathcal{E}_i 1_{Y_i\not\in\tilde{\mathcal{T}}}
\leq 2\int\limits_{0}^{T^m_n}\frac{\gamma^m_{\omega}(X^m_t)}{\gamma^{\omega}(X^m_t)}\dd t.
\]
We know by lemma \ref{lem:8} that there exists a constant $d_m$ such that $\Prob_0$ almost surely:
\[
T^m_n-d_m n\rightarrow -\infty.
\]
We get:
\[
\exists k\in\N,\forall n\geq k, \int\limits_{0}^{T^m_n}\frac{\gamma^m_{\omega}(X^m_t)}{\gamma^{\omega}(X^m_t)}\dd t \leq \int\limits_{0}^{d_mn}\frac{\gamma^m_{\omega}(X^m_t)}{\gamma^{\omega}(X^m_t)}\dd t.
\]
Finally, if $\Prob_0$ almost surely:
\[
\frac{1}{n}\sum\limits_{i=0}^{\tau_n}1_{Y_i\not\in\tilde{\mathcal{T}}} \rightarrow \infty.
\]
Then $\Prob_0$ almost surely:
\[
\frac{1}{n}\int\limits_{0}^{d_mn}\frac{\gamma^m_{\omega}(X^m_t)}{\gamma^{\omega}(X^m_t)}\dd t \rightarrow \infty.
\]
And therefore, since $\Qprob^m_0$ is absolutely continuous with respect to $\Prob_0$ we get that $\Qprob^m_0$ almost surely:
\[
\frac{1}{n}\int\limits_{0}^{d_mn}\frac{\gamma^m_{\omega}(X^m_t)}{\gamma^{\omega}(X^m_t)}\dd t \rightarrow \infty.
\]
So we would have, since $\frac{\gamma^m_{\omega}(X^m_t)}{\gamma^{\omega}(X^m_t)}$ is positive:
\[
\frac{1}{n}\E_{\Qprob^m_0}\left(\int\limits_{0}^{d_mn}\frac{\gamma^m_{\omega}(X^m_t)}{\gamma^{\omega}(X^m_t)}\dd t\right) \rightarrow \infty
\]
which would mean, since $\Qprob^m_0$ is a stationary law:
\[
\E_{\Qprob^m}\left(\int\limits_{0}^{1}\frac{\gamma^m_{\omega}(X^m_t)}{\gamma^{\omega}(X^m_t)}\dd t\right)=\infty.
\]
Which is false by lemma \ref{lem:15} so we get the result we want.\\
Now for the case $\kappa<1$. \\
Let $\beta \in \left(\kappa,\frac{\kappa+\kappa^{\prime}}{2}\right)$ be a real such that $\beta \leq 1$. If $\sum\limits_{x\in\Z^d}\left(\sum\limits_{i=\tau_2}^{\tau_3-1} 1_{Y_i=x}\right)^{\beta}1_{x\not\in\tilde{\mathcal{T}}}$ has an infinite expectation (for $\Prob_0$), since the $\left(\sum\limits_{x\in\Z^d}\left(\sum\limits_{i=\tau_{2j}}^{\tau_{2j+1}-1} 1_{Y_i=x}\right)^{\beta}1_{x\not\in\tilde{\mathcal{T}}}\right)_{j\in\N^*}$ are iid, we would have that $\Prob_0$ almost surely:
\[
n^{-1}\sum\limits_{x\in\Z^d}\left(\sum\limits_{i=0}^{\tau_n-1} 1_{Y_i=x}\right)^{\beta}1_{x\not\in\tilde{\mathcal{T}}}\rightarrow\infty.
\]
\\
By lemma \ref{lem:30} we get that there exists a constant $C>0$ such that $\Prob_0$ almost surely:
\[
\exists m\in\N,\forall n\geq m, \sum\limits_{x\in\Z^d}\left(\sum\limits_{i=0}^{\tau_{n+1}-1}\mathcal{E}_i1_{Y_i=x}\right)^{\beta} 1_{x\not\in\tilde{\mathcal{T}}} 
\geq C \sum\limits_{x\in\Z^d}\left(\sum\limits_{i=0}^{\tau_{n+1}-1}1_{Y_i=x}\right)^{\beta}1_{x \not\in\tilde{\mathcal{T}}}.
\]
We also have, by writing $T^m_n=t^m_{\tau_n}$:
\[
\begin{aligned}
\sum\limits_{x\in\Z^d}\left(\sum\limits_{i=0}^{\tau_{n+1}-1}\mathcal{E}_i1_{Y_i=x}\right)^{\beta}1_{x \not\in\tilde{\mathcal{T}}}
\leq & 4^{\beta} \sum\limits_{x\in\Z^d}\left(\sum\limits_{i=0}^{\tau_{n+1}-1}\mathcal{E}_i1_{Y_i=x}\frac{1}{\gamma^{\omega}(x)}\right)^{\beta}1_{x \not\in\tilde{\mathcal{T}}}\\
\leq & 4^{\beta}\sum\limits_{x\in\Z^d}\left(\int\limits_{0}^{T^m_n}\frac{\gamma^m_{\omega}(X^m_t)}{\gamma^{\omega}(X^m_t)}1_{X^m_t=x}\dd t\right)^{\beta}.
\end{aligned}
\]
We know by lemma \ref{lem:8} that there exists a constant $d_m$ such that $\Prob_0$ almost surely:
\[
T^m_n-d_m n\rightarrow -\infty.
\]
We get:
\[
\exists m\in\N,\forall n\geq m, 
\sum\limits_{x\in\Z^d}\left(\int\limits_{0}^{T^m_n}\frac{\gamma^m_{\omega}(X^m_t)}{\gamma^{\omega}(X^m_t)}1_{X^m_t=x}\dd t\right)^{\beta} 
\leq \sum\limits_{x\in\Z^d}\left(\int\limits_{0}^{d_m n}\frac{\gamma^m_{\omega}(X^m_t)}{\gamma^{\omega}(X^m_t)}1_{X^m_t=x}\dd t\right)^{\beta}.
\]
Finally, if $\Prob_0$ almost surely
\[
\frac{1}{n}\sum\limits_{x\in\Z^d}\left(\sum\limits_{i=0}^{\tau_n-1} 1_{Y_i=x}\right)^{\beta}1_{x\not\in\tilde{\mathcal{T}}}\rightarrow\infty
\]
then $\Prob_0$ almost surely
\[
\frac{1}{n}\sum\limits_{x\in\Z^d}\left(\int\limits_{0}^{d_mn}\frac{\gamma^m_{\omega}(X^m_t)}{\gamma^{\omega}(X^m_t)}1_{X^m_t=x}\dd t\right)^{\beta}\rightarrow\infty.
\]
And therefore, since $\Qprob^m_0$ is absolutely continuous with respect to $\Prob_0$ we get that $\Qprob^m_0$ almost surely:
\[
\frac{1}{n}\sum\limits_{x\in\Z^d}\left(\int\limits_{0}^{d_m n}\frac{\gamma^m_{\omega}(X^m_t)}{\gamma^{\omega}(X^m_t)}1_{X^m_t=x}\dd t\right)^{\beta}\rightarrow\infty.
\]
So we would have:
\[
\frac{1}{n}\E_{\Qprob^m_0}\left(\sum\limits_{x\in\Z^d}\left(\int\limits_{0}^{d_m n}\frac{\gamma^m_{\omega}(X^m_t)}{\gamma^{\omega}(X^m_t)}1_{X^m_t=x}\dd t\right)^{\beta}\right) \rightarrow \infty.
\]
And therefore:
\[
\frac{1}{n}\sum\limits_{i=0}^{d_m n}\E_{\Qprob^m_0}\left(\sum\limits_{x\in\Z^d}\left(\int\limits_{i}^{i+1}\frac{\gamma^m_{\omega}(X^m_t)}{\gamma^{\omega}(X^m_t)}1_{X^m_t=x}\dd t\right)^{\beta}\right) \rightarrow \infty.
\]
This would mean, since $\Qprob^m_0$ is a stationary law that
\[
\E_{\Qprob^m_0}\left(\sum\limits_{x\in\Z^d}\left(\int\limits_{0}^{1}\frac{\gamma^m_{\omega}(X^m_t)}{\gamma^{\omega}(X^m_t)}1_{X^m_t=x}\dd t\right)^{\beta}\right)=\infty
\]
which is false by lemma \ref{lem:15}. Therefore there exists a constant $C>0$ such that $\Prob_0$ almost surely:
\[
\frac{1}{n}\sum\limits_{x\in\Z^d}\left(\sum\limits_{i=0}^{\tau_n-1} 1_{Y_i=x}\right)^{\beta}1_{x\not\in\tilde{\mathcal{T}}}\rightarrow C.
\]
So $\Prob_0$ almost surely for $n$ large enough:
\[
\frac{1}{n}\left(\sum\limits_{x\in\Z^d}\sum\limits_{i=0}^{\tau_n-1} 1_{Y_i=x}1_{x\not\in\tilde{\mathcal{T}}}\right)^{\beta}
\leq \frac{1}{n}\sum\limits_{x\in\Z^d}\left(\sum\limits_{i=0}^{\tau_n-1} 1_{Y_i=x}\right)^{\beta}1_{x\not\in\tilde{\mathcal{T}}} 
\leq 2C.
\]
And therefore:
\[
\sum\limits_{x\in\Z^d}\sum\limits_{i=0}^{\tau_n-1} 1_{Y_i=x}1_{x\not\in\tilde{\mathcal{T}}} \leq (2Cn)^{\frac{1}{\beta}}.
\]
\end{proof}

\subsection{Proof of the theorems}

Now we can finally prove both theorems.
\begin{theo}
Set $d\geq 3$ and $\alpha\in(0,\infty)^{2d}$. Let $Y^n(t)$ be defined by:
\[
Y^n(t)=n^{-\kappa}Y_{\lfloor nt \rfloor}.
\]
If $\kappa<1$ and $d_{\alpha}\not=0$, there exists positive constants $c_1,c_2,c_3$ such that for the $J_1$ topology and for $\Prob^{(\alpha)}_0$:
\[
\left(t\rightarrow n^{-\frac{1}{\kappa}}\tau_{\lfloor nt \rfloor}\right)\rightarrow c_1 \mathcal{S}^{\kappa},
\]
for the $M_1$ topology and for $\Prob^{(\alpha)}_0$:
\[
\left(t\rightarrow n^{-\frac{1}{\kappa}}\inf\{t\geq 0, Y(t).e_1\geq nx\}\right) \rightarrow c_2 \mathcal{S}^{\kappa}
\]
and for the $J_1$ topology and for $\Prob^{(\alpha)}_0$:
\[
Y^n\rightarrow c_3 \tilde{\mathcal{S}}^{\kappa}d_{\alpha}.
\]
\end{theo}
\begin{proof}
The proof will be divided in three parts, one for each result. The second part and the third one rely on the first part. However, the second part and the third part are independent from one another. \\

First Part\\
\\
First we will prove that there exists a constant $c$ such that for any $t\in\R^+$ and any increasing sequence $(x_n)$ such that $x_n \rightarrow \infty$, we have the following convergence in law, for $\Prob_0$:
\[
x_n^{-\frac{1}{\kappa}}\tau_{\lfloor x_nt \rfloor}\rightarrow c t^{\frac{1}{\kappa}} \mathcal{S}_1^{\kappa}.
\] 
The result is obvious for $t=0$. For $t>0$, lemmas \ref{lem:19} and \ref{lem:20} tell us that we only have to consider the time spent in traps in directions $j$ such that $\kappa_j=\kappa$. Then lemma \ref{lem:16} tells us that with probability larger than $1-\varepsilon$ the time spent in such traps is not more than the time spent in traps where the walks come back at most $m_{\varepsilon}$ times (for some $m_{\varepsilon}$) plus at most $\varepsilon x_n^{\frac{1}{\kappa}}$. We also know by lemma \ref{lem:18} that for any $m_\varepsilon$ there exists a constant $c_{\varepsilon}$ such that the time spent in traps where the walks come back at most $m_{\varepsilon}$ times renormalized by $x_n^{-\frac{1}{\kappa}}$ converges in law (for $\Prob_0$) to $c_{\varepsilon} t^{\frac{1}{\kappa}} S^{\kappa}$ so we get the result we want by having $\varepsilon$ go to $0$ since $c_{\varepsilon}$ is increasing and cannot go to infinity. Since the $(\tau_{i+1}-\tau_{i})_{i\geq 1}$ are iid (for $\Prob_0$) by proposition $\ref{prop:1}$, we also get that for any sequence $(n_i)_{i\in\N^*}$ with $n_i\geq 1$, $\left(i^{-\frac{1}{\kappa}}(\tau_{n_i+it}-\tau_{n_i})\right)_{i\geq 1}$ converges in law (for $\Prob_0$) to $c_1 t^{\frac{1}{\kappa}} \mathcal{S}_1^{\kappa}$.\\
Now we want to show that the family of process $\left(t\rightarrow x_n^{-\frac{1}{\kappa}}\tau_{\lfloor x_n t \rfloor}\right)_{n\in\N}$ is tight. We will only look at the convergence and tightness for the processes on an interval $[0,A]$. We use the characterisation given in Theorem 15.3 of $\cite{Billingsley}$:
\[
\begin{aligned}
(i)&\text{ for each positive } \varepsilon \text{ there exists } c \text{ such that:} \\
&\Prob\left(\sup\limits_{t\in [0,T]} |f(t)|>c \right)\leq\varepsilon, \\
(ii)&\text{ for each } \varepsilon>0 \text{ and } \eta>0 \text{, there exist a } \delta,\ 0<\delta<T \text{, and an integer } n_0 \text{ such that:}\\
& \forall n\geq n_0,\ \Prob(w_{f_n}(\delta)\geq  \eta)\leq \varepsilon \\
&\text{ and }\\
&  \forall n\geq n_0,\ \Prob(v_{f_n}(0,\delta)\geq  \eta)\leq \varepsilon \text{ and } 
\Prob(v_{f_n}(T,\delta)\geq  \eta)\leq \varepsilon,
\end{aligned}
\]
where $w_f$ and $v_f$ are defined by:
\[
\begin{aligned}
w_f(\delta) &= \sup\{\min \left(|f(t)-f(t_1)|,|f(t_2)-f(t)|\right) t_1\leq t\leq t_2 \leq T, t_2-t_1 \leq \delta\},\\
v_f(t,\delta) &= \sup\{|f(t_1)-f(t_2)|: t_1,t_2\in[0,T]\cap(t-\delta,t+\delta)\}.
\end{aligned}
\]
For a sequence of non-decreasing processes $(W_n)$ defined on $[0,T]$, this characterization is implied by the following: \\
\[
\begin{aligned}
(i) & \text{  for each positive } \varepsilon \text{ there exist } C \text{ such that} \\
& \Prob(W_n(T) \geq C)\leq \varepsilon, \text{      for } n\geq 1, \\
(ii) & \text{  for each } \varepsilon>0 \, \text{ there exist a } \delta\in (0,T),\text{ such that for } n\geq 1 \:\\
(a)& \ \ \forall x\in [\delta,T-\delta],\ \Prob(W_n(x+\delta)-W_n(x)\geq \varepsilon \text{ and } W_n(x)-W_n(x-\delta)\geq \varepsilon) \leq \varepsilon\\
& \text{and} \\
(b)& \ \ \Prob(W_n(\delta)-W_n(0)\geq \varepsilon)\leq \varepsilon\\
& \text{and} \\
(c)& \ \ \Prob(W_n(T)-W_n(T-\delta) \geq \varepsilon) \leq \varepsilon.
\end{aligned}
\]
For the first property, since we know that the sequence $\left(x_n^{-\frac{1}{\kappa}}\tau_{\lfloor x_n A \rfloor}\right)_{n\in\N}$ converges in law for $\Prob_0$, the family $\left(x_n^{-\frac{1}{\kappa}}\tau_{\lfloor x_n A \rfloor}\right)_{n\in\N}$ is tight and therefore for any $\varepsilon>0$ there exists $B_{\varepsilon}$ such that:
\[
\forall n\in\N,\ \Prob_0\left(x_n^{-\frac{1}{\kappa}}\tau_{\lfloor x_nA \rfloor}\in [0,B_{\varepsilon}] \right)\geq 1-\varepsilon.
\]
So:
\[
\forall \varepsilon>0,\exists B_{\varepsilon}, \forall n\in\N, \Prob_0\left(\forall t\in [0,A],\ x_n^{-\frac{1}{\kappa}}\tau_{\lfloor x_nt \rfloor}\in [0,B_{\varepsilon}] \right)\geq 1-\varepsilon.
\]
Now we will prove the two side conditions ($ii.b$ and $ii.c$). For $(ii.b)$, we first choose $\delta$ such that $\Prob_0\left(c_1 \delta^{\frac{1}{\kappa}}\mathcal{S}_1^{\kappa}\geq \varepsilon \right) \leq \frac{\varepsilon}{2}$. This proves the result for $n$ large enough and then, since the processes we consider are c\`adl\`ag, we decrease $\delta$ up to the point where we have the result for $n$ small and we get the result we want.\\
\\
For $(ii.c)$, the proof will be essentially the same. Since the increments are iid (except for the first one of which we do not know the law) the law of $x_n^{-\frac{1}{\kappa}}\tau_{\lfloor x_nA \rfloor}-x_n^{-\frac{1}{\kappa}}\tau_{\lfloor x_n(A-\delta) \rfloor}$ converges to $c_1\delta^{\frac{1}{\kappa}}S_{\kappa}$. So we get that for some $\delta$, for $n$ large enough we have the result we want. For small $n$ we only use the fact that the processes are c\`adl\`ag so we get the result we want by decreasing $\delta$.\\
\\
Now we can prove $(ii.a)$. Let $J=\{j\in[\![1,d]\!], \kappa_j=\kappa \}$. First we have, by lemmas \ref{lem:19} and \ref{lem:20}, that for $n$ large enough, the time spent in vertices that are not part of a trap in a direction $j\in J$ before time $\tau_{\lfloor x_nt \rfloor}$ is smaller than $\frac{1}{3}\varepsilon x_n^{\frac{1}{\kappa}}$ with probability at least $1-\frac{1}{3}\varepsilon$. Similarly by lemma \ref{lem:18} there exists $m_{\varepsilon}$ such that for $n$ large enough the time spent in traps in direction $j\in J$ such that the walk enters at least $m_{\varepsilon}$ times the trap is lower than $\frac{1}{3}\varepsilon x_n^{\frac{1}{\kappa}}$ with probability at least $1-\frac{1}{3}\varepsilon$. And finally, there exists $\beta_{\varepsilon}$ such that for $n$ large enough, by lemma \ref{lem:31}, with probability at least $1-\frac{1}{3}\varepsilon$ the time spent in traps in direction $j\in J$ such that their strength is at most $\beta_{\varepsilon} x_n^{\frac{1}{\kappa}}$ is lower than $\frac{1}{3}\varepsilon  x_n^{\frac{1}{\kappa}}$. Condition $(ii.c)$ is not verified if either of the previous three events are not verified which happens with probability at most $1-\varepsilon$. However if the previous events are verified and there is no $i$ such that there are at least two traps of strength at least $\beta_{\varepsilon} x_n^{\frac{1}{\kappa}}$ visited at most $m_{\varepsilon}$ times between times $\tau_i$ and $\tau_{i+2\delta x_n}-1$ then the main condition is true.\\ 

So now we just have to prove that for $\delta$ small enough, with high probability there is no $i$ such that there are at least two traps of strength at least $\beta_{\varepsilon} x_n^{\frac{1}{\kappa}}$ visited at most $m_{\varepsilon}$ times between times $\tau_i$ and $\tau_{i+2\delta x_n}-1$. By lemma \ref{lem:12} we have that for any $m\in\N$ the probability that there exists $i\leq x_n$ such that there are two traps of strength at least $\beta_{\varepsilon} x_n^{\frac{1}{\kappa}}$ between times $\tau_i$ and $\tau_{i+m}-1$ goes to $0$ when $n$ goes to infinity. So let $B_i$ be the event: "there exists a trap of strength at least $\beta_{\varepsilon} x_n^{\frac{1}{\kappa}}$ visited at most $m_{\varepsilon}$ times between times $\tau_i$ and $\tau_{i+1}-1$". We define the finite sequence $(n_i)$ by:
\[
\begin{aligned}
n_1 =& \inf \{ j \geq 1, B_j\}, \\
n_{i+1}=& \inf \{ j\geq n_i +m, B_j\}.
\end{aligned}
\]
We also define $ \tilde{n}_i$ by $ \tilde{n}_i= \sup \{j, n_j\leq x_i\}$. First we want to prove that $\tilde{n}_i$ cannot be too large. We know that there exists a constant $C$ such that if $M(n)$ is the number of different traps in a direction $j$ visited before time $\tau_{n}$ then for $n$ large enough: $\Prob_0(M(x_n) \geq C x_n) \leq \varepsilon$ and by lemma \ref{lem:13} we clearly have that $\E(\tilde{n}_n1_{M(x_n) \leq C x_n}) \leq \frac{c C}{\beta^{\kappa}}$. Therefore if we take $B\geq \frac{c C}{\varepsilon \beta^{\kappa}}$ we get that for $n$ large enough, $\Prob_0(\tilde{n}_n\geq B) \leq 2\varepsilon$.
Now we want to show that for $\delta >0$ small enough, $\Prob_0(\exists i\leq B, n_{i+1}-n_i\leq 2\delta x_n) \leq \varepsilon$ which would yields the desired result. For any $i$, we have, by proposition \ref{prop:1}:
\[
\Prob_0(n_{i+1}-n_i \leq 2\delta x_n) \leq \Prob_0 (n_1 \leq 2\delta x_n).
\]
And therefore:
\[
\Prob_0(\exists i \leq \tilde{n}_n,\ n_{i+1}-n_i \leq 2\delta x_n ) \leq \Prob_0(\tilde{n}_n>B)+B\Prob_0 (n_1 \leq 2\delta x_n).
\]
We have that there is a constant $C$ such that for $n$ large enough, $\Prob_0(M(2\delta x_n)\geq 2C\delta x_n)\leq \frac{\varepsilon}{B}$. And then by lemma \ref{lem:13} we have that the expectation of the number of traps of strength at least $\beta x_n^{\frac{1}{\kappa}}$ among the first $2 \delta x_n$ traps is lower than $2 \delta x_n \frac{c}{\beta^{\kappa} x_n}$ and therefore for $\delta$ small enough, $\Prob_0(\exists i \leq \tilde{n}_n,\ n_{i+1}-n_i \leq 2\delta x_n )\leq \varepsilon$. So we have that the sequence of processes is tight.\\
Now we want to show that its limit is $c_1\mathcal{S}^{\kappa}$. Let $m$ be an integer and  $(x_i)_{0\leq i \leq n}$ be reals such that $0=y_0<y_1<\dots<y_{m-1}<y_m=1$. We have, since the $(\tau_{i+1}-\tau_i)_{i\geq 1}$ are iid and independent from $\tau_1$:
\[
(x_n^{-\frac{1}{\kappa}}\tau_{\lfloor x_n y_i \rfloor})_{0\leq i \leq m} \rightarrow (\mathcal{S}^{\kappa}(y_i))_{0\leq i \leq m}.
\]
So we have convergence in the $J_1$ topology for any increasing sequence $x_i$ that goes to infinity.\\

Second Part\\
\\
Let $L$ be defined by:
\[
L(t):=\inf \{n,Y_n.e_1\geq t \}.
\]
And let $L_n$ be the renormalized $L$:
\[
L_n(t)=n^{-\frac{1}{\kappa}}L(nt).
\]
We have, by definition of $\tau$ and $L$:
\[
\forall n\in\N^*,\ L(Y_{\tau_n}.e_1)=\tau_n.
\]
We first want to show that the sequence $L_n$ is tight in the $M_1$ topology. We use the characterisation given in Theorem 12.12.3 of $\cite{Whitt}$:
\[
\begin{aligned}
(i)&\text{ for each positive } \varepsilon \text{ there exists } c \text{ such that:} \\
&\Prob\left(\sup\limits_{t\in [0,T]} |f(t)|>c \right)\leq\varepsilon, \\
(ii)&\text{ for each } \varepsilon>0 \text{ and } \eta>0 \text{, there exist a } \delta,\ 0<\delta<T \text{, and an integer } n_0 \text{ such that:}\\
& \forall n\geq n_0,\ \Prob(w_{f_n}(\delta)\geq  \eta)\leq \varepsilon \\
&\text{ and }\\
&  \forall n\geq n_0,\ \Prob(v_{f_n}(0,\delta)\geq  \eta)\leq \varepsilon \text{ and } 
\Prob(v_{f_n}(T,\delta)\geq  \eta)\leq \varepsilon.
\end{aligned}
\]
Where $w_f$ and $v_f$ are defined by:
\[
\begin{aligned}
w_f(\delta) &= \sup\{\inf\limits_{\alpha\in [0,1]} |f(t)-(\alpha f(t_1)+ (1-\alpha)f(t_2))|, t_1\leq t\leq t_2 \leq T, t_2-t_1 \leq \delta\},\\
v_f(t,\delta) &= \sup\{|f(t_1)-f(t_2)|: t_1,t_2\in[0,T]\cap(t-\delta,t+\delta)\}.
\end{aligned}
\]
First we have:
\[
\begin{aligned}
\Prob_0\left(\sup\limits_{t\in [0,T]} |L_n(t)|>c \right)
=& \Prob_0\left( L(nT)>cn^{\frac{1}{\kappa}} \right) \\
\leq & \Prob_0 \left( \tau_{nT}>cn^{\frac{1}{\kappa}} \right),
\end{aligned}
\]
which is smaller than $\varepsilon$ for all $n$, for $c$ large enough.\\
Next, since $H_n$ is non-decreasing, we have:
\[
\Prob_0(w_{L_n}(\delta)=0)=1.
\]
Then, we first use the fact that:
\[
v_{L_n}(0,\delta)\leq n^{-\frac{1}{\kappa}}\tau_{n\delta}
\]
to get that for $\delta$ small enough:
\[
\forall n\geq n_0,\ \Prob_0(v_{L_n}(0,\delta)\geq \eta)\leq \varepsilon. 
\]
The bound for $v_{L_n}(T,\delta)$ is similar but slightly trickier. We know that for $c=\left(\E(Y_{\tau_2}-Y_{\tau_1}).e_1\right)^{-1}$, $\Prob_0$ almost surely:
\[
\frac{1}{n}(Y_{\tau_{cn(T-2\delta)}}.e_1,Y_{\tau_{cn(T+\delta)}}.e_1)\rightarrow (T-2\delta,T+\delta).
\]
Therefore, using the fact that $L_n$ is increasing, with probability going to 1:
\[
L_n(T)-L_n(T-\delta)\leq n^{-\frac{1}{\kappa}}(\tau_{cn(T+\delta)}-\tau_{cn(T-2\delta)}). 
\]
And we have the result we want for $\delta$ small enough and $n$ large enough. So we have that the sequence $(L_n)_{n\in\N^*}$ is tight. Now we just have to show that its limit is $C\mathcal{S}^{\kappa}$ for some constant $C$. Set $c=\left(\E(Y_{\tau_2}-Y_{\tau_1}).e_1\right)^{-1}$. We will show that $L_n(x)$ is almost equal to $\tau_n(cx)$ which will yield the result. Set $\varepsilon>0$ and $x\in [0,\infty)$. We want to show that $\Prob_0(|L_n(x)-\tau_n(cx)|\geq \varepsilon)\rightarrow 0$. We will use the following inequality:
\[
\Prob_0(L_n(x)-\tau_n(cx)\geq \varepsilon)
\leq \inf\limits_{\delta>0} \Prob_0(L_n(x)\geq \tau_n(cx+\delta)) + \Prob_0(\tau_n(cx+\delta)-\tau_n(cx)\geq \varepsilon).
\]
We clearly have, for any $\delta>0$
\[
\limsup\limits_{n\rightarrow\infty}\Prob_0(L_n(x)\geq \tau_n(cx+\delta)) = 0.
\]
And for some constant $\tilde{C}$ that does not depend on $x$ or $c$
\[
\Prob_0(\tau_n(cx+\delta)-\tau_n(cx)\geq \varepsilon)\rightarrow \Prob_0(\tilde{C}\mathcal{S}^{\kappa}(\delta)\geq \varepsilon).
\]
Therefore 
\[
\Prob_0(L_n(x)-\tau_n(cx)\geq \varepsilon)\rightarrow 0.
\]
Similarly we get:
\[
\Prob_0(L_n(x)-\tau_n(cx)\leq -\varepsilon)\rightarrow 0.
\]
Therefore the limit of $L^n$ is $t\rightarrow \tilde{C}\mathcal{S}^{\kappa}(ct)$ which is equal to $C\mathcal{S}^{\kappa}$ for some constant $C$. \\

Third Part\\
\\
We will look at a sequence of processes $t\rightarrow \overline{\tau}_n(t)$ such that the law of $\overline{\tau}_n$ is the same as that of $t\rightarrow x_n^{-\frac{1}{\kappa}} \tau_{\lfloor x_n t \rfloor}$ and such that almost surely $\overline{\tau}_n \rightarrow \overline{\tau}$ in the $J_1$ topology with the law of $\overline{\tau}$ being that of $\mathcal{S}^{\kappa}$. We want to show that the law of the inverse of $\overline{\tau}_n$ converges to that of the inverse of $\mathcal{S}^{\kappa}$. This is a direct consequence of lemmas \ref{lem:28} and \ref{lem:29}. Now if we define $L^{\tau}(t)$ by $L^{\tau}(t)=\min\{n\in\N,\tau_n \geq t\}$, we have that in $J_1$ topology:
\[
\frac{1}{x_n}L^{\tau}\left(x_n^{\frac{1}{\kappa}}t\right)\rightarrow \tilde{\mathcal{S}}^{\kappa}(t)
\]
for any increasing sequence $x_n$ such that $x_n\rightarrow \infty$. Therefore, for any increasing sequence $x_n$ such that $x_n\rightarrow \infty$:
\[
\frac{1}{x_n^{\kappa}}L^{\tau}(x_n t)\rightarrow \tilde{\mathcal{S}}^{\kappa}(t).
\]
Now by lemma \ref{lem:10} there exists $v\in\R^d$ such that $\Prob_0$ almost surely:
\[
\frac{Y_{\tau_{\lfloor t \rfloor}}}{t} \rightarrow v.
\]  
This means that in the $J_1$ topology, we have the following convergence (in law):
\[
\left(t \rightarrow \frac{Y_{\tau_{\lfloor x_n t\rfloor}}}{x_n }\right)
\rightarrow \left( t\rightarrow tv\right).
\]
And therefore, in the $J_1$ topology, 
\[
\left(t\rightarrow x_n^{-\frac{1}{\kappa}} \tau_{\lfloor x_n t \rfloor}, t \rightarrow \frac{Y_{\tau_{\lfloor x_nt \rfloor}}}{ x_n}\right)
\rightarrow \left( c_1\mathcal{S}^{\kappa},t\rightarrow tv\right).
\]
Now we will look at $\left(\overline{\tau}_n,d_n\right)$ where for any $n$ the law of $(\overline{\tau}_n,d_n)$ is the same as the law of $t\rightarrow x_n^{-\frac{1}{\kappa}} \tau_{\lfloor x_n t \rfloor}, t \rightarrow \frac{Y_{\tau_{\lfloor x_nt \rfloor}}}{x_n}$ and such that almost surely:
\[
\left( \overline{\tau}_n,d_n \right)
\rightarrow \left( c_1\mathcal{S}^{\kappa},t\rightarrow tv\right).
\]
Let $\overline{\tau}$ be such that almost surely $\overline{\tau}_n \rightarrow \overline{\tau}$. Let $\Delta_{[0,A]}$ be the distance associated with the infinite norm on [0,A].  \\
If we look at $d_{\overline{\tau}_n^{-1}(t)}$ where $\overline{\tau}_n^{-1}(t)=\inf\{x,\overline{\tau}_n(x)\geq t\}$ we get:
\[
\begin{aligned}
\Delta_{[0,A]}\left( d_n(\overline{\tau}_n^{-1}(t)), \overline{\tau}^{-1}(t)v \right)
& \leq \Delta_{[0,A]}\left( d_n(\overline{\tau}_n^{-1}(t), \overline{\tau}_n^{-1}(t)v\right) +  \Delta_{[0,A]}\left(\overline{\tau}_n^{-1}(t)v, \overline{\tau}^{-1}(t)v \right) \\
& = \Delta_{[0,A]}\left( d_n(\overline{\tau}_n^{-1}(t)), \overline{\tau}_n^{-1}(t)v\right) + ||v|| \Delta_{[0,A]}\left(\overline{\tau}_n^{-1}(t), \overline{\tau}^{-1}(t) \right).
\end{aligned}
\]
So for any $B,\varepsilon>0$:
\[
\begin{aligned}
&\Prob_0(\Delta_{[0,A]}\left( d_n(\overline{\tau}_n^{-1}(t)), \overline{\tau}^{-1}(t)v \right) \geq \varepsilon)\\
\leq & \Prob_0(\overline{\tau}_n^{-1}(A)>B) + \Prob_0\left( \exists t\in [0,B], ||d_n(t)-tv||\geq \frac{\varepsilon}{2} \right) +
\Prob_0\left(\Delta_{[0,A]}\left(\overline{\tau}_n^{-1}(t), \overline{\tau}^{-1}(t) \right)\geq \frac{\varepsilon}{2}\right) \\
=& \Prob_0(\overline{\tau}_n^{-1}(A)>B) +o(1) \\
=& \Prob_0 (\overline{\tau}_n(B) < A) +o(1) \\
=& \Prob_0 (\overline{\tau}(B) < A) +o(1).
\end{aligned}
\]
We clearly have that when $B$ goes to infinity, $\Prob_0 (\overline{\tau}(B) < A)$ goes to $0$ so we have that in the $J_1$ topology:
\[
d_n(\overline{\tau}_n^{-1}(t)) \rightarrow \overline{\tau}^{-1}(t)v.
\] 
Since we have that in law (in the following we will write $\tau(x)$ instead of $\tau_x$ for the formulas to stay readable):
\[
d_n(\overline{\tau}_n^{-1}(t))=\frac{1}{x_n}Y_{\tau\left(\lfloor x_n ( x_n^{-1} L^{\tau}((x_n)^{\frac{1}{\kappa}}t)) \rfloor\right)} = \frac{1}{x_n} Y_{\tau\left(\lfloor \overline{\tau}((x_n)^{\frac{1}{\kappa}}t) \rfloor\right)}
\]
we get that in the $J_1$ topology for any increasing sequence $x_n$:
\[
x_n^{-\kappa} Y_{\tau\left(\lfloor L^{\tau}(x_n t) \rfloor\right)} \rightarrow c_1^{-\kappa} \tilde{\mathcal{S}}^{\kappa}(t) v.
\]
Now we only have to show that $Y_{\tau\left(\lfloor L^{\tau}(x_n t) \rfloor\right)}$ and $Y_t$ are almost equal. For every $i>0$ let $R_i$ be the number of different points visited between times $\tau_i$ and $\tau_{i+1}-1$ and let $R_0$ be the number of different points visited before time $\tau_i -1$ (0 if $\tau_i=0$). The $(R_i)_{i\in\N}$ are independent and the $(R_i)_{i\in\N^*}$ are iid with finite expectation by lemma \ref{lem:9}. Let $\varepsilon>0$ be a constant and let $B>0$ be such that for $x$ large enough, $\Prob_0(x^{-\kappa}L^{\tau}(xA) \geq B) \leq \frac{\varepsilon}{2}$ (taking $B$ such that $\Prob_0(c_1^{-\kappa}\tilde{\mathcal{S}}^{\kappa}(A)\geq B)\leq \frac{\varepsilon}{4}$ works). We get that for $x$ large enough:
\[
\begin{aligned}
\Prob_0(\exists t \leq x A, x^{-\kappa}||Y_{\tau\left(\lfloor L^{\tau}( t) \rfloor\right)}-Y_t || \geq \varepsilon)
\leq &\frac{\varepsilon}{2} + \Prob_0(\exists i \leq B x^{\kappa}, R_i \geq \varepsilon x^{\kappa})\\
\leq & \frac{\varepsilon}{2} +\Prob_0(R_0 \geq \varepsilon x^{\kappa}) + \Prob_0(\exists i \in [\![1, B x^{\kappa} ]\!] , R_i \geq \varepsilon x^{\kappa})  \\
\leq & \frac{\varepsilon}{2} + o(1) + B x^{\kappa} \Prob_0(R_1 \geq \varepsilon x^{\kappa}) \\
= & \frac{\varepsilon}{2} + o(1) .
\end{aligned}
\]
So for any $\varepsilon>0$ we have that for $x$ large enough:
\[
\Prob_0(\exists t \leq x A, x^{-\kappa}||Y_{\tau\left(\lfloor L^{\tau}(x t) \rfloor\right)}-Y_t || \geq \varepsilon) \leq \varepsilon.
\]
So we get that in the $J_1$ topology:
\[
x^{-\kappa}Y_{\lfloor xt \rfloor} \rightarrow  \tilde{\mathcal{S}}^{\kappa}(t) v.
\]
Since $v$ and $d_{\alpha}$ are collinear, we get the result we want. 
\end{proof}
\begin{theo}
If $d\geq 3$ and $\kappa=1$, there exists positive constants $c_1,c_2,c_3$ such that we have the following convergences in probability (for $\Prob_0$):
\[
\frac{1}{n\log(n)}\tau_n \rightarrow c_1,
\]
\[
\frac{1}{n\log(n)}\inf\{i,Y_i.e_1\geq n\}\rightarrow c_2,
\]
\[
\frac{\log(n)}{n}(Y_n)\rightarrow c_3 d_{\alpha}.
\]
\end{theo}
\begin{proof}
Let $J=\{j\in[\![1,d]\!], \kappa_j = \kappa \}$. \\
By lemma \ref{lem:19} we get that there exists a constant $C$ such that $\Prob_0$ almost surely:
\[
\frac{1}{n}\sum\limits_{i=0}^{\tau_n}1_{Y_i\not\in\tilde{\mathcal{T}}}\rightarrow C.
\]
So we only have to look at the time spent in the traps. By lemma \ref{lem:20} we get that for any $\varepsilon>0$, for $n$ large enough:
\[
\Prob_0\left(\frac{1}{n\log(n)}\sum\limits_{i=1}^{\tau_{n+1}-1}1_{Y_i\in\tilde{\mathcal{T}}}1_{Y_i\not\in\tilde{\mathcal{T}}_J}\geq \varepsilon\right)\leq \varepsilon.
\]
Therefore we only have to look at the time spent in traps in a direction $j\in J$. For any trap $\{x,y\}$ let $\tilde{N}_x$ be the number of times the walks exits the trap $\{x,y\}$, we have $\tilde{N}_w=\tilde{N}_y$ . Let $\varepsilon>0$ be a positive constant. By lemma \ref{lem:16} there exists a $m_{\varepsilon}$ such that:
\[
\Prob_0\left(\frac{1}{n\log(n)}\sum\limits_{i=1}^{\tau_{n+1}-1}1_{Y_i\in\tilde{\mathcal{T}}_J}1_{\tilde{N}_{Y_i}\geq m_{\varepsilon}}\geq \varepsilon\right)\leq \varepsilon.
\] 
And by lemma \ref{lem:18} we get that there is a constant $C_{m_{\varepsilon}}$ such that:
\[
\frac{1}{n\log(n)}\sum\limits_{i=1}^{\tau_{n+1}-1}1_{Y_i\in\tilde{\mathcal{T}}_J}1_{\tilde{N}_{Y_i}\leq m_{\varepsilon}}\rightarrow C_{m_{\varepsilon}}\text{ in probability.}
\] 
So for $n$ large enough:
\[
\Prob_0\left(\frac{1}{n\log(n)}\sum\limits_{i=1}^{\tau_{n+1}-1}1_{Y_i\in\tilde{\mathcal{T}}}\in [C_{m_{\varepsilon}}-2\varepsilon,C_{m_{\varepsilon}}+2\varepsilon]\right)\geq 1-2\varepsilon.
\]
This means that there exists a constant $C_{\infty}$ such that:
\[
\frac{1}{n\log(n)}\sum\limits_{i=1}^{\tau_{n+1}-1}1_{Y_i\in\tilde{\mathcal{T}}} \rightarrow C_{\infty}\text{ in probability.}
\]
And therefore:
\[
\frac{1}{n\log(n)} \tau_{n+1} \rightarrow C_{\infty}\text{ in probability.}
\]
So we have proved the first part of the theorem.\\
Now, by lemma \ref{lem:10} we have for some $C>0$, $\Prob_0$ almost surely:
\[
\frac{Y_{\tau_n}.e_1}{n}\rightarrow C.
\]
So for any $\varepsilon>0$, by writing $L(n):=\min \{i,Y_i.e_1\geq n\}$ and $C^+ = \frac{1}{C(1-\varepsilon)}$:
\[
\begin{aligned}
&\Prob_0[L(n)\geq C_{\infty}C^+(1+\varepsilon) n \log(n)]\\
\leq & \Prob_0[L(n) \geq C_{\infty}C^+(1+\varepsilon)n\log(n) \text{ and } \tau_{C^+ n} \leq C_{\infty}C^+(1+\varepsilon)n\log(n)]\\
&\ \ \ +\Prob_0[\tau_{C^+ n}>C_{\infty}C^+(1+\varepsilon)n\log(n)]\\
= & \Prob_0[L(n) \geq C_{\infty}C^+(1+\varepsilon)n\log(n) \text{ and } \tau_{C^+ n} \leq C_{\infty}C^+ (1+\varepsilon)n\log(n)]+o(1)\\
\leq & \Prob_0[L(n) \geq \tau_{C^+ n}]+o(1)\\
=& \Prob_0[Y_{\tau_{C^+ n}}.e_1\leq n] +o(1)\\
=& \Prob_0\left[\frac{Y_{\tau_{C^+ n}}.e_1}{C^+ n} \leq C(1-\varepsilon)\right] +o(1)\\
=&o(1).
\end{aligned}
\]
The same way we get, by taking $C^- =\frac{1}{C(1+\varepsilon)}$:
\[
\begin{aligned}
&\Prob_0(L(n)\leq C_{\infty}C^-(1-\varepsilon) n \log(n))\\
\leq & \Prob_0(L(n) \leq C_{\infty}C^-(1-\varepsilon)n\log(n) \text{ and } \tau_{C^- n} \geq C_{\infty}C^-(1-\varepsilon)n\log(n))\\
&\ \ \ +\Prob_0(\tau_{C^- n}<C_{\infty}C^-(1-\varepsilon)n\log(n))\\
= & \Prob_0(L(n) \leq C_{\infty}C^-(1-\varepsilon)n\log(n) \text{ and } \tau_{C^- n} \geq C_{\infty}C^-(1-\varepsilon)n\log(n))+o(1)\\
\leq & \Prob_0(L(n) \leq \tau_{C^- n})+o(1)\\
=& \Prob_0(Y_{\tau_{C^- n}}.e_1\geq n) +o(1)\\
=& \Prob_0\left(\frac{Y_{\tau_{C^- n}}.e_1}{C^- n} \geq C(1+\varepsilon)\right) +o(1)\\
=&o(1).
\end{aligned}
\]
So we get the second result.
Now for the last result, we define $L^{\tau}(n)=\min\{i,\tau_i\geq n\}$ so $\tau_{L^{\tau}(n)-1}<n\leq \tau_{L^{\tau}(n)}$. We get, for $n$ big enough:
\[
\Prob_0\left(L^{\tau}(n)\geq C_{\infty}^{-1}(1+2\varepsilon)\frac{n}{\log(n)}\right)\\
\leq \Prob_0\left(\tau_{C_{\infty}^{-1}(1+\varepsilon)\frac{n}{\log(n)}}\leq n\right).
\]
And we have:
\[
C_{\infty}^{-1}(1+\varepsilon)\frac{n}{\log(n)}\log\left(C_{\infty}^{-1}(1+\varepsilon)\frac{n}{\log(n)}\right)
=C_{\infty}^{-1}(1+\varepsilon)n(1+o(1)).
\]
And therefore, using the result of part one:
\[
\frac{\tau_{C_{\infty}^{-1}(1+\varepsilon)\frac{n}{\log(n)}}}{n}\rightarrow C_{\infty}C_{\infty}^{-1}(1+\varepsilon)=(1+\varepsilon).
\]
So we get that:
\[
\Prob_0\left(\tau_{C_{\infty}^{-1}(1+\varepsilon)\frac{n}{\log(n)}}\leq n\right) \rightarrow 0.
\]
And therefore:
\[
\Prob_0\left(L^{\tau}(n)\geq C_{\infty}^{-1}(1+2\varepsilon)\frac{n}{\log(n)}\right)\rightarrow 0.
\]
The proof of the lower bound is exactly the same:
\[
\Prob_0\left(L^{\tau}(n)\leq C_{\infty}^{-1}(1-\varepsilon)\frac{n}{\log(n)}\right)\\
\leq \Prob_0\left(\tau_{C_{\infty}^{-1}(1-\varepsilon)\frac{n}{\log(n)}}\geq n\right).
\]
But we have:
\[
n^{-1}\tau_{C_{\infty}^{-1}(1-\varepsilon)\frac{n}{\log(n)}}\rightarrow (1-\varepsilon).
\]
So 
\[
\Prob_0\left(L^{\tau}(n)\leq C_{\infty}^{-1}(1-\varepsilon)\frac{n}{\log(n)}\right)\rightarrow 0.
\]
And therefore:
\[
\frac{\log(n)}{n}L^{\tau}(n)\rightarrow C_{\infty}^{-1}.
\]
Now, by lemma \ref{lem:10} $\frac{Y_i}{L^{\tau}(i)}\rightarrow D$, $\Prob_0$ almost surely so we get:
\[
\frac{\log(n)}{n}Y_n\rightarrow C_{\infty}^{-1}D.
\]
\end{proof}

\section{Annex}

\begin{lem}\label{lem:22}
Let $X$ be a non-negative random variable such that $\E(X)<\infty$. There exists an increasing, positive, concave function $\phi$ such that $\phi(t)$ goes to infinity when $t$ goes to infinity and:
\[
\E(\Phi(X))<\infty,
\]
where $\Phi(t)=\int\limits_{x=0}^t \phi(x) \dd x$.
\end{lem}
\begin{proof}
First we show that there exists a non-decreasing, positive function $f:\R^+\rightarrow \R^+$ such that $f(t)$ goes to infinity when $t$ goes to infinity and:
\[
\E(Xf(X))<\infty.
\]
To do that we first define the sequence $(t_i)$ by:
\[
\begin{aligned}
t_0=&0\\
t_{i+1}=& 1+ \inf \left\{ x\geq t_i , \E(X 1_{X> x} ) \leq 2^{-(i+1)}\E(X) \right\}.
\end{aligned}
\] 
Now we define $f$ by:
\[
f(x) = 1+\sum\limits_{i\geq 0} 1_{x \geq t_i}.
\]
We clearly have that $f$ is non-decreasing, positive ($f(t)\geq 2$) and that $f(t)$ goes to infinity when $t$ goes to infinity. As for the expectation we have:
\[
\begin{aligned}
\E(Xf(X))
= & \E\left(\sum\limits_{i\geq 0} X1_{X \geq t_i}\right) +\E(X) \\
= & \sum\limits_{i\geq 0}\E\left( X1_{X \geq t_i}\right) +\E(X)\\
\leq & \sum\limits_{i\geq 0} 2^{-i}\E(X) +\E(X)\\
\leq & 3 \E(X) < \infty.
\end{aligned}
\]
Now we want to find an increasing concave function $\phi$ lower than $f$ such that $\phi(t)$ goes to infinity when $t$ goes to infinity. To that effect we will define the sequences $(a_i)$ and $(b_i)$ by:
\[
\begin{aligned}
a_0 = & 1,\\
b_0 = & \frac{1}{t_1},\\
\forall i\in\N,\   a_{i+1}=& a_i + b_i(t_{i+1}-t_{i}),\\
\forall i\in\N,\   \min (b_{i+1}=& b_i, \frac{(i+2)-a_i}{t_{i+1}-t_{i}})
\end{aligned}
\]
and we define $\phi$ by:
\[
\forall i\in\N, \forall x\in [t_i,t_{i+1}), \phi(x)= a_i + b_i (x-t_i).
\]
The function $\phi$ is continuous and its slope is decreasing so it is clearly concave.\\
We now have to prove that $\lim\limits_{t\rightarrow \infty}\phi(t)=\infty$ . First we want to show that for every $i\in\N,\  a_i\leq i+1$. It is obvious for $i\in\{0,1\}$ and for $i>0$ we have:
\[
a_i
\leq a_{i-1} + \frac{(i+1)-a_{i-1}}{t_{i}-t_{i-1}}(t_{i}-t_{i-1})
= i+1.
\]
Now we want to show that there can be no $i$ such that $b_i\leq 0$. If there was, we could define $j$ by $j=\min \{i,b_i\leq 0\}$, we would have $j \geq 1$ and:
\[
\frac{(j+1)-a_{j-1}}{t_{j}-t_{j-1}}\leq 0.
\]
But since $a_{j-1}\leq j$ it is impossible so all the $b_i$ are positive and therefore $\phi$ is increasing. Now we will prove that $\lim\limits_{i\rightarrow\infty}a_i=\infty$. First we notice that if $b_{i+1}< b_i$ then $b_{i+1}=\frac{(i+2)-a_i}{t_{i+1}-t_{i}}$ so $a_{i+1}=i+2$. Therefore, either the $b_i$ are stationary and $\phi$ is larger than some affine function with positive slope which implies the result we want or the sequence $b_i$ is not stationary and there are infinitely many $i$ such that $a_{i+1}=i+2$ and therefore we have the result we want.  \\
We still have to show that $\phi\leq f$. We know that $\phi$ is increasing and we have:
\[
\forall i\in\N, \forall x\in [t_i,t_{i+1}), f(x)-\phi(x)= i+2 - \phi(x) \geq i+2 - \phi(t_{i+1}) = i+2 - a_{i+1} \geq 0.
\]
So we have the desired result.
\end{proof} 
\begin{lem}\label{lem:24}
Let $\phi$ be a non-decreasing, positive concave function and $\Phi(x):=\int\limits_{t=0}^x\phi(t)\dd t$. There exists a constant $C_{\phi}$ such that if $X$ is a geometric random variable with success probability $p$:
\[
\frac{1}{2}\Phi\left(\frac{1}{p}\right)
\leq \frac{1}{2} \frac{1}{p}\phi\left(\frac{1}{p}\right) 
\leq \E(\Phi(1+X)) 
\leq C_{\phi} \frac{1}{p}\phi\left(\frac{1}{p}\right)
\leq 2 C_{\phi} \Phi\left(\frac{1}{p}\right).
\]
\end{lem}
\begin{proof}
$\Phi$ is convex so if $X$ is a geometric random variable with success probability $p$:
\[
\begin{aligned}
\E(\Phi(1+X))
\geq & \Phi(\E(1+X))\\
=& \Phi\left(\frac{1}{p}\right)\\
=& \int\limits_{t=0}^{\frac{1}{p}}\phi(t)\dd t \\
=& \frac{1}{p} \int\limits_{t=0}^{1}\phi\left(t\frac{1}{p}\right)\dd t \\
\geq & \frac{1}{p} \int\limits_{t=0}^{1} t \phi\left(\frac{1}{p}\right)+ (1-t)\phi\left(0\right) \dd t \\
\geq & \frac{1}{p} \int\limits_{t=0}^{1} t \phi\left(\frac{1}{p}\right) \dd t \\
= & \frac{1}{2} \frac{1}{p}\phi\left(\frac{1}{p}\right).
\end{aligned}
\] 
Now for the upper bound, we will first look at the case where $p\leq \frac{1}{2}$:
\[
\begin{aligned}
\E(\Phi(1+X))
=&\E\left(\int\limits_{t=0}^{\infty}\phi(t)1_{1+X\geq t}\dd t\right)\\
=&\int\limits_{t=0}^{\infty}\phi(t)\Prob(X\geq t-1)\dd t\\
\leq & \int\limits_{t=0}^{\infty}\phi(t) (1-p)^{t-1} \dd t \\
\leq & 2 \int\limits_{t=0}^{\infty}\phi(t) \exp(t\log(1-p)) \dd t \\
= & 2 \frac{-1}{\log(1-p)} \int\limits_{t=0}^{\infty}\phi\left(-\frac{t}{\log(1-p)}\right) \exp(-t) \dd t \\
\leq & 2 \frac{-1}{\log(1-p)} \left(\phi\left(-\frac{1}{\log(1-p)}\right) +\int\limits_{t=1}^{\infty}\phi\left(-\frac{t}{\log(1-p)}\right) \exp(-t) \dd t\right). \\
\end{aligned}
\]
Now we use the fact that $\phi$ is concave, this gives us, for $t\geq 1$:
\[
\frac{1}{t}\phi\left(-\frac{t}{\log(1-p)}\right)+\left(1-\frac{1}{t}\right)\phi(0)
\leq\phi\left(-\frac{1}{\log(1-p)}\right). 
\]
Since $\phi$ is positive, we get:
\[
\phi\left(-\frac{t}{\log(1-p)}\right)
\leq t\phi\left(-\frac{1}{\log(1-p)}\right).  
\]
So we get:
\[
\begin{aligned}
\E(\Phi(1+X)) 
\leq & 2 \frac{-1}{\log(1-p)}\left(\phi\left(-\frac{1}{\log(1-p)}\right) +\int\limits_{t=1}^{\infty}t\phi\left(-\frac{1}{\log(1-p)}\right) \exp(-t) \dd t\right) \\
\leq & 2 \frac{-1}{\log(1-p)} \left(\phi\left(-\frac{1}{\log(1-p)}\right) +\phi\left(-\frac{1}{\log(1-p)}\right)\right)\\
= & 4 \frac{-1}{\log(1-p)}\phi\left(-\frac{1}{\log(1-p)}\right).
\end{aligned}
\]
Since $-\frac{p}{\log(1-p)}\leq 1$ and $\phi$ is increasing, we get:
\[
-\frac{1}{\log(1-p)}\phi\left(-\frac{1}{\log(1-p)}\right)\leq \frac{1}{p}\phi\left(\frac{1}{p}\right).
\]
And therefore, if $p\leq \frac{1}{2}$:
\[
\E(\Phi(1+X))\leq 4\frac{1}{p}\phi\left(\frac{1}{p}\right).
\]
If $p\geq \frac{1}{2}$ we can couple $X$ with a geometric random variable $Y$ of parameter $\frac{1}{2}$ such that almost surely $Y\geq X$ and since $\Phi$ is increasing:
\[
\E(\Phi(1+X))
\leq \E(\Phi(1+Y))
\leq 8\phi(2)
\leq 8\phi(2) \frac{1}{p}\frac{\phi\left(\frac{1}{p}\right)}{\phi(1)}
= 8 \frac{\phi(2)}{\phi(1)}\frac{1}{p}\phi\left(\frac{1}{p}\right)
 \leq 16 \frac{1}{p}\phi\left(\frac{1}{p}\right).
\]
We get the upper bound we wanted. \\
Now we just have to prove that for any $x\geq 0$, $\frac{1}{2}x\phi(x) \leq \Phi(x)\leq x\phi(x)$. For the upper bound we have:
\[
\begin{aligned}
\Phi(x) = \int\limits_{0}^x \phi(t) \dd t \leq \int\limits_{0}^x \phi(x) \dd t =x\phi(x).
\end{aligned}
\]
And for the lower bound we have:
\[
\Phi(x) = \int\limits_{0}^x \phi(t) \dd t 
= \int\limits_{0}^x \phi\left(\frac{t}{x}x\right) \dd t
\geq  \int\limits_{0}^x \frac{t}{x} \phi\left(x\right) \dd t
=  \frac{1}{2}x\phi(x).
\]
\end{proof}
\begin{lem}\label{lem:25}
Let $X$ be a positive random variable, and let $a=\E(X)$ and $\tilde{X}=X-a$. If $\text{Var}(X)\leq a^2$ then:
\[
\forall \gamma \in [0,1],\ \text{Var}\left(X^{\gamma}\right)\leq 2a^{2\gamma}\left(\frac{\text{Var}(X)}{a^2}\right).
\]
\end{lem}
\begin{proof}
For any $x\in [-1,\infty)$, let $f_x: [0,1]\mapsto \R$ the function defined by
\[
f_x(\gamma) :=\gamma\rightarrow (1+x)^{\gamma}.
\]
This function is convex and $f_x(1)=1+x$ and $f^{\prime}_x(1)=(1+x)\log(1+x)$ so:
\[
\forall \gamma \in[0,1], \ f_x(\gamma)\geq 1+x + (\gamma-1)(1+x)\log(1+x) \geq 1+x -(1-\gamma)(1+x)x \geq 1+\gamma x -(1-\gamma)x^2.
\] 
By Jensen inequality, we have:
\[
\E(X^{\gamma})\leq a^{\gamma}.
\]
Since $\E(X^{\gamma})=a^{\gamma}\E\left(\left(1+\frac{\tilde{X}}{a}\right)^{\gamma}\right)$, we also get:
\[
\E(X^{\gamma})\geq a^{\gamma}\left(1-(1-\gamma)\frac{\text{Var}(X)}{a^2}\right).
\]
So if $\text{Var}(X)\leq a^2$, then 
\[
-\E(X^{\gamma})^2\leq -a^{2\gamma}\left(1-(1-\gamma)\frac{\text{Var}(X)}{a^2}\right)^2
\leq -a^{2\gamma} \left(1-2(1-\gamma)\frac{\text{Var}(X)}{a^2}\right).
\]
We also have:
\[
\E(X^{2\gamma})\leq \E(X^2)^{\gamma}
=\left(a^2+\text{Var}(X)\right)^{\gamma}
\leq a^{2\gamma}\left(1+\gamma\frac{\text{Var}(X)}{a^2}\right).
\]
Finally we get:
\[
\text{Var}\left(X^{\gamma}\right)\leq a^{2\gamma}\left(1+\gamma\frac{\text{Var}(X)}{a^2}-1+2(1-\gamma)\frac{\text{Var}(X)}{a^2}\right)=a^{2\gamma}(2-\gamma)\frac{\text{Var}(X)}{a^2}.
\]
\end{proof}
\begin{lem}\label{lem:23}
Let $p\in (0,\infty)$ be a positive real, $N\geq 1$ an integer, $h\in (\frac{1}{4},1)$ and $q\in (0,\infty)$ with $1\geq q(1-h)\geq \frac{1}{2}$. Let $\left(\varepsilon_i\right)$ be a sequence of integer in $\{0,1\}$. Let $\left(H_i\right)_{i\in\N}$ be a sequence of iid random variables following a geometric law of parameter $h$ (here $h$ is the probability of success). Let $\left(\mathcal{E}_{i,j}\right)_{i,j\in\N}$ be a sequence of iid random variables , independent of $(H_i)$ and following an exponential law of parameter $p$. Now let $Z$ be defined by:
\[
Z=\sum\limits_{i=1}^{N}\sum\limits_{j=1}^{\varepsilon_i+H_i}\mathcal{E}_{i,j}\frac{p}{q}.
\]
There exists a constant C such that if $N\geq 1$:
\[
\forall \gamma \in [0,1],\ \text{Var}\left(Z^{\gamma}\right)\leq C N^{2\gamma-1}\leq C N^{\gamma}.
\]
We also have that there are two constant $c_1,c_2>0$ that do not depend on $\gamma$ such that:
\[
c_1 N^{\gamma}\leq \E(Z^{\gamma}) \leq c_2 N^{\gamma}.
\]
\end{lem}
\begin{proof}
First we look at the expectation of $Z$, we get:
\[
\begin{aligned}
\E(Z)
&=\sum\limits_{i=1}^{N}\E\left(\sum\limits_{j=1}^{\varepsilon_i+H_i}\frac{1}{p}\frac{p}{q}\right)\\
&=\sum\limits_{i=1}^{N}\frac{1}{q}\left(\varepsilon_{i}+\frac{h}{1-h}\right)\\
&=\frac{1}{q(1-h)}\sum\limits_{i=1}^{N}\varepsilon_{i}(1-h)+h.
\end{aligned}
\]
Now we will look at the variance but first we need a small result to simplify the notations, for this result, $M$ will be a non negative random variable and $(X_i)_{i\in\N}$ a sequence of iid real random variables, independent of $M$. We get: 
\[
\begin{aligned}
\text{Var}\left(\sum_{i=1}^{M}X_i\right)
&=\E\left(\left(\sum_{i=1}^{M}X_i\right)^2\right)-\left(\E\left(\sum_{i=1}^{M}X_i\right)\right)^2\\
&=\E\left(M\E(X_1^2)+M(M-1)\E(X_1)^2\right)-\E(M)^2\E(X_1)^2\\
&=\E(M)\text{Var}(X_1)+\text{Var}(M)\E(X_1)^2.
\end{aligned}
\]
Now we can compute the variance of $Z$. First we have:
\[
\text{Var}\left(Z\right)
=\sum\limits_{i=1}^{N}\text{Var}\left(\sum\limits_{j=1}^{\varepsilon_i+H_i}\mathcal{E}_{i,j}\frac{p}{q}\right)
=\frac{p^2}{q^2}\sum\limits_{i=1}^{N}\text{Var}\left(\sum\limits_{j=1}^{\varepsilon_i+H_i}\mathcal{E}_{i,j}\right).
\]
Then we have:
\[
\begin{aligned}
\frac{p^2}{q^2}\sum\limits_{i=1}^{N}\E(\varepsilon_i+H_i)\text{Var}(\mathcal{E}_{i,1})
&=\frac{p^2}{q^2}\sum\limits_{i=1}^{N}\left(\varepsilon_i+\frac{h}{1-h}\right)\frac{1}{p^2}\\
&=\frac{1}{q^2(1-h)^2}\sum\limits_{i=1}^{N}\varepsilon_i(1-h)^2+h(1-h),
\end{aligned}
\]
\[
\begin{aligned}
\frac{p^2}{q^2}\sum\limits_{i=1}^{N}\text{Var}((\varepsilon_i+H_i)^2)\E(\mathcal{E}_{i,1})^2
&=\frac{p^2}{q^2}\sum\limits_{i=1}^{N}\frac{h}{1-h}\frac{1}{p^2}\\
&=\frac{1}{q^2(1-h)^2}\sum\limits_{i=1}^{N}h(1-h),
\end{aligned}
\]
So we get, by summing these two equalities:
\[
\text{Var}\left(Z\right)
=\frac{1}{q^2(1-h)^2}\sum\limits_{i=1}^{N}\varepsilon_i(1-h)^2+2h(1-h).
\]
We have assumed that $h\geq\frac{1}{4}$ and $\frac{1}{2}\leq q(1-h)\leq 1$ therefore we have:
\[
\frac{1}{4}N\leq \E(Z) \leq 4N,
\]
\[
\text{Var}(Z)\leq 20 N.
\]
Therefore we have:
\[
\frac{\text{Var}(Z)}{(\E(Z))^2}\leq 320 \frac{1}{N}.
\]
So by lemma \ref{lem:25}, for $N\geq 320$ we have:
\[
\forall \gamma \in [0,1],\ \text{Var}\left(Z^{\gamma}\right)\leq 2\E(Z)^{2\gamma}\left(\frac{\text{Var}(Z)}{\E(Z)^2}\right)\leq 4^{2\gamma} N^{2\gamma}\frac{640}{N}.
\]
And if $N\leq 320$ we have: 
\[
\forall \gamma \in [0,1],\ \text{Var}\left(Z^{\gamma}\right)\leq \E(Z^{2\gamma})\leq\E(Z^2)\leq (20N+16N^2).
\]
So there exist a constant $C$ such that if $1\leq N\leq 320$:
\[
\text{Var}\left(Z^{\gamma}\right)\leq C\frac{1}{N}.
\]
So
\[
\forall \gamma \in [0,1],\ \text{Var}\left(Z^{\gamma}\right)\leq C N^{2\gamma-1}.
\]
So we have that there exists a constant C such that if $N\geq 1$:
\[
\forall \gamma \in [0,1],\ \text{Var}\left(Z^{\gamma}\right)\leq C N^{2\gamma-1}\leq C N^{\gamma}.
\]
Now for the expectation, we first have the upper bound:
\[
\E(Z^{\gamma})\leq \E(Z)^{\gamma}\leq (4N)^{\gamma}.
\]
For the lower bound, we will use Holder inequality:
\[
\E(Z)=\E\left(Z^{\frac{\gamma}{2-\gamma}}Z^{2\frac{1-\gamma}{2-\gamma}}\right)\leq\E\left(Z^{\frac{\gamma}{2-\gamma}(2-\gamma)}\right)^{\frac{1}{2-\gamma}}\E\left(Z^{2\frac{1-\gamma}{2-\gamma}\frac{2-\gamma}{1-\gamma}}\right)^{\frac{1-\gamma}{2-\gamma}}.
\]
This yields:
\[
\E(Z)^{2-\gamma}\leq\E(Z^{\gamma})\E(Z^2)^{1-\gamma}
\]
ie:
\[
E(Z^{\gamma})\geq\frac{\E(Z)^{2-\gamma}}{\E(Z^2)^{1-\gamma}}.
\]
Now we have $E(Z^2)=\text{Var}(Z)+\E(Z)^2 $ since $\text{Var}(Z)\leq 80\E(Z)$ and $\E(Z)\geq\frac{1}{4}$ we have $\text{Var}(Z)\leq 320 \E(Z)^2$ and therefore: $E(Z^2)\leq 321 E(Z)^2$ which yields:
\[
E(Z^{\gamma})\geq\frac{\E(Z)^{2-\gamma}}{(321\E(Z)^2)^{1-\gamma}}\geq\frac{\E(Z)^{\gamma}}{321^{1-\gamma}}\geq\frac{\E(Z)^{\gamma}}{321}.
\]
\end{proof}

\begin{lem}\label{lem:30}
Let $\beta\in [0,1]$. Let $(N_i)_{i\in\N^*}$ be a sequence of random positive integers and $(A_i)_{i\in\N}$ be a sequence of random finite subsets of $\N$ with the following two properties:   
\[
\forall i\geq 0, A_i\subset A_{i+1},
\]
\[
\# A_i \rightarrow \infty.
\]
Let $(Z_i)_{i\in \N}$ be independent exponential random variables of parameter $1$  independent of $(A_i),(N_i)$. \\
Then there exists a constant $C>0$ such that almost surely:
\[
\exists m\in\N, \forall n\geq m, \sum\limits_{i\in A_n}\left(\sum\limits_{j=1}^{N_i}Z_i\right)^{\beta} \geq C \sum\limits_{i\in A_n}\left(N_i\right)^{\beta}.
\]
\end{lem}
\begin{proof}
Let $C$ be such that $2C-2^{1-\beta}>0$
Let $(n_i)_{i\in\N}$ be the sequence defined by:
\[
n_i=\min\left\{i,\#\sum\limits_{i\in A_n}\left(N_i\right)^{\beta}\geq 2^{i}\right\}.
\]
We have that if
\[
\exists m\in\N, \forall j\geq m, \sum\limits_{i\in A_{n_j}}\left(\sum\limits_{k=1}^{N_i}Z_i\right)^{\beta} \geq 2C \sum\limits_{i\in A_{n_j}}\left(N_i\right)^{\beta}
\]
and $M$ is such an $m$ then for every $n\geq n_M$, if $j$  is the integer that satisfies $n_j\leq n < n_{j+1}$, we have:
\[
\begin{aligned}
\sum\limits_{i\in A_n}\left(N_i\right)^{\beta} 
\leq& 2^{j+1} \\
\leq& 2 \sum\limits_{i\in A_{n_j}}\left(N_i\right)^{\beta} \\
\leq& 2C \sum\limits_{i\in A_{n_j}}\left(\sum\limits_{k=1}^{N_i}Z_i\right)^{\beta} \\
\leq & 2C \sum\limits_{i\in A_n}\left(\sum\limits_{k=1}^{N_i}Z_i\right)^{\beta}.
\end{aligned}
\]
By lemma \ref{lem:25}, for any $i\in\N^{*}$: 
\[
\text{Var}\left(\left(\sum\limits_{j=1}^{N_i}\mathcal{E}_{i,j}\right)^{\beta}|(A_k),(N_k)\right)\leq 2 (N_i)^{2\beta-1}\leq 2 (N_i)^{\beta}.
\]
And by H\"older:
\[
\begin{aligned}
\E\left(\left(\sum\limits_{j=1}^{N_i}\mathcal{E}_{i,j}\right)^{\beta}|(A_k),(N_k)\right)
\geq & \E\left(\sum\limits_{j=1}^{N_i}\mathcal{E}_{i,j}|(A_k),(N_k)\right)^{2-\beta}\E\left(\left(\sum\limits_{j=1}^{N_i}\mathcal{E}_{i,j}\right)^2|(A_k),(N_k)\right)^{-(1-\beta)}\\
= & (N_i)^{2-\beta} (N_i^2+N_i)^{-(1-\beta)}\\
\geq & (N_i)^{2-\beta} (2N_i^2)^{-(1-\beta)} \\
= & 2^{\beta-1} (N_i)^{\beta}.
\end{aligned}
\]
Now we get:
\[
\begin{aligned}
&\sum\limits_{j\geq 0}\Prob\left(\sum\limits_{i\in A_{n_j}}\left(\sum\limits_{k=1}^{N_i}Z_i\right)^{\beta} \leq 2C \sum\limits_{i\in A_{n_j}}\left(N_i\right)^{\beta}\right) \\
\leq &\sum\limits_{j\geq 0} \E\left(\frac{\text{Var}\left(\sum\limits_{i\in A_{n_j}}\left(\sum\limits_{k=1}^{N_i}Z_i\right)^{\beta}|(A_k),(N_k)\right)}{\left((2C-2^{1-\beta})\sum\limits_{i\in A_{n_j}}(N_i)^{\beta}\right)^2}\right)  \\
\leq & \sum\limits_{j\geq 0} \E\left(\frac{2}{(2C-2^{1-\beta})^2\sum\limits_{i\in A_{n_j}}(N_i)^{\beta}}\right) \\
\leq & \frac{2}{(2C-2^{1-\beta})^2} \sum\limits_{j\geq 0} 2^{-j} < \infty.
\end{aligned}
\]
So by Borell-Cantelli we get the result we want
\end{proof}
\begin{lem}\label{lem:28}
Let $f,g$ be two non-decreasing positive c\`adl\`ag functions with $f(0)=g(0)=0$. Let $A,B>0$ be constants such that $f(A)\geq B$ and $g(A) \geq B$. Let $\varepsilon, \delta>0$ be such that:
\[
\forall t\in [0,A+\varepsilon], g(t+\varepsilon)\geq g(t) +\delta
\]
and 
\[
\sup \{|f(t)-g(t)|,t\in [0,A+2\varepsilon]\} \leq \frac{\delta}{2}.
\]
Then:
\[
\sup \{|f^{-1} (x)- g^{-1}(t)|,t\in [0,B]\} \leq 2\varepsilon.
\]
\end{lem}
\begin{proof}
Let $t$ be in $[0,B]$. First we have:
\[
f\left(g^{-1}(t) +2\varepsilon \right)
\geq g\left(g^{-1}(t) +2\varepsilon \right) - \frac{\delta}{2}
\geq g(g^{-1}(t)+\varepsilon) +\delta - \frac{\delta}{2}
\geq t.
\]
Therefore $f^{-1}(t) \leq g^{-1}(t) +\varepsilon$. Similarly we have:
\[
f\left(g^{-1}(t) -\varepsilon \right)
\leq g\left(g^{-1}(t) -\varepsilon \right) + \frac{\delta}{2}
\leq g(g^{-1}(t)) -\delta + \frac{\delta}{2}
< t.
\]
Therefore $f^{-1}(t) \geq g^{-1}(t) -\varepsilon$. So we have the result we want.
\end{proof}
\begin{lem}\label{lem:29}
Let $t\rightarrow\mathcal{S}^{\kappa}(t)$ be the jump process where $\mathcal{S}^{\kappa}(1)$ is a completely asymmetric, positive stable law of parameter $\kappa$. For any $\varepsilon>0$ and any $B>0$ there exists $A>0$ and $\delta>0$ such that:
\[
\Prob(\mathcal{S}^{\kappa}(A)\geq B) \geq 1-\varepsilon,
\]
\[
\Prob(\exists t\leq A-\varepsilon,\mathcal{S}^{\kappa}(t+\varepsilon)-\mathcal{S}^{\kappa}(t) < \delta)\leq \varepsilon.
\] 
\end{lem}
\begin{proof}
There clearly exists an $A$ that satisfies the first property. Now we need to find a $\delta$ that satisfies the second inequality for this $A$. We will look at a slightly different property: 
\[
\exists i \leq \frac{2A}{\varepsilon},\mathcal{S}^{\kappa}\left(i\frac{\varepsilon}{2}\right)-\mathcal{S}^{\kappa}\left((i+1)\frac{\varepsilon}{2}\right)\leq \delta.
\]
Since for every $t\leq A-\varepsilon$ there exists $i \leq \frac{2A}{\varepsilon}$ such that: $[i\frac{\varepsilon}{2},(i+1)\frac{\varepsilon}{2}]\subset [t,t+\varepsilon]$, we have that for any $\delta>0$:
\[
\Prob(\exists t\leq A-\varepsilon,\mathcal{S}^{\kappa}(t+\varepsilon)-\mathcal{S}^{\kappa}(t+\varepsilon)\leq \delta)
\leq \Prob\left(\exists i \leq \frac{2A}{\varepsilon},\mathcal{S}^{\kappa}\left(i\frac{\varepsilon}{2}\right)-\mathcal{S}^{\kappa}\left((i+1)\frac{\varepsilon}{2}\right)\leq \delta \right).
\]
And there clearly exists $\delta$ such that 
\[
\Prob\left(\exists i \leq \frac{2A}{\varepsilon},\mathcal{S}^{\kappa}\left(i\frac{\varepsilon}{2}\right)-\mathcal{S}^{\kappa}\left((i+1)\frac{\varepsilon}{2}\right)\leq \delta \right) \leq \varepsilon.
\]
So we get the result we want.
\end{proof}

\section{Acknowledgement}
I would like to thank my Ph.D advisor Christophe Sabot for suggesting me this problem and Alexander Fribergh for helpful discussions on the subject.

\bibliographystyle{abbrv}
\bibliography{biblio}

\begin{thebibliography}{10}

\bibitem{BouchaudBenArous}
G.~Ben-Arous.
\newblock Aging and spin-glass dynamics.
\newblock {\em Proceedings of the ICM}, 3:3--14, 2002.

\bibitem{FriberghSubbalTree}
G.~Ben~Arous, A.~Fribergh, and N.~Gantert.
\newblock Biased random walks on galton watson trees with leaves.
\newblock {\em Ann.Probab}, 40(1):280--338, 2012.

\bibitem{BergEffectPol}
N.~Berger, A.~Drewitz, and A.~Ram\'irez.
\newblock Effective polynomial ballisticity conditions for random walk in
  random environment.
\newblock {\em Comm.Pure Appl. Math.}, 67:1947--1973, 2014.

\bibitem{BergerZeitouniTCL}
N.~Berger and O.~Zeitouni.
\newblock A quenched invariance principle for certain ballistic random walks in
  i.i.d environments.
\newblock {\em Progress in Probability}, 60:137--160, 2009.

\bibitem{Billingsley}
P.~Billingsley.
\newblock Convergence of probability measures.
\newblock {\em John Wiley \& Sons Inc., New-York}, 1968.

\bibitem{ExitBall0}
E.~Bolthausen and O.~Zeitouni.
\newblock Multiscale analysis of exit distributions for random walks in random
  environment.
\newblock {\em Probab.Theory Related Fields}, 138:581--645, 2007.

\bibitem{BouchetSubbal}
E.~Bouchet.
\newblock Sub-ballistic random walk in dirichlet environment.
\newblock {\em Electon.J.Probab}, 18(58), 2013.

\bibitem{ballistic1}
E.~Bouchet, A.~Ram\'irez, and C.~Sabot.
\newblock Sharp ellipticity conditions for ballistic behaviour of random walks
  in random environment.
\newblock {\em Bernoulli}, 22(2):969--994, 2016.

\bibitem{BricKupia}
J.~Bricmont and A.~Kupiainen.
\newblock Random walks in asymmetric random environments.
\newblock {\em Comm. Math. Phys}, 142:345--420, 1991.

\bibitem{ElliptCriteria}
D.~Campos and A.~Ram\'irez.
\newblock Ellipticity criteria for ballistic behaviour of random walks in
  random environment.
\newblock {\em Probab.Theory Related Fields}, 160:189--251, 2014.

\bibitem{Durrett}
R.~Durrett.
\newblock Probability: Theory and examples. 2nd ed.
\newblock {\em Duxbury Press, Belmont, CA}, 1996.

\bibitem{RWDE}
N.~Enriquez and C.~Sabot.
\newblock Random walks in a dirichlet environment.
\newblock {\em Electron.J.Probab}, 11(31):802--817, 2006.

\bibitem{LimLawsZ}
N.~Enriquez, C.~Sabot, and O.~Zindy.
\newblock Limit laws for transient random walks in random environment on
  $\mathbb{Z}$.
\newblock {\em Ann.Inst.Fourier (Grenoble)}, 59:2469--2508, 2007.

\bibitem{FriberghSubbalPerc}
A.~Fribergh and A.~Hammond.
\newblock Phase transition for the speed of the biased random walk on the
  supercritical percolation cluster.
\newblock {\em Commun.Pur.Appl.Math.}, 67:173--245, 2014.

\bibitem{friberghtrap}
A.~Fribergh and D.~Kious.
\newblock Local trapping for elliptic random walks in random environments in
  $\mathbb{Z}^d$.
\newblock {\em Probab.Theory.Relat.Fields}, 165(3):795--834, 2016.

\bibitem{FriberghSubbal}
A.~Fribergh and D.~Kious.
\newblock Scaling limits for sub-ballistic biased random walks in random
  conductances.
\newblock {\em arxiv:1601.00889}, 2016.

\bibitem{Kesten}
H.~Kesten, M.~Kozlov, and F.~Spitzer.
\newblock A limit law for random walk in a random environment.
\newblock {\em Compositio Math}, 30:145--168, 1975.

\bibitem{LowDis2}
C.~Laurent, A.~F. Ram\'irez, C.~Sabot, and S.~Saglietti.
\newblock Velocity estimates for symmetric random walks at low ballistic
  disorder.
\newblock {\em Ann. Probab}, 2017.

\bibitem{Pemantle}
R.~Pemantle.
\newblock Phase transition in reinforced random walk and rwre on trees.
\newblock {\em Ann.Probab}, 16(3):1229--1241, 1988.

\bibitem{RamiEquiv}
A.~Ramirez and E.~Guerra.
\newblock A proof of sznitman's conjecture about ballistic rwre.
\newblock {\em Preprint:1809.02011v3}, 2018.

\bibitem{SeppaTCL}
F.~Rassoul-Agha and T.~Seppalainen.
\newblock Almost sure functional central limit theorem for ballistic random
  walk in random environment.
\newblock {\em Ann. Inst. Poincar\'e Probab. Statist.}, 45(2):373--420, 2009.

\bibitem{LowDis1}
C.~Sabot.
\newblock Ballistic random walks in random environment at low disorder.
\newblock {\em Ann. Probab}, 32(4):2996--3023, 2004.

\bibitem{transdim}
C.~Sabot.
\newblock Random walks in random dirichlet environment are transient in
  dimension $d \geq 3$.
\newblock {\em Probab.Theory Relat.Fields}, 151(1-2):297--317, 2011.

\bibitem{DirEnvirPOVPart}
C.~Sabot.
\newblock Random dirichlet environment viewed from the particle in dimension $d
  \geq 3$.
\newblock {\em Ann.Probab}, 41(2):722--743, 2013.

\bibitem{DirOverview}
C.~Sabot and L.~Tournier.
\newblock Random walks in dirichlet environment: an overview.
\newblock {\em arxiv:1601.08219}, 2016.

\bibitem{Sinai}
Y.~Sina\"i.
\newblock The limit behavior of a one-dimensional random walk in a random
  environment (russian).
\newblock {\em Teor.Veroyatnost.i Primenen}, 27(2):1851--1869, 1999.

\bibitem{Skorohod}
A.~Skorohod.
\newblock Limit theorems for stocastic processes.
\newblock {\em Teor.Veroyatnost. i Primenen}, 1:289--319, 1956.

\bibitem{Solomon}
F.~Solomon.
\newblock Random walks in a random environment.
\newblock {\em Ann.Probab}, 3:1--31, 1975.

\bibitem{SznitmanEffective}
A.~Sznitman.
\newblock An effective criterion for ballistic behaviour of random walks in
  random environment.
\newblock {\em Probab. Theory Relat. Fields}, 122:509--544, 2002.

\bibitem{SznitmanCLT}
A.-S. Sznitman.
\newblock Slowdown estimates and central limit theorem for random walks in
  random environment.
\newblock {\em J.Eur.Math.Soc.}, 2(2):93--143, 2000.

\bibitem{SznitmanT}
A.-S. Sznitman.
\newblock On a class of transient random walks in random environment.
\newblock {\em Ann.Probab}, 29(2):724--765, 2001.

\bibitem{SznitZeitCLTBrownian}
A.-S. Sznitman and O.~Zeitouni.
\newblock On the diffusive behaviour of isotropic diffusions in a random
  environment.
\newblock {\em C.R.Acad.Sci.Paris}, 339(6):429--434, 2004.

\bibitem{SznitmanZerner}
A.-S. Sznitman and M.~Zerner.
\newblock A law of large numbers for random walks in random environment.
\newblock {\em Ann.Probab}, 27(4):1851--1869, 1999.

\bibitem{IntegExitTime}
L.~Tournier.
\newblock Integrability of exit times and ballisticity for random walks in
  dirichlet environment.
\newblock {\em Electron.J.Probab}, 14(16):431--451, 2009.

\bibitem{transdir2}
L.~Tournier.
\newblock Asymptotic direction of random walks in dirichlet environment.
\newblock {\em Ann.Inst.Henri.Poincare}, 51(2):716--726, 2015.

\bibitem{Whitt}
W.~Whitt.
\newblock Stochastic-process limits.
\newblock {\em Springer Series in Operations Research.
  Springer-Verlag,New-York}, 2002.

\end{thebibliography}

\end{document}